\documentclass[a4paper]{amsart}
\usepackage{a4wide}
\usepackage[11pt]{moresize}

\usepackage[hyphens]{url} 
\usepackage[dvips]{graphicx} 
\usepackage[new]{old-arrows}

\usepackage{amsthm, amsmath, amssymb, amsfonts} 
\usepackage{mathtools, mathabx, epsfig, mathrsfs}
\usepackage{physics}
\usepackage{import}

\usepackage{xcolor}
\usepackage{xspace}
\usepackage{stackrel}
\usepackage{makecell}
\usepackage{enumerate}
\usepackage{enumitem}
\usepackage{footnote}
\usepackage[textsize=tiny]{todonotes}

\usepackage{hyperref}
\hypersetup{
   colorlinks,
    linkcolor={red!70!black},
    citecolor={blue!50!black},
    urlcolor={blue!80!black}
}
\usepackage{url}

\usepackage{tikz,tqft}
\usetikzlibrary{matrix,arrows,cd,decorations.markings,calc,patterns,angles,quotes,backgrounds}
\tikzset{>=latex}

\theoremstyle{plain}
 \newtheorem{thm}{Theorem}[section]
 \newtheorem{thmx}{Theorem}
 
 \newtheorem{prop}[thm]{Proposition}
 \newtheorem{lem}[thm]{Lemma}
 \newtheorem{cor}[thm]{Corollary}
\theoremstyle{definition}
 
 \newtheorem{dfn}[thm]{Definition}
 \newtheorem{conv}[thm]{Convention}
 
\theoremstyle{remark}
 \newtheorem{rem}[thm]{Remark}
 \numberwithin{equation}{section}
 
\def\acts{\mathrel{\reflectbox{$\righttoleftarrow$}}}

\newcommand{\R}{\mathbb{R}}
\newcommand{\C}{\mathbb{C}} 	
\newcommand{\CP}{\mathbb{CP}}
\newcommand{\Z}{\mathbb{Z}}
\newcommand{\T}{\mathbb{T}}
  
\newcommand{\Hyp}{\mathbb{H}}

\renewcommand{\restriction}{\mathord{\upharpoonright}}

\DeclareMathOperator{\parab}{par}
\DeclareMathOperator{\hyp}{hyp}  
\DeclareMathOperator{\rot}{rot}
\DeclareMathOperator{\vol}{vol}

\DeclareMathOperator{\Ker}{Ker}

\DeclareMathOperator{\Rep}{Rep}

\DeclareMathOperator{\Ad}{Ad}

\DeclareMathOperator{\Out}{Out}
\DeclareMathOperator{\Aut}{Aut}
\DeclareMathOperator{\Inn}{Inn}
\DeclareMathOperator{\fix}{fix}

\DeclareMathOperator{\SU}{SU}
\DeclareMathOperator{\Sp}{Sp}

\DeclareMathOperator{\psl}{PSL(2,\R)}
\DeclareMathOperator{\Mod}{Mod}
\DeclareMathOperator{\ChTri}{ChTri}
\newcommand{\ChainTriang}{\ChTri_\alpha}
\newcommand{\g}{\mathfrak{g}}
\newcommand{\G}{\mathcal{G}}
\newcommand{\Cgot}{\mathfrak{C}}
\newcommand{\FS}{\mathcal{FS}}

\newcommand{\charvar}{\Rep (\Sigma_n,G)}
\newcommand{\relcharvar}{\Rep_\alpha (\Sigma_n,G)}
\newcommand{\relcharvarthree}{\Rep_\alpha (\Sigma_3,G)}
\newcommand{\dtrelcharvar}{\Rep^{\textnormal{\tiny{DT}}}_\alpha(\Sigma_n,G)\xspace}
\newcommand{\regdtrelcharvar}{\mathring{\Rep}^{\textnormal{\tiny{DT}}}_\alpha(\Sigma_n,G)\xspace}

\newcommand{\E}{\mathcal{E}}

\renewcommand{\setminus}{\smallsetminus}

\title[Action-angle coordinates in genus zero]{Action-angle coordinates for surface group representations in genus zero}

\author[Arnaud Maret]{Arnaud Maret} 

\address{
Mathematisches Institut \\ 
Ruprecht-Karls-Universit\"at Heidelberg   \\ 
Germany}
\email{amaret@mathi.uni-heidelberg.de}

\begin{document}


\begin{abstract}
We study a compact family of totally elliptic representations of the fundamental group of a punctured sphere into $\psl$ discovered by Deroin and Tholozan and named after them. We describe a polygonal model that parametrizes the relative character variety of Deroin--Tholozan representations in terms of chains of triangles in the hyperbolic plane. We extract action-angle coordinates from our polygonal model as geometric quantities associated to chains of triangles. The coordinates give an explicit isomorphism between the space of representations and the complex projective space. We prove that they are almost global Darboux coordinates for the Goldman symplectic form.
\end{abstract}

\maketitle

\section{Introduction} 
A character variety is, broadly speaking, a symplectic manifold associated to a closed oriented surface $\Sigma$ and a quadrable\footnote{We call a Lie group \emph{quadrable} if its Lie algebra can be equipped with a non-degenerate, symmetric, $\Ad$-invariant bilinear form, see \cite[§1.1]{Ma22} for more details.} Lie group $G$. It is defined as the space of conjugacy classes of representations of the fundamental group of $\Sigma$ into $G$. A relative character variety is the counterpart for punctured surfaces. Only representations that map the peripheral curves around each puncture inside a prescribed conjugacy class of $G$ are considered, see Section ~\ref{sec:preliminaries} for more details. To study the topology of a character variety, one approach consists in finding a specific set of coordinates that parametrize (most of) it. A foundational example is the parametrization of Teichmüller space via Fenchel-Nielsen coordinates. 

In this paper we study the geometry of compact components of relative character varieties of representations of the fundamental group of an $n$ times punctured sphere into $\psl$, discovered by Deroin-Tholozan in 
\cite{DeTh19}. We refer to these compact components as \emph{Deroin-Tholozan relative character varieties} and to the representations themselves as \emph{Deroin-Tholozan representations}. There exists such a component for every choice of $\alpha=(\alpha_1,\ldots,\alpha_n)\in (0,2\pi)^n$ with the property that $\alpha_1+\ldots+\alpha_n>2\pi(n-1)$, see Section ~\ref{sec:preliminaries} for more details. They have the fundamental property of mapping any simple closed curve to an elliptic element of $\psl$, see Proposition ~\ref{prop:totally-elliptic}. We say they are \emph{totally elliptic}.

\subsection{The results}\label{sec:results} Let $\Sigma_n$ denote a connected and oriented surface of genus zero with $n\geq 3$ labelled punctures. We fix a pants decomposition $\mathcal P$ of $\Sigma_n$ and a Deroin-Tholozan representation $\phi\colon\pi_1(\Sigma_n)\to\psl$. To each of the $n-2$ pairs of pants $P_0,\ldots,P_{n-3}$ in $\mathcal P$, we associate a geodesic triangle $\Delta_i$ in the upper half-plane whose vertices are the unique fixed points of the images of the three boundary curves of $P_i$ (we use that $\phi$ is totally elliptic). Two vertices are joined by a geodesic segment if they correspond to two boundary curves of the same pair of pants. This produces a chain of $n-2$ geodesic triangles $\Delta_0,\ldots,\Delta_{n-3}$ in the upper half-plane, see Figure ~\ref{fig:introduction}.

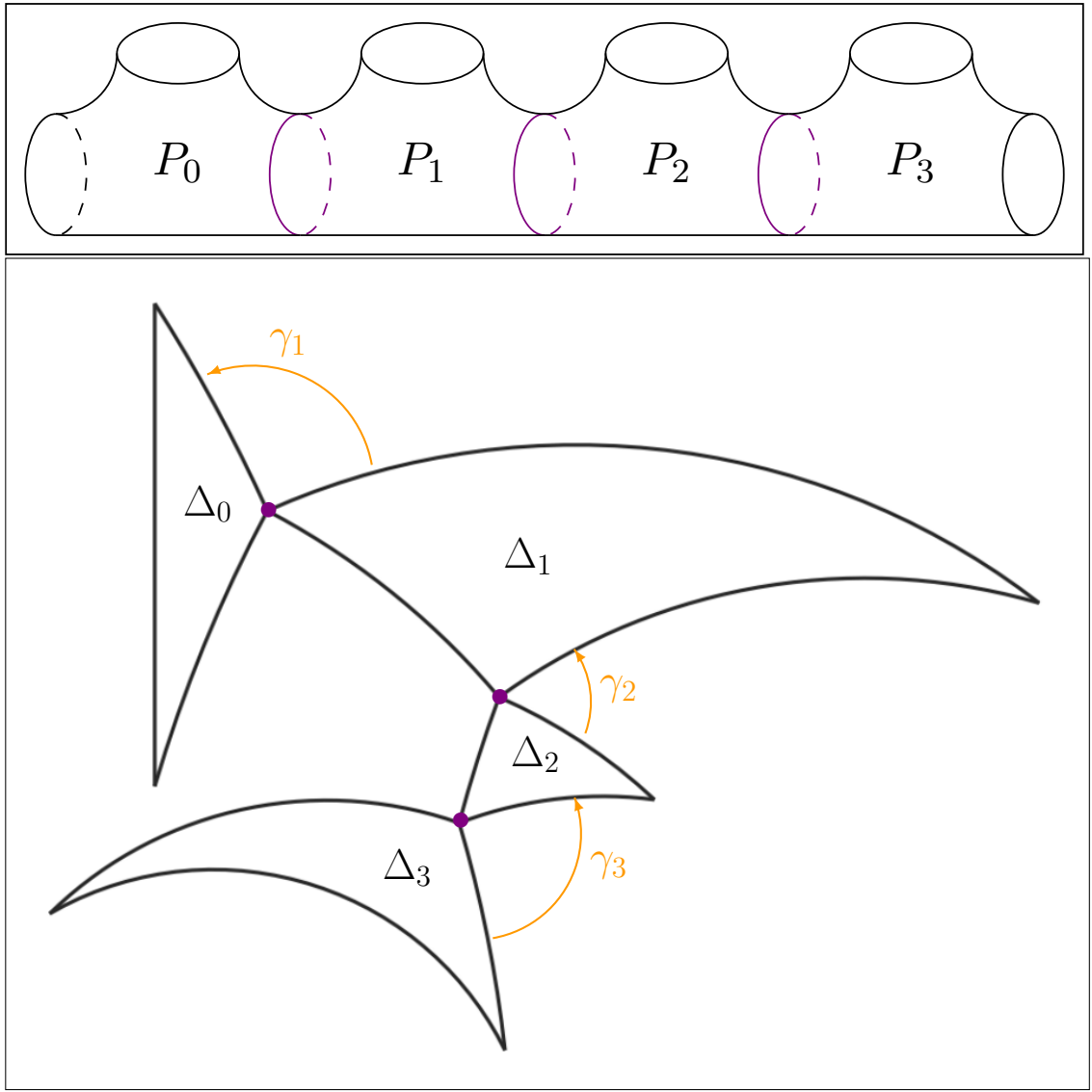
\begin{figure}[h!]
\begin{center}
\resizebox{.8\textwidth}{!}{%
\begin{tikzpicture}[framed]
  \draw (0,-.5) arc(-90:-270: .25 and .5);
  \draw[dashed] (0,.5) arc(90:-90: .25 and .5);
  \draw[violet] (2,-.5) arc(-90:-270: .25 and .5);
  \draw[violet, dashed] (2,.5) arc(90:-90: .25 and .5);
  \draw[violet] (4,-.5) arc(-90:-270: .25 and .5);
  \draw[violet, dashed] (4,.5) arc(90:-90: .25 and .5);
  \draw[violet] (6,-.5) arc(-90:-270: .25 and .5);
  \draw[violet, dashed] (6,.5) arc(90:-90: .25 and .5);
  \draw  (8,0) ellipse (.25 and .5);
  
  \draw (1,1) ellipse (.5 and .25);
  \draw (3,1) ellipse (.5 and .25);
  \draw (5,1) ellipse (.5 and .25);
  \draw (7,1) ellipse (.5 and .25);
   
  \draw (0,.5) to[out=0,in=-90] (.5,1);
  \draw (1.5,1) to[out=-90,in=180] (2,.5);
  \draw (0,-.5) to[out=0,in=180] (2,-.5);
  
  \draw (2,.5) to[out=0,in=-90] (2.5,1);
  \draw (3.5,1) to[out=-90,in=180] (4,.5);
  \draw (2,-.5) to[out=0,in=180] (4,-.5);
  
  \draw (4,.5) to[out=0,in=-90] (4.5,1);
  \draw (5.5,1) to[out=-90,in=180] (6,.5);
  \draw (4,-.5) to[out=0,in=180] (6,-.5);
  
  \draw (6,.5) to[out=0,in=-90] (6.5,1);
  \draw (7.5,1) to[out=-90,in=180] (8,.5);
  \draw (6,-.5) to[out=0,in=180] (8,-.5);
  
  \draw (1,.1) node{\large $P_0$};
  \draw (3,.1) node{\large $P_1$};
  \draw (5,.1) node{\large $P_2$};
  \draw (7,.1) node{\large $P_3$};
\end{tikzpicture}
}
\resizebox{.8\textwidth}{!}{%
\begin{tikzpicture}[framed,font=\sffamily,decoration={
    markings,
    mark=at position 1 with {\arrow{>}}}]]
    
\node[anchor=south west,inner sep=0] at (0,0) {\includegraphics[width=\textwidth]{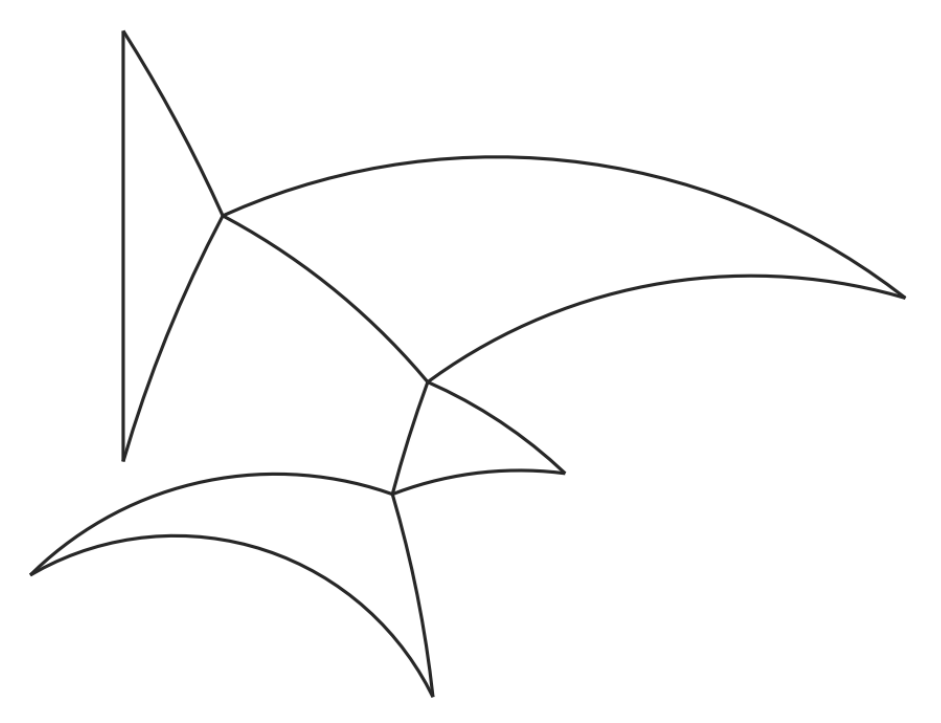}};

\draw[violet] (3.55,8.05) node{\huge $\bullet$};

\draw[thick, postaction={decorate}, red!40!yellow] (5,8.7) arc (10:112:1.7) node[near end, above right]{\huge $\gamma_1$};

\draw[violet] (6.8,5.4) node{\huge $\bullet$};

\draw[thick, postaction={decorate}, red!40!yellow] (8,4.9) arc (-20:35:1.3) node[near start, above right]{\huge $\gamma_2$};

\draw[violet] (6.25,3.65) node{\huge $\bullet$};

\draw[thick, postaction={decorate}, red!40!yellow] (6.7,2) arc (-80:20:1.5) node[near end, below right]{\huge $\gamma_3$};

\draw (2.7,8.15) node{\huge $\Delta_0$};
\draw (7.2,7.4) node{\huge $\Delta_1$};
\draw (7.3,4.6) node{\huge $\Delta_2$};
\draw (5.5,3) node{\huge $\Delta_3$};
\end{tikzpicture}
}
\caption{On top: a pants decomposition of a sphere with six punctures into four pairs of pants. This illustration is modelled on 
\cite[Fig.\ ~2]{DeTh19}. On bottom: a corresponding chain of geodesic triangles in the upper half-plane. The angles between consecutive triangles in the chain are denoted by $\gamma_i$.}\label{fig:introduction}
\end{center}
\end{figure}

We explain in Section ~\ref{sec:polygonal-model} that rotating around the common vertices of two consecutive triangles in the chain defines a maximal Hamiltonian torus action on each Deroin-Tholozan relative character variety. We denote by $a_i$ the double of the area of the triangle $\Delta_i$ and by $\gamma_i$ the angle between the triangles $\Delta_{i-1}$ and $\Delta_i$ as on Figure ~\ref{fig:introduction} (the parameters $\gamma_i$ are defined formally in Section ~\ref{sec:complex-proj-coordinates}).

\begin{thmx}\label{thm:action-angle-coordinates}
If we set $\sigma_i\coloneqq\gamma_1+\ldots+\gamma_i$, then
\[
\big\{a_1,\ldots,a_{n-3},\sigma_1,\ldots,\sigma_{n-3}\big\}
\]
are action-angle coordinates for each Deroin-Tholozan relative character variety.
\end{thmx}

Theorem ~\ref{thm:action-angle-coordinates} says that the dynamical system induced by rotating the triangles around their common vertices is integrable with canonical coordinates $\big\{a_1,\ldots,a_{n-3},\sigma_1,\ldots,\sigma_{n-3}\big\}$. The notion of \emph{action-angle coordinates} refers to the canonical coordinates of an integrable system in the sense of the Arnold-Liouville Theorem, see e.g.\ 
\cite{Can01}. In fact, the dynamics corresponds to the maximal Hamiltonian torus action on each Deroin-Tholozan relative character variety defined in \cite{DeTh19} by considering the twist flows \emph{à la Goldman} along the separating curves defining the pants decomposition of $\Sigma_n$. This action equips each Deroin-Tholozan relative character variety with the structure of a \emph{symplectic toric manifold}. We deduce Theorem ~\ref{thm:action-angle-coordinates} from

\begin{thmx}[see Theorem ~\ref{thm:symplectomorphism}]\label{thm:symplectomorphism-intro}
The map from each Deroin-Tholozan relative character variety to $\CP^{n-3}$ defined in homogeneous coordinates by 
\[
\left[ \sqrt{a_0}:\sqrt{a_1}e^{i\sigma_1}:\ldots:\sqrt{a_{n-3}}e^{i\sigma_{n-3}}\right]
\]
is an isomorphism of symplectic toric manifolds.
\end{thmx}

Both spaces involved in the statement of Theorem ~\ref{thm:symplectomorphism-intro} are equipped with a natural symplectic structure: the Goldman symplectic form $\omega_\G$ for the character variety and the Fubini-Study form for the complex projective space. It is known since 
\cite{DeTh19} that each Deroin-Tholozan character variety and $\CP^{n-3}$ are abstractly isomorphic. It is a consequence of Delzant's classification of symplectic toric manifolds. Theorem ~\ref{thm:symplectomorphism-intro} goes one step further and provides an explicit isomorphism.

The map of Theorem ~\ref{thm:symplectomorphism-intro} from a Deroin-Tholozan relative character variety to $\CP^{n-3}$ factors through the sphere $S^{2n-5}\subset \C^{n-2}$ of radius $\alpha_1+\ldots+\alpha_n-2\pi(n-1)$. Indeed, we will prove in Section \ref{sec:complex-proj-coordinates} that $a_0+a_1+\ldots+a_n$ is constant for each Deroin-Tholozan relative character variety and equals $\alpha_1+\ldots+\alpha_n-2\pi(n-1)$.

The main difficulty in the proof of Theorem ~\ref{thm:symplectomorphism-intro} relies in checking that the map is differentiable. This requires a careful analysis of all the parameters involved. The primary source of trouble is the erratic behaviour of the parameters $\sigma_i$ when a triangle in a chain degenerates to a single point and the presence of square roots on the parameters $a_i$. An immediate consequence of Theorem ~\ref{thm:symplectomorphism-intro} is

\begin{thmx}[see Corollary ~\ref{cor:analogue-Wolpert-formula}]\label{thm:wolpert-like-formula}
On an open and dense subset of each Deroin-Tholozan relative character variety, it holds that
\[
\omega_\G=\frac{1}{2}\sum_{i=1}^{n-3}da_i\wedge d\sigma_i,
\]
where $\omega_\G$ is the Goldman symplectic form. In particular, the 2-form $\sum_{i=1}^{n-3}da_i\wedge d\sigma_i$ is independent of the pants decomposition used to define the coordinates $\big\{a_1,\ldots,a_{n-3},\sigma_1,\ldots,\sigma_{n-3}\big\}$.
\end{thmx}

The open and dense subset of Theorem \ref{thm:wolpert-like-formula} consists of the conjugacy classes of representations for which no triangles in the corresponding chain is degenerate to a point. Equivalently, it is the locus of the Deroin-Tholozan relative character variety on which $a_i>0$ for every $i=0,1,\ldots,n$.

Theorem ~\ref{thm:wolpert-like-formula} is the analogue of a famous result of Wolpert known as \emph{Wolpert's magic formula} in the context of Teichmüller space. We briefly explain the analogy. The Teichmüller space of a closed hyperbolic surface of genus $g$ can be identified with $(0,\infty)^{3g-3}\times \R^{3g-3}$ using Fenchel-Nielsen coordinates, see e.g.\ 
\cite{FaMa12}. Fenchel-Nielsen coordinates consist of length parameters $\ell_1,\ldots,\ell_{3g-3}$ and twist parameters $\theta_1,\ldots,\theta_{3g-3}$. They are defined for any pants decomposition of the surface. Wolpert proved in 
\cite{Wol83} that the length and twist parameters are dual to each other and that the 2-form 
\begin{equation}\label{eq:2-form-Fenchel-Nielsen-coordinates}
\frac{1}{2}\sum_{i=1}^{3g-3}d\ell_i\wedge d\theta_i
\end{equation}
is independent of the choice of the pants decomposition. He did so by proving that the 2-form \eqref{eq:2-form-Fenchel-Nielsen-coordinates} is equal to the Weil-Petersson form on Teichmüller space. This relation is nowadays known as Wolpert's magic formula. Goldman proved in 
\cite{Gol84} that the Weil-Petersson form is a multiple of the Goldman symplectic form if one sees Teichmüller space as a component of the character variety of representations of the fundamental group of the surface into $\psl$.

Another analogy can be drawn between Theorem ~\ref{thm:action-angle-coordinates} and the work of Kapovich-Millson. It was proved in 
\cite{KaMi96} that the moduli space of polygons in $\R^3$ with fixed side lengths  is isomorphic to the character variety of representations of the fundamental group of a punctured sphere into the isometries of $\R^3$. The authors further proved that the lengths $\ell_i$ of the diagonals emanating from the same vertex and the angle $\theta_i$ between the triangles they define are action-angle coordinates. In the context of Theorem ~\ref{thm:action-angle-coordinates}, the definition of the coordinates $\big\{a_1,\ldots,a_{n-3},\sigma_1,\ldots,\sigma_{n-3}\big\}$ relies on the modelling of Deroin-Tholozan representations in terms of chains of geodesic triangles in the upper half-plane. In the spirit of \cite{KaMi96}, we introduce a moduli space of \emph{chains of geodesic triangles} in the upper half-plane and show that it is in one-to-one correspondence with the relative character variety of Deroin-Tholozan representations. We refer to it as the \emph{polygonal model} for Deroin-Tholozan representations. 


The question of the geometrization of Deroin-Tholozan representations has been treated in \cite{DeTh19}, where it was proven that they arise as monodromies of hyperbolic metrics on $\Sigma_n$ with conical singularities. The metrics under consideration have two types of singularities: $n$ cone angles $2\pi-\alpha_i$ at the punctures, called \emph{fractional singularities}, and $n-3$ further cone angles $4\pi$ that can coalesce with fractional singularities and each other. The resulting moduli space is homeomorphic to the product of the Teichmüller space of $\Sigma_n$ with the Deroin-Tholozan relative character variety, see \cite[§4]{DeTh19}. We denote it $\overline{\text{Hyp}}_{\alpha}$. In the companion paper \cite{FeMa23}, together with Aaron Fenyes, we provide a piece-by-piece construction of such hyperbolic cone metrics based on chains of triangles and we prove that any Deroin-Tholozan representation can be obtained as the holonomy of a hyperbolic cone metric built this way.
\subsection{Organization of the paper}

Section ~\ref{sec:preliminaries} fixes notation and recalls some fundamental notions about character varieties, including the notion of volume of a representation. It introduces the Deroin-Tholozan relative character variety
\[
\dtrelcharvar
\]
and also provides a brief recap of the main results of \cite{DeTh19}. These include total ellipticity and the toric structure.

In Section ~\ref{sec:polygonal-model} we introduce the polygonal model for Deroin-Tholozan representations as a certain moduli space denoted by
\[
\ChainTriang
\]
and consisting of chains of geodesic triangles in the upper half-plane. Sufficient conditions for a chain of triangles to lie inside $\ChainTriang$ are described in Lemma ~\ref{lem:description-chain-triangles-alpha}, see also Figure ~\ref{fig:admissible_chain}. We show that there is a one-to-one correspondence 
\[
\mathfrak P\colon \dtrelcharvar\to \ChainTriang.
\]
We also explain how the parameters used to define Deroin-Tholozan representations and quantities associated to them, such as their volume, can be measured from the associated chain of triangles. We explain how to visualize the integrable system dynamics in that picture.

In Section ~\ref{sec:complex-proj-coordinates} we formally define the coordinates $\big\{a_1,\ldots,a_{n-3},\sigma_1,\ldots,\sigma_{n-3}\big\}$. The coordinates are used to define the map
\[
\Cgot\colon \dtrelcharvar\to\CP^{n-3}
\]
of Theorem ~\ref{thm:symplectomorphism-intro} which is carefully stated at the start of Section ~\ref{sec:complex-proj-coordinates} as Theorem ~\ref{thm:symplectomorphism}. At the end of the section we explain how Theorem ~\ref{thm:wolpert-like-formula} (see Corollary ~\ref{cor:analogue-Wolpert-formula}) follows from Theorem ~\ref{thm:symplectomorphism-intro}.

Section ~\ref{sec:proof-of-theorem} is dedicated to the proof of Theorem ~\ref{thm:symplectomorphism-intro}. The proof is divided into successive steps where we first address the questions of bijectivity, continuity and then differentiability. 

\begin{figure}[h]
\begin{center}
\begin{tikzpicture}[scale=.8, framed,font=\sffamily,decoration={
    markings,
    mark=at position 1 with {\arrow{>}}}]]
    
\node[anchor=south west,inner sep=0] at (.8,-.5) {\includegraphics[scale=.12]{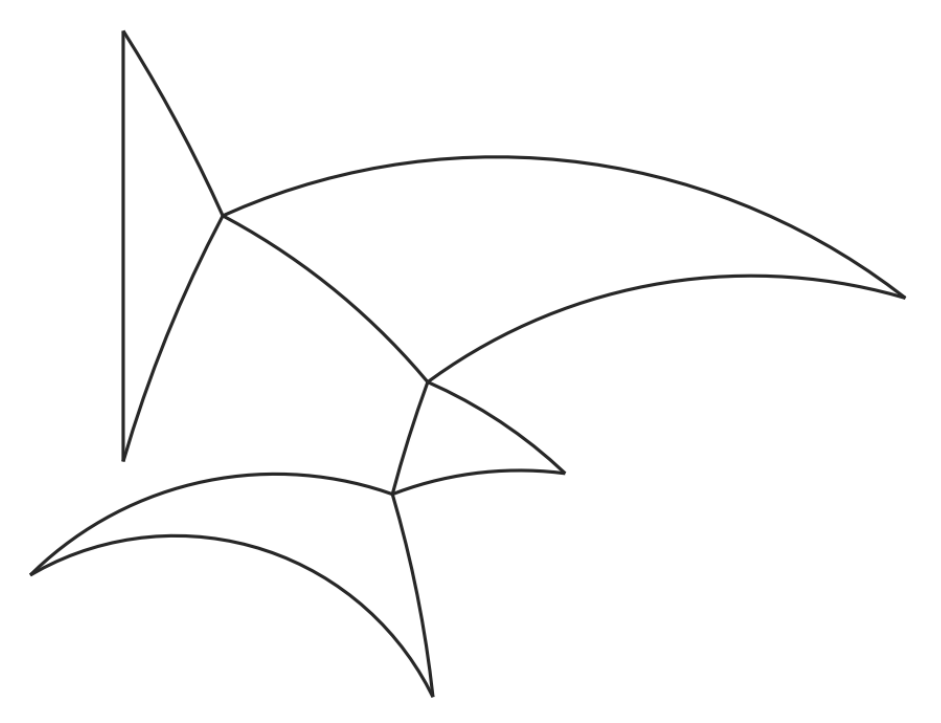}};
		 
\draw (2.4,2.5) node{\huge $\ChainTriang$};

\draw (7,5.5) node{\huge $\dtrelcharvar$};

\draw (-2.5,5.5) node{\huge $\CP^{n-3}$};

\draw (2.4,-3) node{\huge $\overline{\textnormal{Hyp}}_{\alpha}$};

\draw[->] (7,5) to[bend left =40] node[below right]{\large $\mathfrak P$} (3.5, 2.5);

\draw[->] (5,5.5) to node[below]{\large $\Cgot$} (-1.5,5.5);

\draw[->, dashed] (3.3,-3) to[bend right =30] (7.2,5);
\draw[->, dashed] (1,2.5) to[bend left =30] (-2,5);
\draw[->, dashed] (1.3,-3) to[bend left =30] (-2.7,5);

\end{tikzpicture}
\caption{The various aspects of Deroin-Tholozan representations are drawn as a diagram.  The space $\overline{\textnormal{Hyp}}_{\alpha}$ refers to the space of hyperbolic metrics with conical singularities introduced at the end of Section~\ref{sec:results}. It projects to $\dtrelcharvar$ as mentioned earlier. The projection corresponds to the outer dashed arrows.}
\end{center}
\end{figure}

\subsection{Acknowledgements} I would like to thank first and foremost my doctoral advisers Peter Albers and Anna Wienhard for their constant support during my time as a PhD student. This work benefited greatly from informal discussions with many members of the Heidelberg math community. Special thanks must be addressed to Andy Sanders and Beatrice Pozzetti. I am also grateful to Fernando Camacho Cadena, Aaron Fenyes, Xenia Flamm, and Nicolas Tholozan for their precious comments. I extend special appreciation to the anonymous referee for suggesting many insightful corrections.

This work was supported by the Deutsche Forschungsgemeinschaft under Germany’s Excellence Strategy EXC2181/1 -- 390900948 (the Heidelberg STRUCTURES Excellence Cluster), the Collaborative Research Center SFB/TRR 191 -- 281071066 (Symplectic Structures in Geometry, Algebra and Dynamics), and the Research Training Group RTG 2229 -- 281869850 (Asymptotic Invariants and Limits of Groups and Spaces).

\section{Preliminaries}\label{sec:preliminaries}
Let $G\coloneqq\psl$ be the Lie group of orientation-preserving isometries of the upper half-plane $\Hyp$. We denote by $\g$ its Lie algebra consisting of traceless two-by-two real matrices. Let $n\geq 3$ be an integer. We denote by $\Sigma_n$ an oriented and connected surface of genus zero with $n$ labelled punctures. Let $\pi_1(\Sigma_n)$ be its fundamental group. It can be presented as
\begin{equation}\label{eq:fundamental-group-SIgma_n}
\pi_1(\Sigma_n)=\langle c_1,\ldots,c_n:\prod_{i=1}^n c_i=1\rangle.
\end{equation}
We think of the generator $c_i$ of $\pi_1(\Sigma_n)$ as the homotopy class of based loops enclosing the $i$th puncture of $\Sigma_n$. The \emph{character variety} associated to the pair $(\Sigma_n,G)$ is the topological quotient of the space of group homomorphisms $\pi_1(\Sigma_n)\to G$ by the conjugation action of $G$. It is denoted $\charvar$.

We are interested in representations $\pi_1(\Sigma_n)\to G$ that map each generator $c_i$ of $\pi_1(\Sigma_n)$ to an \emph{elliptic} element inside $G$. Elliptic transformations in $G$ are, by definition, those which act on $\Hyp$ with a unique fixed point. Equivalently, $A\in G$ is elliptic if and only if the absolute value of its trace is smaller than two. The subset of elliptic elements of $G$ is denoted by
\[
\E\coloneqq\left\{A\in G : A \text{ is elliptic}\right\}.
\]
It is analytically diffeomorphic to the open ball $\Hyp \times (0,2\pi)$. In fact, an analytic diffeomorphism $(\fix,\vartheta)\colon\E\to \Hyp \times (0,2\pi)$ can be defined as follows. The first component $\fix\colon\E\to\Hyp$ maps $A\in\E$ to its unique fixed point $\fix(A)$ in $\Hyp$. The second component $\vartheta\colon \E\to (0,2\pi)$ records the \textit{angle of rotation} $\vartheta (A)$ of $A\in \E$, see Appendix ~\ref{appendix-PSL2R} for a precise definition. The angle of rotation is invariant under the $G$-action on $\E$ by conjugation and completely distinguishes the conjugacy classes of $\E$. We refer the reader to Appendix ~\ref{appendix-PSL2R} for more considerations on $\E$ and various assorted formulae.

Given a tuple of angles $\alpha=(\alpha_1,\ldots,\alpha_n)\in (0,2\pi)^n$, we define the \emph{relative character variety} of representations of $\pi_1(\Sigma_n)$ into $G$ associated to $\alpha$ to be the subspace of the character variety given by
\[
\relcharvar\coloneqq\big\{[\phi]\in\charvar:\vartheta(\phi(c_i))=\alpha_i,\quad \forall i=1,\ldots,n\big\}.
\]
The relative character variety $\relcharvar$ is an analytic manifold of dimension $2(n-3)$ for all possible choices of $\alpha$. Each of these relative character varieties has a natural symplectic structure for any choice of a non-degenerate, symmetric, $\Ad$-invariant bilinear form on the Lie algebra $\g$ of $G$, see 
\cite{Gol84} and 
\cite{GHJW97} for the original results and 
\cite{Mon16} for a recent account. The symplectic structure is named after Goldman who originally constructed it in the case of closed surfaces in 
\cite{Gol84}. 

\begin{conv}
We equip the relative character varieties $\relcharvar$ with the Goldman symplectic form $\omega_\G$ built from the trace form:
\begin{align*}
\trace\colon &\g\times \g\longrightarrow\R\nonumber \\
&(A, B)\mapsto \trace(AB).
\end{align*}
\end{conv}

\subsection{Volume of a representation} There is an important tool to study the topology of relative character varieties called \emph{volume of a representation} (or \emph{Toledo number}). We briefly recall its definition below for the case of representations of the punctured sphere into $G$. The reader is referred to 
Burger-Iozzi-Wienhard \cite{BIW10} for more details and for the general construction.

Given three points $z_1$, $z_2$ , $z_3$ in the upper half-plane, we denote by $\Delta (z_1,z_2,z_3)$ the oriented geodesic triangle with vertices $z_1$, $z_2$, $z_3$. Its signed area, computed with the standard volume form of $\Hyp$, is denoted by
\[
[\Delta(z_1,z_2,z_3)].
\]
Fix a base point $z\in \Hyp$. Consider the function
\begin{align}
c\colon &G\times G\to \R \label{eq:cocycle-c}\\
&(A_1,A_2)\to \big[\Delta\big(z,A_1z,A_1A_2z\big)\big]. \nonumber
\end{align}
A rapid computation (see for instance \cite[Lemma 5.1.2]{Ma22}) shows that $c$ satisfies the cocycle condition
\begin{equation}\label{eq:area-cocycle-formula}
c(A_2,A_3)-c(A_1A_2,A_3)+c(A_1,A_2A_3)-c(A_1,A_2)=0
\end{equation}
for every $A_1$, $A_2$, $A_3 \in G$. Therefore, $c$ defines a group cohomology class $\kappa\coloneqq[c]$ inside $H^2(G;\R)$. Moreover the function $c$ is bounded because the area of a hyperbolic triangle is always smaller than $\pi$. The cohomology class $\kappa$ is thus a class in \emph{bounded} group cohomology $\kappa\in H^2_b(G;\R)$. The class $\kappa$ is independent of the choice of the base point $z$ involved in the definition of $c$ (whereas $c$ does depend on the point $z$), see \cite[Lemma 5.1.3]{Ma22} for a proof. Given a representation $\phi\colon \pi_1(\Sigma_n)\to G$ we can pull back $\kappa$ to the class $\phi^\ast \kappa$ inside $H^2_b(\pi_1(\Sigma_n);\R)$. We denote by $\partial \pi_1(\Sigma_n)$ the collection $\{\langle c_1\rangle,\ldots,\langle c_n\rangle\}$ of rank 1 subgroups of $\pi_1(\Sigma_n)$. An important property of bounded cohomology of amenable groups says that the map
\[
j\colon H^2_b(\pi_1(\Sigma_n),\partial \pi_1(\Sigma_n);\R)\to H^2_b(\pi_1(\Sigma_n);\R)
\]
from the long exact sequence in cohomology of the pair $\big(\pi_1(\Sigma_n),\partial \pi_1(\Sigma_n)\big)$ is an isomorphism. Integrating along the fundamental class $[\pi_1(\Sigma_n),\partial \pi_1(\Sigma_n)]$ corresponding to the orientation of $\Sigma_n$ provides an isomorphism from $H^2_b(\pi_1(\Sigma_n),\partial \pi_1(\Sigma_n);\R)$ to $\R$.

\begin{dfn}\label{def:volume-representation}
The \emph{volume} of a representation $\phi\colon \pi_1(\Sigma_n)\to G$ is the real number defined by
\[
\vol(\phi)\coloneqq j^{-1}(\phi^\ast \kappa)\frown [\pi_1(\Sigma_n),\partial \pi_1(\Sigma_n)].
\]
\end{dfn}

The volume is a well-defined function of the character variety
\[
\vol\colon \charvar\to \R.
\]
To state the properties of the volume it is convenient to introduce the following function. The angle function $\vartheta\colon \E\to (0,2\pi)$ can be extended to an upper semi-continuous function $\overline\vartheta\colon G\to [0,2\pi]$ by
\[
\overline\vartheta(A)\coloneqq\left\{\begin{array}{ll}
\vartheta(A), &\text{ if $A$ is elliptic,}\\
0, &\text{ if $A$ is hyperbolic or positively parabolic,}\\
2\pi, &\text{ if $A$ is the identity or negatively parabolic.}
\end{array}\right.
\]
The notions of being positively and negatively parabolic refer to the two conjugacy classes of parabolic elements in $G$, see \eqref{eq:parabolic-PSL(2,R)}.

\begin{thm}[\cite{BIW10}]\label{thm:volume}
The volume function $\vol\colon \charvar\to \R$ has the following properties:
\begin{enumerate}
\item $\vol$ is locally constant on each relative character variety,
\item \emph{(Milnor-Wood inequality)} $\vol$ is bounded:
\[
\vert \vol\vert \leq 2\pi(n-2),
\]
\item \emph{(additivity)} if $\Sigma_n$ is separated by a simple closed curve into two surfaces $S_1$ and $S_2$, then, for every $[\phi]\in \charvar$,
\[
\vol([\phi])=\vol([\phi\restriction_{\pi_1(S_1)}])+\vol([\phi\restriction_{\pi_1(S_2)}]),
\]
\item for every $[\phi]\in \charvar$, there exists an integer $k([\phi])$ satisfying
\[
\vol([\phi])=2\pi k([\phi])-\sum_{i=1}^n\overline\vartheta(\phi(c_i)).
\]
\end{enumerate}  
\end{thm}

Deroin-Tholozan called the integer $k([\phi])$ the \emph{relative Euler class} of $[\phi]$. We stick to this terminology.

To conduct explicit computations of the volume of a representation, as it will be the case in the proof of Lemma ~\ref{lem:volume-case-n=3}, it is convenient to fix a resolution for group (co)homology. We choose to work with the bar resolution. The reader is referred to 
\cite[Chapter 7]{Nos17} or 
\cite{Loh10} for the definition of the bar complex and all the relevant formulae. The fundamental class $[\pi_1(\Sigma_n),\partial \pi_1(\Sigma_n)]$ can be expressed in the bar resolution as follows. First consider the 2-chain 
\begin{equation}\label{eq:fundamental-class}
e\coloneqq(c_1,c_2)+(c_1c_2,c_3)+\ldots+(c_1c_2\cdots c_{n-1},c_n)+(1,1)
\end{equation}
in the bar complex of the group $\pi_1(\Sigma_n)$. The 2-chain $(e,c_1,\ldots,c_n)$ in the relative bar complex of the pair $(\pi_1(\Sigma_n),\partial \pi_1(\Sigma_n))$ is closed and its homology class is the fundamental class $[\pi_1(\Sigma_n),\partial \pi_1(\Sigma_n)]$, see 
\cite{GHJW97} or \cite[Lemma 4.2.12]{Ma22} for explicit computations. 

\subsection{Remarkable connected components}
Deroin-Tholozan proved in 
\cite{DeTh19} that the relative Euler class of any $[\phi]\in\relcharvar$ is bounded above by $n-1$. Furthermore, they proved that there exists $[\phi]\in\relcharvar$ with $k([\phi])=n-1$ if and only if 
\begin{equation}\label{eq:angles-condition}
\alpha_1+\ldots+\alpha_n> 2\pi(n-1).
\end{equation}
These representations were originally called \emph{supra-maximal} because they maximize the relative Euler class. However, these representations do not have maximal volume and are thus not \emph{maximal} in the sense of 
\cite{BIW10}. They even tend to minimize the volume in absolute value. Indeed, if $[\phi]\in\relcharvar$ satisfies $k([\phi])=n-1$, then
\[
\vol([\phi])=2\pi(n-1)-\alpha_1-\ldots-\alpha_n \in (-2\pi,0).
\]
The range of the volume over $\charvar$, according to the Milnor-Wood inequality stated in Theorem ~\ref{thm:volume}, is $[-2\pi(n-2),2\pi(n-2)]$. To avoid any further confusion we prefer the terminology of \textit{Deroin-Tholozan representations} instead of that of supra-maximal representations.

\begin{dfn}\label{def:scaling-factor}
The real number
\[
\lambda\coloneqq \alpha_1+\ldots+\alpha_n -2\pi(n-1)<2\pi
\]
is called the \emph{scaling factor}. The condition \eqref{eq:angles-condition}, or equivalently the condition $\lambda >0$, is referred to as the \emph{angles condition} on $\alpha$.
\end{dfn}

Observe that $[\phi]\in\relcharvar$ satisfies $k([\phi])=n-1$ if and only if it satisfies $\vol([\phi])=-\lambda$.

\begin{conv}
From now on, and unless otherwise stated, a vector of angles $\alpha\in (0,2\pi)^n$ is assumed to satisfy the angles condition \eqref{eq:angles-condition}.
\end{conv}

\begin{dfn}
The subset of $\relcharvar$ consisting of those classes of representations $[\phi]$ with $\vol([\phi])=-\lambda$ is called the \emph{Deroin-Tholozan relative character variety} and is denoted by
\[
\dtrelcharvar.
\]
Any representation whose conjugacy class lies inside $\dtrelcharvar$ is called a \emph{Deroin-Tholozan representation}.
\end{dfn}

\begin{thm}[\cite{DeTh19}]\label{thm:deroin-tholozan-symplecto}
The Deroin-Tholozan relative character variety is a nonempty and compact connected component of the relative character variety. It is moreover symplectomorphic to the complex projective space of complex dimension $n-3$:
\[
\big(\dtrelcharvar,\omega_\G\big)\cong \big(\CP^{n-3},\lambda\cdot\omega_\FS\big),
\]
where $\omega_\FS$ is the Fubini-Study symplectic form on $\CP^{n-3}$ with volume $\pi^{n-3}/(n-3)!$.
\end{thm}

\begin{rem}
These compact connected components were already discovered by Benedetto-Goldman in the case $n=4$ 
\cite{BeGo99}.
\end{rem}

Deroin-Tholozan representations have an important property called total ellipticity. A representation is called \emph{totally elliptic} if it maps any simple closed curve to an elliptic element of $G$. Total ellipticity for Deroin-Tholozan representation was originally proved in 
\cite{DeTh19} for a particular collection of simple closed curves. The argument generalizes immediately to any simple closed curve, see 
\cite{Ma20}.
\begin{prop}[total ellipticity]\label{prop:totally-elliptic}
Let $a\in\pi_1(\Sigma_n)$ denote the homotopy class of a simple closed curve on $\Sigma_n$. Then $\phi(a)\in G$ is elliptic for any $[\phi]\in\dtrelcharvar$.
\end{prop}

\begin{rem}
Tholozan-Toulisse found in \cite{ThTo21} analogous compact components in character varieties of representations of $\pi_1(\Sigma_n)$ into general Hermitian Lie groups such as $\SU(p,q)$ and $\Sp(2n,\R)$. The representations in these components admit very similar properties to Deroin-Tholozan representations. For instance, they are also totally elliptic, in the sense that the complex eigenvalues of the images of simple closed curves have modulus 1.
\end{rem}

The Deroin-Tholozan relative character variety admits a natural maximal and effective Hamiltonian torus action. Recall that a torus action on a symplectic manifold is called \emph{maximal} if the dimension of the torus is half the dimension of the manifold and it is called \emph{effective} if only the identity element acts trivially. Here, it is constructed following the work of Goldman in 
\cite{Gol86} on invariant functions. Recall that the angle of rotation $\vartheta\colon \E\to (0,2\pi)$ is a function invariant under conjugation defined on the subspace $\E\subset G$ of elliptic elements. Therefore, by Proposition ~\ref{prop:totally-elliptic}, any simple closed curve $a$ on $\Sigma_n$ gives a Hamiltonian function
\begin{align*}
\vartheta_a\colon & \dtrelcharvar\longrightarrow (0,2\pi) \\
&[\phi]\mapsto \vartheta(\phi(a)).
\end{align*}
The associated Hamiltonian flow $\Phi_a$ has minimal period $\pi$, see 
\cite{DeTh19}. We refer to this flow as the \emph{twist flow} along the curve $a$. Goldman proved in 
\cite{Gol86} that two twist flows $\Phi_{a_1}$ and $\Phi_{a_2}$ commute if the curves $a_1$ and $a_2$ are disjoint. Recall that a maximal collection of disjoint and non-homotopic simple closed curves on $\Sigma_n$ has cardinality $n-3$. Each such collection of curves therefore defines a Hamiltonian action of the torus $(\R/\pi\Z)^{n-3}$ on $\dtrelcharvar$ via the associated twist flows. Since $\dtrelcharvar$ has dimension $2(n-3)$, this action is maximal and equips $\dtrelcharvar$ with the structure of a symplectic toric manifold.

Deroin-Tholozan proved Theorem ~\ref{thm:deroin-tholozan-symplecto} using Delzant's classification of symplectic toric manifolds, see e.g.\ 
\cite{Can01} for a neat presentation of Delzant's classification. To any symplectic toric manifold you can associate a polytope called the \emph{moment polytope}. Delzant's classification says that the moment polytopes of two symplectic toric manifolds agree if and only if the two symplectic toric manifolds are isomorphic. Here isomorphism means an equivariant symplectomorphism. 

The goal of this article is to describe an explicit equivariant symplectomorphism between the Deroin-Tholozan relative character variety and the complex projective space. This amounts to describing angle coordinates to supplement the action coordinates given by the moment map.

\section{A polygonal model}\label{sec:polygonal-model}

The coordinates for the Deroin-Tholozan relative character variety we are about to construct depend on the choice of a pants decomposition of $\Sigma_n$. For simplicity, we will only detail the construction for a specific choice of pants decomposition (``without crossroads'') like the one illustrated on Figure \ref{fig:introduction}. Each other choice of pants decomposition of $\Sigma_n$ leads to action-angle coordinates by the same construction.

The pants decomposition we are considering is determined by a maximal collection of disjoint and non-homotopic simple closed curves. Specifically, we work with the curves\footnote{We are abusing terminology here. We use the word ``curve'' to mean an actual curve on $\Sigma_n$, its free homotopy class and its lift inside $\pi_1(\Sigma_n)$. Thereafter, the symbols $b_i$ and $c_i$ should, nevertheless, always be interpreted as elements of $\pi_1(\Sigma_n)$. If one wishes to repeat the construction of the coordinates from an abstract pants decomposition, one should start by choosing coherent lifts inside $\pi_1(\Sigma_n)$ of the free isotopy classes of the loops defining the pants decomposition. We refer the reader to Appendix ~\ref{appendix:pants-decomposition} for more consideration on the issue.} $$b_i\coloneqq c_{i+1}^{-1}c_{i}^{-1}\cdots c_1^{-1}\in \pi_1(\Sigma_n)$$ for $i=1,\ldots,n-3$, where the curves $c_i$ refer to the generators of $\pi_1(\Sigma_n)$ fixed in \eqref{eq:fundamental-group-SIgma_n}. The curves $b_i$ are illustrated on Figure ~\ref{fig:curves-b_i}.  We set $b_0\coloneqq c_1^{-1}$ and $b_{n-2}\coloneqq c_n$ for convenience. Below, we fix a maximal Hamiltonian torus action on $\dtrelcharvar$ using a combination of the twist flows along the disjoint curves $b_1,\ldots,b_{n-3}$, see Section ~\ref{sec:torus-action-revisited}. To describe angle coordinates for this torus action, we introduce a polygonal model for Deroin-Tholozan representations.

\begin{figure}[h!]
\begin{center}
\resizebox{\textwidth}{!}{%
\begin{tikzpicture}[framed, decoration={
    markings,
    mark=at position 0.6 with {\arrow{>}}}]
  \draw[postaction={decorate}] (0,-.5) arc(-90:-270: .25 and .5) node[midway, left]{$c_1$};
  \draw[dashed] (0,.5) arc(90:-90: .25 and .5);
  \draw[red!50!yellow, postaction={decorate}] (2,.5) arc(90:270: .25 and .5) node[midway, left]{$b_1$};
  \draw[red!50!yellow, dashed] (2,.5) arc(90:-90: .25 and .5);
  \draw[red!50!yellow, postaction={decorate}] (4,.5) arc(90:270: .25 and .5) node[midway, left]{$b_2$};
  \draw[red!50!yellow, dashed] (4,.5) arc(90:-90: .25 and .5);
  \draw[red!50!yellow, postaction={decorate}] (6,.5) arc(90:270: .25 and .5) node[midway, left]{$b_3$};
  \draw[red!50!yellow, dashed] (6,.5) arc(90:-90: .25 and .5);
  \draw[postaction={decorate}] (8,.5) arc(90:270: .25 and .5) node[midway, left]{$c_6$};
  \draw (8,.5) arc(90:-90: .25 and .5);
  
  \draw (.5,1) arc(180:0: .5 and .25) node[midway, above]{$c_2$};
  \draw[postaction={decorate}] (.5,1) arc(-180:0: .5 and .25);
  \draw (2.5,1) arc(180:0: .5 and .25)node[midway, above]{$c_3$};
  \draw[postaction={decorate}] (2.5,1) arc(-180:0: .5 and .25);
  \draw (4.5,1) arc(180:0: .5 and .25)node[midway, above]{$c_4$};
  \draw[postaction={decorate}] (4.5,1) arc(-180:0: .5 and .25);
  \draw (6.5,1) arc(180:0: .5 and .25)node[midway, above]{$c_5$};
  \draw[postaction={decorate}] (6.5,1) arc(-180:0: .5 and .25);
   
  \draw (0,.5) to[out=0,in=-90] (.5,1);
  \draw (1.5,1) to[out=-90,in=180] (2,.5);
  \draw (0,-.5) to[out=0,in=180] (2,-.5);
  
  \draw (2,.5) to[out=0,in=-90] (2.5,1);
  \draw (3.5,1) to[out=-90,in=180] (4,.5);
  \draw (2,-.5) to[out=0,in=180] (4,-.5);
  
  \draw (4,.5) to[out=0,in=-90] (4.5,1);
  \draw (5.5,1) to[out=-90,in=180] (6,.5);
  \draw (4,-.5) to[out=0,in=180] (6,-.5);
  
  \draw (6,.5) to[out=0,in=-90] (6.5,1);
  \draw (7.5,1) to[out=-90,in=180] (8,.5);
  \draw (6,-.5) to[out=0,in=180] (8,-.5);
\end{tikzpicture}
}
\caption{The case of a 6-punctured sphere: The simple closed curves $b_1,b_2,b_3$ and the peripheral curves $c_1,\ldots,c_6$. }
\label{fig:curves-b_i}
\end{center}
\end{figure}
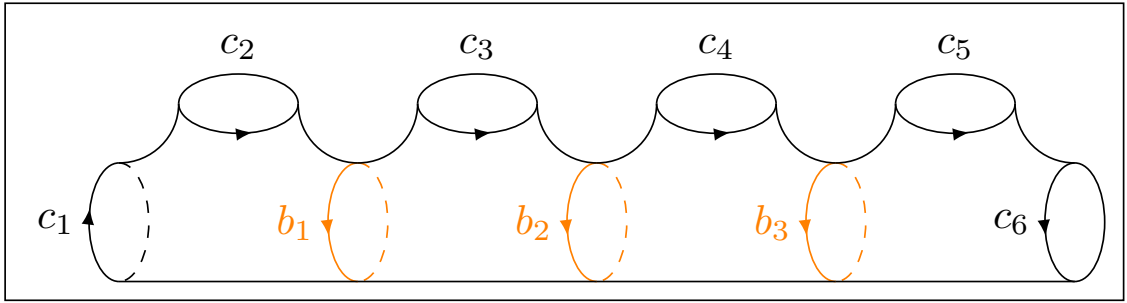

Let $[\phi]$ denote the conjugacy class of a Deroin-Tholozan representation $\phi\colon\pi_1(\Sigma_n)\to G$. By definition of the Deroin-Tholozan relative character variety, $\phi(c_i)$ is elliptic and satisfies $\vartheta(\phi(c_i))=\alpha_i$ for every $i=1,\ldots,n$. Let $$C_1(\phi),\ldots,C_n(\phi)\in\Hyp$$ be the fixed points of $\phi(c_1),\ldots,\phi(c_n)$, respectively. Proposition ~\ref{prop:totally-elliptic} says that $\phi(b_i)$ is elliptic for all $i=1,\ldots,n-3$. Let $$B_1(\phi),\ldots,B_{n-3}(\phi)\in \Hyp$$ be the fixed points of $\phi(b_1),\ldots,\phi(b_{n-3})$, respectively. We emphasize that those fixed points are associated to the representation $\phi$ and not to its conjugacy class $[\phi]$. A different representative of the class $[\phi]$ leads to a different set of fixed points. However, for $A\in G$, it holds that $C_i(A\phi A^{-1})=A\cdot C_i(\phi)$ and $B_i(A\phi A^{-1})=A\cdot B_i(\phi)$. This observation motivates the following. Let $\Hyp^n=\Hyp\times\ldots\times\Hyp$. We introduce the topological quotient $(\Hyp^n\times \Hyp^{n-3})/G$ where $G$ acts diagonally on $\Hyp^n\times \Hyp^{n-3}$. We refer to it as \textit{the moduli space of point configurations} in $\Hyp$. It allows for the definition of a map
\[
\mathfrak P\colon \dtrelcharvar\longrightarrow (\Hyp^n\times \Hyp^{n-3})/G
\] 
that sends $[\phi]$ to the equivalence class of the points $(C_1(\phi),\ldots,C_n(\phi),B_1(\phi),\ldots,B_{n-3}(\phi))$ in the moduli space of point configurations. The map $\mathfrak P$ is injective because a Deroin-Tholozan representation $\phi$ is entirely determined by the fixed points of $\phi(c_1),\ldots,\phi(c_n)$ (recall that the angles of rotation $\alpha_1,\ldots,\alpha_n$ are fixed parameters). Let
\[
\ChainTriang\subset (\Hyp^n\times \Hyp^{n-3})/G
\]
denote the image of the map $\mathfrak P$. The inverse map
\[
\mathfrak P^{-1}\colon \ChainTriang\longrightarrow \dtrelcharvar
\]
maps an equivalence class of points $(C_1,\ldots,C_n,B_1,\ldots,B_{n-3})$ to the conjugacy class of the representation $\phi\colon \pi_1(\Sigma_n)\to G$ that sends each generator $c_i$ of $\pi_1(\Sigma_n)$ to the rotation of angle $\alpha_i$ around $C_i$.

The notation $\ChainTriang$ for the image of $\mathfrak P$ is an abbreviation of \emph{chain of triangles} and is motivated by the following construction. Let $(C_1,\ldots,C_n,B_1,\ldots,B_{n-3})$ be a configuration of points in $\Hyp^n\times\Hyp^{n-3}$ whose isometry class lies in $\ChainTriang$. For convenience, we let $B_0\coloneqq C_1$ and $B_{n-2}\coloneqq C_n$. For every $i=0,\ldots,n-3$, we consider the oriented geodesic triangle $$\Delta_i\coloneqq\Delta(B_i,C_{i+2},B_{i+1})$$ in the upper half-plane, see Figure ~\ref{fig:parameters_symplectom}. The triangles $\Delta_i$ and $\Delta_{i+1}$ share the common vertex $B_i$. The geometric quantities associated to the triangles $\Delta_i$, such as their area or interior angles, are invariant of the isometry class of $(C_1,\ldots,C_n,B_1,\ldots,B_{n-3})$. 
We refer to $(\Delta_0,\ldots,\Delta_{n-3})$ as a \emph{chain of triangles}. We note that the term ``necklace" was used in 
\cite[§0.3]{DeTh19} to hint at the construction; we will, however, stick to ``chain of triangles".  Chains of triangles constitute the polygonal model for the Deroin-Tholozan relative character variety.

\begin{figure}[h]
\begin{center}
\resizebox{\textwidth}{!}{%
\begin{tikzpicture}[framed,font=\sffamily,decoration={
    markings,
    mark=at position 1 with {\arrow{>}}}]]
    
\node[anchor=south west,inner sep=0] at (0,0) {\includegraphics[width=\textwidth]{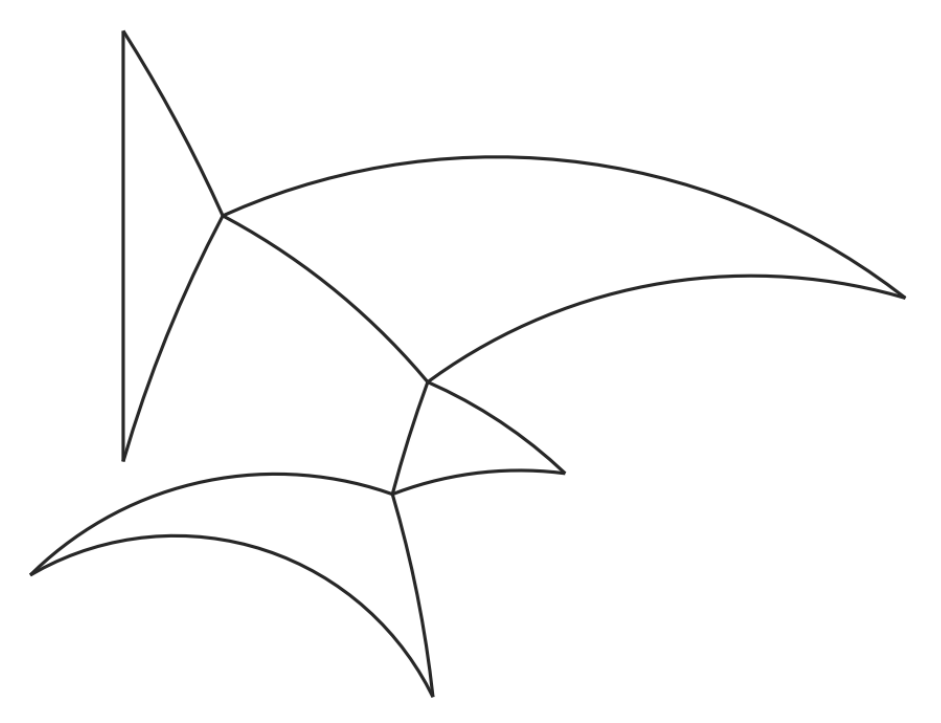}};

\draw (3.55,8) node{\huge $\bullet$};
\draw (3.6,7.5) node{\Large $B_1$};
\draw (6.75,5.4) node{\huge $\bullet$} node[left]{\Large $B_2$};
\draw (6.3,3.7) node{\huge $\bullet$} node[above left]{\Large $B_3$};

\draw (2,4.2) node{\huge $\bullet$} node[below left]{\Large $C_1=B_0$};
\draw (1.95,10.9) node{\huge $\bullet$} node[above left]{\Large $C_2$};
\draw (14.4,6.75) node{\huge $\bullet$} node[below right]{\Large $C_3$};
\draw (8.9,4) node{\huge $\bullet$} node[below right]{\Large $C_4$};
\draw (6.9,.45) node{\huge $\bullet$} node[below right]{\Large $C_5$};
\draw (.55,2.4) node{\huge $\bullet$} node[below right]{\Large $C_6=B_4$};

\draw[violet] (2.7,8.15) node{\huge $\Delta_0$};
\draw[violet] (7.2,7.4) node{\huge $\Delta_1$};
\draw[violet] (7.3,4.6) node{\huge $\Delta_2$};
\draw[violet] (5.5,3) node{\huge $\Delta_3$};
\end{tikzpicture}
}
\end{center}
\caption{Example of a configuration of the fixed points and the associated chain of triangles in the case $n=6$.}\label{fig:parameters_symplectom}
\end{figure}

We advertise two results to convince the reader about the pertinence of the polygonal model for $\dtrelcharvar$. The first concerns angle coordinates which can be read directly from the chain of triangles. We prove below in Section ~\ref{sec:proof-of-theorem} that the angles between the geodesic rays $\overrightarrow{B_iC_{i+2}}$ and $\overrightarrow{B_iC_{i+1}}$ are angle coordinates for the Hamiltonian torus action on $\dtrelcharvar$, see Figure ~\ref{fig:angle-area-parameters}. 

The second example concerns the action coordinates which also appear as geometric quantities in the chain of triangles. For $i=1,\ldots,n-3$, we write
\begin{equation}\label{eq:defn-beta_i}
\beta_i([\phi])\coloneqq\vartheta_{b_i}([\phi])=\vartheta(\phi(b_i))
\end{equation}
for the angle of rotation of the elliptic element $\phi(b_i)\in G$. Let further, in accordance to our previous conventions, $\beta_0([\phi])\coloneqq 2\pi-\alpha_1$ and $\beta_{n-2}([\phi])\coloneqq\alpha_n$. The functions $\beta_i$ are the components of the moment map for the torus action defined by the twist flows along the curves $b_i$. We prove the following below in Subsection ~\ref{sec:the-general-case}, see Figure ~\ref{fig:admissible_chain}.

\begin{lem}\label{lem:interior-angles-chain-of-triangles}
Let $\Delta_i$ be a non-degenerate triangle in the chain built from $\mathfrak P ([\phi])$ for some $[\phi]\in\dtrelcharvar$. The following holds: The triangle $\Delta_i$ is clockwise oriented and the interior angle of $\Delta_i$ at $B_i$ equals $\beta_i([\phi])/2$, the interior angle at $C_{i+2}$ equals $\pi-\alpha_{i+2}/2$ and the interior angle at $B_{i+1}$ equals $\pi-\beta_{i+1}([\phi])/2$.
\end{lem}

The remainder of this section is dedicated to the study of the possible configurations of points inside $\ChainTriang$. We want to find sufficient geometrical conditions for a chain of triangles to be a configuration of fixed points associated to a Deroin-Tholozan representation. We start with the case $n=3$ and then explain how the cases $n\geq 4$ are built from the case $n=3$. 

\subsection{The case of the thrice-punctured sphere}

Assume that $n=3$ and let $\Sigma_3$ be an oriented and connected sphere with three labelled punctures. Let $\alpha=(\alpha_1,\alpha_2,\alpha_3)\in (0,2\pi)^3$ be a triple of angles. At this stage, we make no particular assumption concerning a lower bound for $\alpha_1+\alpha_2+\alpha_3$. Let $[\phi]\in \relcharvarthree$. The following lemma describes the possible configurations of the fixed points $C_1$, $C_2$, $C_3$ of $\phi(c_1)$, $\phi(c_2)$, $\phi(c_3)$. The lemma is transcribed from
\cite{DeTh19} and the proof is included for completeness.

\begin{lem}[\cite{DeTh19}]\label{lem:trichotomy-case-n=3}
The points $C_1$, $C_2$, $C_3\in \Hyp$ are arranged in one of the following three configurations:
\begin{enumerate}
\item All three points coincide and $\alpha_1+\alpha_2+\alpha_3\in \{2\pi,4\pi\}$.
\item The points form a non-degenerate triangle $\Delta(C_1,C_2,C_3)$ which is oriented clockwise and has interior angles $\pi-\alpha_i/2$ at $C_i$ for $i=1,2,3$. Moreover, $\alpha_1+\alpha_2+ \alpha_3>4\pi$.
\item The points form a non-degenerate triangle $\Delta(C_1,C_2,C_3)$ which is oriented anti-clockwise and has interior angles $\alpha_i/2$ at $C_i$ for $i=1,2,3$. Moreover, $\alpha_1+\alpha_2+ \alpha_3<2\pi$.
\end{enumerate}
\end{lem}
\begin{proof}
Assume that $C_i=C_j$ for some $i\neq j$. Let $k\in \{1,2,3\}$ be the third index. Up to permutation of $i$ and $j$, it holds that $\phi(c_k)=\phi(c_i)^{-1}\phi(c_j)^{-1}$ because $c_1c_2c_3=1$ by assumption. So, $\phi(c_k)$ fixes both $C_k$ and $C_i=C_j$. Therefore, all three points must coincide because $\phi(c_k)$ is elliptic. It means that $\phi(c_1),\phi(c_2)$ and $\phi(c_3)$ are rotations about the same point. Since their product is the identity, $\alpha_1+\alpha_2+\alpha_3$ is an integer multiple of $2\pi$.

Assume now that $C_1,C_2$ and $C_3$ are distinct. Let $\zeta_3$ be the geodesic through $C_1$ and $C_2$. Let $\zeta_2$ be the image of $\zeta_3$ by a clockwise rotation of $\pi-\alpha_1/2$ around $C_1$. Let $\zeta_1$ be the image of $\zeta_3$ by an anti-clockwise rotation of $\pi-\alpha_2/2$ around $C_2$, see Figure ~\ref{fig:triangles-case-n=3}. We denote by $\tau_i\colon \Hyp\to\Hyp$ the reflection through the geodesic $\zeta_i$. By construction, $\phi(c_1)=\tau_2\tau_3$ and $\phi(c_2)=\tau_3\tau_1$. Hence, $\phi(c_3)=\phi(c_2)^{-1}\phi(c_1)^{-1}=\tau_1\tau_2$. Since $\phi(c_3)$ fixes $C_3$, the geodesics $\zeta_1$ and $\zeta_2$ must intersect at $C_3$.

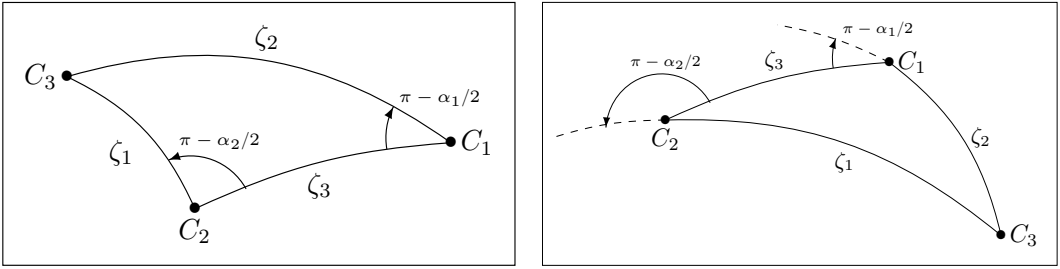
\begin{figure}
\resizebox{7cm}{3.5cm}{
\begin{tikzpicture}[framed,font=\sffamily,scale=.8,decoration={
    markings,
    mark=at position 1 with {\arrow{>}}}]]
 \path (0,0) coordinate (A) (3,-0.7) coordinate (B) (2,-2) coordinate (C) (6,-1) coordinate (D);
 \draw (A) node[left]{$C_3$} node{$\bullet$} to[bend left=25] node[above]{$\zeta_2$} 
 (D) node[right]{$C_1$} node{$\bullet$} to[bend right=10] node[below]{$\zeta_3$}
 (C) node[below]{$C_2$} node{$\bullet$} to[bend right=20] node[below left]{$\zeta_1$} cycle;
 
 \draw[postaction={decorate}] (5,-1.1) arc (190:152:1) node[near end, above right]{\tiny $\pi-\alpha_1/2$};
 \draw[postaction={decorate}] (2.8,-1.7) arc (30:110:1) node[above right]{\tiny $\pi-\alpha_2/2$};
 
\end{tikzpicture}
}
\resizebox{7cm}{3.5cm}{
\begin{tikzpicture}[framed,font=\sffamily,scale=.8,decoration={
    markings,
    mark=at position 1 with {\arrow{>}}}]]
 \path (8,-4) coordinate (A) (3,-0.7) coordinate (B) (2,-2) coordinate (C) (6,-1) coordinate (D);
 \draw (A) node[right]{$C_3$} node{$\bullet$} to[bend right=20] node[right]{\small $\zeta_2$} 
 (D) node[right]{$C_1$} node{$\bullet$} to[bend right=10] node[above]{\small $\zeta_3$}
 (C) node[below]{$C_2$} node{$\bullet$} to[bend left=20] node[below]{\small $\zeta_1$} cycle;
 
 \draw[dashed] (D) to[bend right=8] (4,-0.35);
 \draw[dashed] (C) to[bend right=8] (0,-2.3);
 
 \draw[postaction={decorate}] (5,-1.1) arc (190:160:1) node[near end, above right]{\tiny $\pi-\alpha_1/2$};
 \draw[postaction={decorate}] (2.8,-1.7) arc (30:176:1);
 \draw (2,-1) node{\tiny $\pi-\alpha_2/2$};
 
\end{tikzpicture}
}

\caption{The two non-degenerate configurations of fixed points. On the left: the configuration where $\Delta(C_1,C_2,C_3)$ is clockwise oriented and the interior angles are $\pi-\alpha_i/2$. On the right: the configuration where $\Delta(C_1,C_2,C_3)$ is anti-clockwise oriented and the interior angles are $\alpha_i/2$.}\label{fig:triangles-case-n=3}
\end{figure}

We distinguish two cases according to the orientation of $\Delta(C_1,C_2,C_3)$. 
\begin{itemize}
\item First, assume that the triangle is clockwise oriented. It that case, $\tau_2\tau_3$ is a clockwise rotation around $C_1$ of twice the interior angle at $C_1$. Since $\phi(c_1)$ is by definition an anti-clockwise rotation of angle $\alpha_1$ around $C_1$ and $\phi(c_1)=\tau_2\tau_3$, the interior angle at $C_1$ must be $\pi-\alpha_1/2$. For the same reason, the interior angles at $C_2$ and $C_3$ are $\pi-\alpha_2/2$ and $\pi-\alpha_3/2$, respectively. The positive area of the triangle $\Delta(C_1,C_2,C_3)$ is equal to the angle defect:
\[
\pi-\sum_{i=1}^3 (\pi-\alpha_i/2)=\frac{1}{2}(\alpha_1+\alpha_2+\alpha_3-4\pi).
\]
We conclude that $\alpha_1+\alpha_2+\alpha_3>4\pi$.
\item Conversely, if the triangle is anti-clockwise oriented, then the same argument shows that the interior angle at $C_i$ is $\alpha_i/2$. In this case, the positive area of the triangle $\Delta(C_1,C_2,C_3)$ is equal to
\[
\pi-\sum_{i=1}^3 \alpha_i/2=\frac{1}{2}(2\pi-\alpha_1-\alpha_2-\alpha_3).
\]
We conclude that $\alpha_1+\alpha_2+\alpha_3<2\pi$.
\end{itemize}
\end{proof}

A consequence of Lemma ~\ref{lem:trichotomy-case-n=3} is that $\relcharvarthree$ is empty whenever $\alpha_1+\alpha_2+\alpha_3\in (2\pi,4\pi)$. The next lemma shows that the volume of $[\phi]$ is directly proportional to the signed area of the triangle $\Delta(C_1,C_2,C_3)$.

\begin{lem}\label{lem:volume-case-n=3}
Let $[\phi]\in \relcharvarthree$. Then
\[
\vol([\phi])=-2\cdot [\Delta(C_1,C_2,C_3)].
\]
\end{lem}
\begin{proof}
The proof is an explicit computation of $\vol([\phi])$ from Definition ~\ref{def:volume-representation}. The computations are conducted in the bar resolution for group cohomology and use the explicit form of the fundamental class $[\pi_1(\Sigma_3),\partial \pi_1(\Sigma_3)]$ described in \eqref{eq:fundamental-class}. 

Let $z$ be a base point in $\Hyp$. We start by computing the preimage of the cocycle $\phi^\ast \kappa\in H^2_b(\pi_1(\Sigma_3);\R)$ under the isomorphism $j\colon H^2_b(\pi_1(\Sigma_3),\partial\pi_1(\Sigma_3);\R)\to H^2_b(\pi_1(\Sigma_3);\R)$. This means finding primitives for $\phi^\ast c\colon \pi_1(\Sigma_3)\times \pi_1(\Sigma_3)\to \R$  restricted to the subgroup $\langle c_i\rangle$ of $\pi_1(\Sigma_3)$, where $c$ is the cocycle defined in \eqref{eq:cocycle-c}. For $i=1,2,3$, consider the functions $k_i\colon \langle c_i\rangle \to \R$ defined by $$k_i(c_i)\coloneqq [\Delta(C_i,z,\phi(c_i)z)].$$ We claim that the functions $k_i$ are the desired primitives. By definiton of the bar complex, $k_i$ is a primitive for $\phi^\ast c$ restricted to $\langle c_i\rangle$ if for any two integers $a$ and $b$, it holds that $k_i(c_i^a)+k_i(c_i^b)-k_i(c_i^{a+b})=c(\phi(c_i)^a,\phi(c_i)^b)$. We compute $k_i(c_i^a)+k_i(c_i^b)-k_i(c_i^{a+b})$. This is, by definition of $k_i$, equal to
\[
[\Delta(C_i,z,\phi(c_i)^az)]+[\Delta(C_i,z,\phi(c_i)^bz)]-[\Delta(C_i,z,\phi(c_i)^{a+b}z)].
\]
Since $\phi(c_i)^a$ is an orientation-preserving isometry of the upper half-plane that fixes $C_i$, it holds that
\[
[\Delta(C_i,z,\phi(c_i)^bz)]=[\Delta(C_i,\phi(c_i)^az,\phi(c_i)^{a+b}z)].
\]
Now, we use the following formula. For any four points $A,B,C,D$ in $\Hyp$, we have
\begin{equation}\label{eq:formula-sums_area_four_points}
[\Delta(A,B,C)]+[\Delta(C,D,A)]=[\Delta(B,C,D)]+[\Delta(B,D,A)].
\end{equation}
Note that this property is equivalent to the cocycle formula \eqref{eq:area-cocycle-formula}. Formula \eqref{eq:formula-sums_area_four_points} can be obtained by double counting of the area of the quadrilateral $ABCD$.
\begin{center}
\begin{tikzpicture}[font=\sffamily,scale=.8]
 \path (0,0) coordinate (A) (4,1) coordinate (B) (2,-2) coordinate (C) (6,-1) coordinate (D);
 \draw[thick] (A) node[left]{$A$} to[bend left=25]  
 (D) node[right]{$C$} to[bend right=10]
 (C) node[below]{$D$} to[bend right=20] cycle;
 
 \draw[thick] (D) to[bend right=10] (B);
 \draw[thick] (A) to[bend left=10] (B);
 \draw[dashed] (C) to[bend left=20] (B);
 
 \draw (B) node[above]{$B$};
\end{tikzpicture}
\end{center}
Thus, with $A=C_i$, $B=z$, $C=\phi(c_i)^az$ and $D=\phi(c_i)^{a+b}z$, we deduce
\begin{align*}
k_i(c_i^a)+k_i(c_i^b)-k_i(c_i^{a+b})&=[\Delta(z,\phi(c_i)^az,\phi(c_i)^{a+b}z)]\\
&=c(\phi(c_i)^a,\phi(c_i)^b).
\end{align*}
This proves the claim. Hence
\[
j^{-1}(\phi^\ast\kappa)=[(\phi^\ast c,k_1,k_2,k_3)].
\]
Definition ~\ref{def:volume-representation} says that
\[
\vol([\phi])=[(\phi^\ast c,k_1,k_2,k_3)]\frown [\pi_1(\Sigma_3),\partial\pi_1(\Sigma_3)].
\]
Recall that $[\pi_1(\Sigma_3),\partial\pi_1(\Sigma_3)]$ is the homology class of the 2-chain $(e,c_1,c_2,c_3)$ where $e$ was explicitly described in \eqref{eq:fundamental-class}. Using the explicit expression of the cap product in the bar complex, see e.g.\ \cite[Proposition 5.8]{KaMi96}, we obtain
\begin{align}
\vol([\phi])&=(\phi^\ast c)(e)-k_1(c_1)-k_2(c_2)-k_3(c_3)\nonumber\\
&=[\Delta(z,\phi(c_1)z,\phi(c_1c_2)z)]+[\Delta(z,\phi(c_1c_2)z,\phi(c_1c_2c_3)z)]+[\Delta(z,z,z)]\nonumber\\
&\qquad -[\Delta(C_1,z,\phi(c_1)z)]-[\Delta(C_2,z,\phi(c_2)z)]-[\Delta(C_3,z,\phi(c_3)z)]\nonumber\\
&=[\Delta(z,\phi(c_1)z,\phi(c_1c_2)z)]-\sum_{i=1}^3 [\Delta(C_i,z,\phi(c_i)z)].\label{eq:65}
\end{align}
The volume is independent of the choice of the base point $z$, so we may as well assume $z=C_1$. After obvious cancellations, \eqref{eq:65} becomes
\begin{align*}
\vol([\phi])&= [\Delta(C_1,C_2,\phi(c_2)C_1)]+[\Delta(C_1,C_3,\phi(c_3)C_1)].
\end{align*}
Using $\phi(c_3)C_1=\phi(c_2)^{-1}C_1$, we further compute
\begin{align}
\vol([\phi])&=[\Delta(C_1,C_2,\phi(c_2)C_1)]+[\Delta(\phi(c_2)C_1,\phi(c_2)C_3,C_1)].\label{eq:63}
\end{align}
We make use of \eqref{eq:formula-sums_area_four_points} again. Letting $A=C_1$, $B=C_2$, $C=\phi(c_2)C_1$ and $D=\phi(c_2)C_3$, the relation \eqref{eq:63} becomes
\begin{align}
\vol([\phi])&=[\Delta(C_2,\phi(c_2)C_1,\phi(c_2)C_3)]+[\Delta(C_2,\phi(c_2)C_3,C_1)]\nonumber\\
&=-[\Delta(C_1,C_2,C_3)]+[\Delta(C_1,C_2,\phi(c_2)C_3)].\label{eq:96}
\end{align}
If $C_1=C_2=C_3$ then $\vol([\phi])=0$ by \eqref{eq:96}, and so $\vol([\phi])=-2[\Delta(C_1,C_2,C_3)]$ as desired. Otherwise, we know from the proof of Lemma ~\ref{lem:trichotomy-case-n=3} that all three points are distinct and $\phi(c_2)=\tau_3\tau_1$. Observe that the triangle $\Delta(C_1,C_2,\phi(c_2)C_3)$ is the image under $\tau_3$ of the triangle $\Delta(C_1,C_2,C_3)$ because $\tau_3$ fixes $C_1$ and $C_2$ and $\tau_1$ fixes $C_3$. Hence, for $\tau_3$ is orientation-reversing,
\[
[\Delta(C_1,C_2,\phi(c_2)C_3)]=-[\Delta(C_1,C_2,C_3)].
\]
and \eqref{eq:96} becomes $\vol([\phi])=-2[\Delta(C_1,C_2,C_3)]$. This finishes the proof of the lemma.
\end{proof}

We can compile the conclusions of Lemma ~\ref{lem:trichotomy-case-n=3} and Lemma ~\ref{lem:volume-case-n=3} into the following summary table, see Table ~\ref{tab:configurations-cae-n=3}.

\begin{table}[h]
\centering
  \begin{tabular}{ | c | c | c | c | }
    \hline
    angles & volume & relative Euler class & \makecell{configuration\\ of $\Delta(C_1,C_2,C_3)$} \\ \hline\hline
    $\sum \alpha_i\in\{2\pi,4\pi\}$ & $0$ & \makecell{$k=1$ if $\sum\alpha_i=2\pi$, \\
    $k=2$ if $\sum\alpha_i=4\pi$} & $C_1=C_2=C_3$ \\ \hline
    $\sum \alpha_i>4\pi$ & $4\pi-\sum\alpha_i$ & $k=2$ & \makecell{clockwise oriented, \\ interior angles $\pi-\alpha_i/2$}\\ \hline
    $\sum \alpha_i<2\pi$ & $2\pi-\sum\alpha_i$ & $k=1$ & \makecell{anti-clockwise oriented, \\interior angles $\alpha_i/2$}\\
    \hline
  \end{tabular}
  \bigskip
  \caption{Summary of the different configurations of fixed points in the case $n=3$.}\label{tab:configurations-cae-n=3}
\end{table}

So far, we discussed the properties of the elements of $\relcharvarthree$. Now, we address the question of existence and uniqueness of such elements. If $\alpha_1+\alpha_2+\alpha_3>4\pi$, then there exists a unique clockwise oriented triangle $\Delta_\alpha$ in $\Hyp$, up to orientation-preserving isometries, with interior angles $\pi-\alpha_i/2$. The composition of the reflections through the sides of $\Delta_\alpha$, as in the proof of Lemma ~\ref{lem:trichotomy-case-n=3}, defines an element of $\relcharvarthree$. This element is unique because $\Delta_\alpha$ is unique up to isometry. If $\alpha_1+\alpha_2+\alpha_3=4\pi$, then $\Delta_\alpha$ is degenerate to a point. The rotations of angle $\alpha_i$ around that point define an element of $\relcharvarthree$. This element is unique because $G$ acts transitively on the upper half-plane. The case $\alpha_1+\alpha_2+\alpha_3\leq 2\pi$ is similar. In conclusion, we obtain

\begin{lem}\label{lem:uniqueness-representation-case-n=3}
If $\alpha_1+\alpha_2+\alpha_3\in (0,2\pi]\cup [4\pi,6\pi)$, then $\relcharvarthree$ is a singleton and $\ChainTriang$ consists only of the isometry class of $\Delta_\alpha$. If $\alpha_1+\alpha_2+\alpha_3\in (2\pi,4\pi)$, then $\relcharvarthree$ and $\ChainTriang$ are empty.
\end{lem}

\subsection{The general case}\label{sec:the-general-case}

Let us first prove that the chain of triangles associated to a Deroin-Tholozan representation has the geometric properties stated in Lemma ~\ref{lem:interior-angles-chain-of-triangles}. The curves $b_1,\ldots,b_{n-3}$ illustrated in Figure ~\ref{fig:curves-b_i} define a pants decomposition of $\Sigma_n$ into $n-2$ pair of pants $P_0,\ldots,P_{n-3}$. The pair of pants $P_i$ has boundary curves $b_i^{-1}$, $c_{i+2}$ and $b_{i+1}$ (with the convention that $b_0=c_1^{-1}$ and $b_{n-2}=c_n$). Let $[\phi]\in\dtrelcharvar$. The conjugacy class $[\phi\restriction_{P_i}]$ of the restriction of $\phi$ to $P_i$ lies in the relative character variety $\Rep_{\varpi_i}(P_i,G)$ where $\varpi_i$ is the vector of angles $(2\pi-\beta_i([\phi]),\alpha_{i+2},\beta_{i+1}([\phi]))$. Indeed, the functions $\beta_i$, introduced in \eqref{eq:defn-beta_i}, measure the angle of rotation of the evaluation on the curve $b_i$. Deroin-Tholozan observed in 
\cite{DeTh19} that the relative Euler classes of all the $[\phi\restriction_{P_i}]$ are automatically maximal. The argument is simple. Since the volume of a representation is additive, it holds that $\vol([\phi])=\vol([\phi\restriction_{P_0}])+\ldots+\vol([\phi\restriction_{P_{n-3}}])$ or equivalently 
\begin{align}
2\pi (n-1)-\sum_{i=1}^n\alpha_i&=\sum_{i=0}^{n-3} \left(2\pi k([\phi\restriction_{P_i}])-(2\pi-\beta_i([\phi])+\alpha_{i+2}+\beta_{i+1}([\phi]))\right) \label{eq:sum-areas}\\
&=2\pi\sum_{i=0}^{n-3}k([\phi\restriction_{P_i}]) - 2\pi(n-3)- \sum_{i=1}^n\alpha_i. \nonumber
\end{align}
So, we conclude $k([\phi\restriction_{P_0}])+\ldots+k([\phi\restriction_{P_{n-3}}])=2(n-2)$. Table ~\ref{tab:configurations-cae-n=3} says that $k([\phi\restriction_{P_i}])\in\{1,2\}$ for every $i$. Therefore, it must hold $k([\phi\restriction_{P_i}])=2$ for every $i=0,\ldots,n-3$ and the relative Euler class of each $[\phi\restriction_{P_i}]$ is indeed maximal. 

We can apply the case distinction of Table ~\ref{tab:configurations-cae-n=3} to the triangles $\Delta_0,\ldots,\Delta_{n-3}$ built from $\mathfrak P([\phi])$. Let $\Delta_i$ be any of these triangles. For $k([\phi\restriction_{P_i}])=2$, we have $2\pi-\beta_i([\phi])+\alpha_{i+2}+\beta_{i+1}([\phi])\geq 4\pi$ or equivalently
\[
\alpha_{i+2}+\beta_{i+1}([\phi])-\beta_i([\phi])\geq 2\pi.
\]
If $\alpha_{i+2}+\beta_{i+1}([\phi])-\beta_i([\phi])> 2\pi$, then $\Delta_i$ is a non-degenerate, clockwise oriented, triangle with interior angles $\beta_i([\phi])/2$, $\pi-\alpha_{i+1}/2$ and $\pi-\beta_{i+1}([\phi])/2$, such as stated in Lemma ~\ref{lem:interior-angles-chain-of-triangles}. If $\alpha_{i+2}+\beta_{i+1}([\phi])-\beta_i([\phi])= 2\pi$, then $\Delta_i$ is degenerate to a point. In both cases,
\[
\vol([\phi\restriction_{P_i}])=-2[\Delta_i]=-(\alpha_{i+2}+\beta_{i+1}([\phi])-\beta_i([\phi])-2\pi).
\]
Observe that, thanks to the clockwise orientation of $\Delta_i$, its area is always nonnegative. Table ~\ref{tab:configurations-general-case} summarizes the above discussion.

\begin{table}
 \centering
  \begin{tabular}{ | c | c | c | }
    \hline
    angles & $\vol([\phi\restriction_{P_i}])$  &   \makecell{configuration of \\ $\Delta_i=\Delta(B_i,C_{i+2},B_{i+1})$} \\ \hline\hline
    $\alpha_{i+2}+\beta_{i+1}-\beta_i> 2\pi$ & $-(\alpha_{i+2}+\beta_{i+1}-\beta_i-2\pi)$ &  \makecell{\footnotesize{clockwise oriented,}\\ \footnotesize{interior angles $\beta_i/2$,}\\ \footnotesize{$\pi-\alpha_{i+2}/2$ and $\pi-\beta_{i+1}/2$}}\\ \hline
    $\alpha_{i+2}+\beta_{i+1}-\beta_i= 2\pi$ & $0$ & \makecell{\footnotesize{degenerate,}\\ $B_i=C_{i+2}=B_{i+1}$.} \\ \hline
  \end{tabular}
  \bigskip
  \caption{The two different natures of $[\phi\restriction_{P_i}]$.}\label{tab:configurations-general-case}
  \end{table}

It turns out that Lemma ~\ref{lem:interior-angles-chain-of-triangles} completely determines $\ChainTriang$ in the case the triangles are non-degenerate. This allows for a purely geometric description of the subset $\ChainTriang$ of the moduli space of point configurations in $\Hyp$. This is the purpose of Lemma ~\ref{lem:description-chain-triangles-alpha}. In the case none of the triangles are degenerate, there is a cleaner formulation of the sufficient conditions for a chain of triangles to lie in $\ChainTriang$. We state it as Corollary ~\ref{cor:description-chain-triangles-alpha-non-degenerate}.

\begin{lem}\label{lem:description-chain-triangles-alpha}
Let $(C_1,\ldots,C_n,B_1,\ldots,B_{n-3})$ be a configuration of points in the upper half-plane and let $(\Delta_0,\ldots,\Delta_{n-3})$ be the chain of triangles defined by $\Delta_i=\Delta(B_i,C_{i+2},B_{i+1})$, with the usual convention that $B_0=C_1$ and $B_{n-2}=C_{n}$. Further, for $i=0,\ldots,n-4$, let
\[
\beta_{i+1}\coloneqq\sum_{j=0}^i 2[\Delta_j]-\sum_{j=1}^{i+2}\alpha_j+2(i+2)\pi.
\]
The isometry class of $(C_1,\ldots,C_n,B_1,\ldots,B_{n-3})$ lies in $\ChainTriang$ if and only if the following conditions on $\Delta_0,\ldots,\Delta_{n-3}$ are fulfilled.
\begin{enumerate}
\item If $[\Delta_i]>0$, then $\Delta_i$ is clockwise oriented and has interior angle $\beta_i/2$ at $B_i$, $\pi-\alpha_{i+2}/2$ at $C_{i+2}$ and $\pi-\beta_{i+1}/2$ at $B_{i+1}$. Moreover, if $i=0$, then $\Delta_0$ has interior angle $\pi-\alpha_{1}/2$ at $C_{1}$ and if $i=n-3$, then $\Delta_{n-3}$ has interior angle $\pi-\alpha_{n}/2$ at $C_{n}$.
\item If $[\Delta_i]=0$, then $B_i=C_{i+2}=B_{i+1}$.
\end{enumerate}
\end{lem}
\begin{proof}
The forward implication follows from the discussion that led to Table ~\ref{tab:configurations-general-case}. To prove the backward implication, start with a configuration of points $(C_1,\ldots,C_n,B_1,\ldots,B_{n-3})$ in the upper half-plane that satisfy the properties \textit{(1)} and \textit{(2)}. We construct a Deroin-Tholozan representation $[\phi]$ such that $\mathfrak P ([\phi])$ is the isometry class of $(C_1,\ldots,C_n,B_1,\ldots,B_{n-3})$. Define $\phi(c_i)$ to be the rotation of angle $\alpha_i$ with fixed point $C_i$. We first claim that $\phi$ is a representation $\pi_1(\Sigma_n)$ into $G$, i.e.\ $\phi(c_1)\cdots\phi(c_n)=1$. Indeed, arguing as in the proof of Lemma ~\ref{lem:trichotomy-case-n=3}, we observe that $\phi(c_2)^{-1}\phi(c_1)^{-1}$ is a rotation of angle
\[
2[\Delta_0]-\alpha_1-\alpha_2+4\pi
\]
around $B_1$. This angle is by definition equal to $\beta_1$. Similarly, $\phi(c_{n-2})^{-1}\cdots\phi(c_1)^{-1}$ is a rotation of angle 
\[
2[\Delta_{n-4}]-\alpha_{n-2}-(2\pi-\beta_{n-4})+4\pi
\]
around $B_{n-3}$. Again, observe that this angle is by definition equal to $\beta_{n-3}$. Moreover, the same argument shows that $\phi(c_{n-1})\phi(c_n)$ is also a rotation of angle $\beta_{n-3}$ around $B_{n-3}$. Hence $\phi(c_{n-2})^{-1}\cdots\phi(c_1)^{-1}=\phi(c_{n-1})\phi(c_n)$. This proves that $\phi$ is a representation of $\pi_1(\Sigma_n)$ into $G$. It is immediate from the definition of $\phi$ that $[\phi]\in \relcharvar$. We now prove that $\vol([\phi])=-\lambda$. In fact, using both the additivity of the volume and Lemma ~\ref{lem:volume-case-n=3}, we obtain
\[
\vol([\phi])=-2\sum_{i=0}^{n-3}[\Delta_i].
\]
We express $[\Delta_{n-3}]$ in therms of the interior angles of $\Delta_{n-3}$:
\begin{align*}
-2[\Delta_{n-3}]&=-2\pi+(2\pi-\alpha_n)+(2\pi-\alpha_{n-1})+\beta_{n-3}\\
&=2\pi-\alpha_n-\alpha_{n-1}+\beta_{n-3}.
\end{align*}
By definition of $\beta_{n-3}$ it holds
\[
-2\sum_{i=0}^{n-4}[\Delta_i]=-\beta_{n-3}-\sum_{i=1}^{n-2}\alpha_i+2(n-2)\pi.
\]
We conclude that $\vol([\phi])=2\pi(n-1)-\sum_{i=1}^n\alpha_i=-\lambda$ and thus $[\phi]\in\dtrelcharvar$. By construction, the chain of triangles built from $\mathfrak P([\phi])$ is $(\Delta_0,\ldots,\Delta_{n-3})$. We conclude that the isometry class of $(C_1,\ldots,C_n,B_1,\ldots,B_{n-3})$ lies in $\ChainTriang$ as desired.
\end{proof}

If all the triangles are non-degenerate, then Lemma ~\ref{lem:description-chain-triangles-alpha} admits a cleaner formulation which we state as a corollary.

\begin{cor}\label{cor:description-chain-triangles-alpha-non-degenerate}
Let $(C_1,\ldots,C_n,B_1,\ldots,B_{n-3})$ be a configuration of points in the upper half-plane and let $(\Delta_0,\ldots,\Delta_{n-3})$ be the chain of triangles it defines. Assume that none of the triangles $\Delta_i$ are degenerate. The isometry class of $(C_1,\ldots,C_n,B_1,\ldots,B_{n-3})$ lies in $\ChainTriang$ if and only if the following conditions on $\Delta_0,\ldots,\Delta_{n-3}$ are fulfilled.
\begin{enumerate}
\item The triangle $\Delta_i$ is clockwise oriented and has interior angle $\pi-\alpha_{i+2}/2$ at $C_{i+2}$. Moreover, if $i=0$, then $\Delta_0$ has interior angle $\pi-\alpha_{1}/2$ at $C_{1}$ and if $i=n-3$, then $\Delta_{n-3}$ has interior angle $\pi-\alpha_{n}/2$ at $C_{n}$.
\item The interior angles of $\Delta_i$ and $\Delta_{i+1}$ at their common vertex $B_{i+1}$ are supplementary.
\end{enumerate}
\end{cor}

The conditions of Corollary ~\ref{cor:description-chain-triangles-alpha-non-degenerate} are illustrated on Figure ~\ref{fig:admissible_chain}.

\begin{figure}[h]
\centering
\resizebox{\textwidth}{!}{%
\begin{tikzpicture}[framed,font=\sffamily,decoration={
    markings,
    mark=at position 1 with {\arrow{>}}}]]
    
\node[anchor=south west,inner sep=0] at (0,0) {\includegraphics[width=\textwidth]{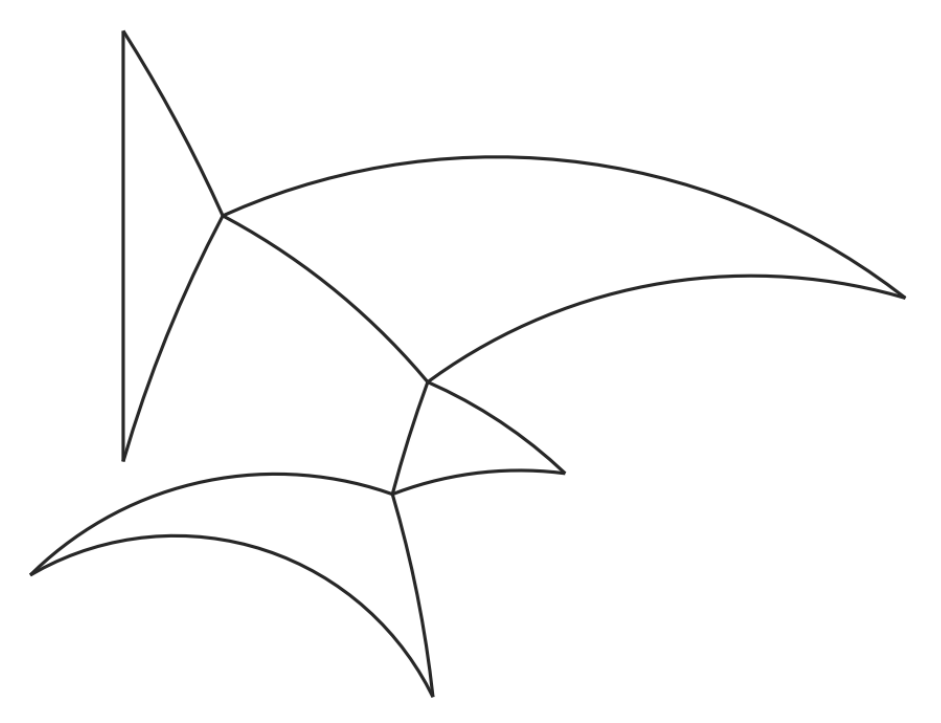}};

\draw (3.55,8) node{\huge $\bullet$};
\draw (3.6,7.5) node{\Large $B_1$};
\draw (6.75,5.4) node{\huge $\bullet$} node[left]{\Large $B_2$};
\draw (6.3,3.7) node{\huge $\bullet$} node[above left]{\Large $B_3$};

\draw (2,4.2) node{\huge $\bullet$} node[below left]{\Large $C_1$};
\draw (1.95,10.9) node{\huge $\bullet$} node[above left]{\Large $C_2$};
\draw (14.4,6.75) node{\huge $\bullet$} node[below right]{\Large $C_3$};
\draw (8.9,4) node{\huge $\bullet$} node[below right]{\Large $C_4$};
\draw (6.9,.45) node{\huge $\bullet$} node[below right]{\Large $C_5$};
\draw (.55,2.4) node{\huge $\bullet$} node[below right]{\Large $C_6$};

\draw[thick, violet, postaction={decorate}] (2.6,6) arc (45:95:.8) node[at end, left]{$\pi-\alpha_1/2$};
\draw[thick, violet, postaction={decorate}] (2,9.5) arc (-90:-60:1.4) node[at end, right]{$\pi-\alpha_2/2$};
\draw[thick, violet, postaction={decorate}] (12,8.1) arc (150:180:2) node[midway, left]{$\pi-\alpha_3/2$};
\draw[thick, violet, postaction={decorate}] (8,4.7) arc (150:185:1.2) node[near start, above right]{$\pi-\alpha_4/2$};
\draw[thick, violet, postaction={decorate}] (6.6,2) arc (95:140:1) node[at start, right]{$\pi-\alpha_5/2$};
\draw[thick, violet, postaction={decorate}] (3,3) arc (0:20:2.2) node[midway, right]{$\pi-\alpha_6/2$};

\draw[thick, red!50!yellow, postaction={decorate}] (3,9) arc (109:250:1) node[at end, below right]{$\pi-\beta_1/2$};
\draw[thick, red!50!yellow, postaction={decorate}] (4.5,7.5) arc (-30:25:1) node[midway, right]{$\beta_1/2$};
\draw[thick, red!50!yellow, postaction={decorate}] (7.5,5.9) arc (25:127:1) node[midway, above]{$\pi-\beta_2/2$};
\draw[thick, red!50!yellow, postaction={decorate}] (6.6,4.7) arc (255:325:.8) node[at end, above right]{$\beta_2/2$};
\draw[thick, red!50!yellow, postaction={decorate}] (7,3.9) arc (5:70:.8) node[at start, below right]{$\pi-\beta_3/2$};
\draw[thick, red!50!yellow, postaction={decorate}] (5.5,3.8) arc (170:280:.8) node[at end, right]{$\beta_3/2$};
\end{tikzpicture}
}
\caption{Example of a configuration of points whose isometry class lies in $\ChainTriang$ in the case $n=6$.}\label{fig:admissible_chain}
\end{figure}

\begin{rem}
We point out that $\mathfrak P(\phi)$ makes sense for any totally elliptic representation $\phi\colon \pi_1(\Sigma_n)\to G$ and not only for Deroin-Tholozan representations. The triangles in the induced chain, however, do not need to be clockwise oriented (whereas it is the case for Deroin-Tholozan representations by Lemma ~\ref{lem:description-chain-triangles-alpha}). This makes it harder to define similar coordinates for such totally elliptic representations. Nevertheless, the question whether totally elliptic representations $[\phi]\in \relcharvar$ are necessarily of type Deroin-Tholozan (i.e.\ satisfy $\vol([\phi])=-\lambda$) remains, to the author's knowledge, open for $n\geq 5$. The claim is true for $n\in\{3,4\}$ because of volume considerations, as explained in~\cite[Remark~2.8]{Ma20}.
\end{rem}

\subsection{The torus action revisited}\label{sec:torus-action-revisited}

We explained how to use Proposition ~\ref{prop:totally-elliptic} to associate to a maximal collection of simple closed curves on $\Sigma_n$ a maximal torus action on the Deroin-Tholozan relative character variety. In this section we first fix a parametrization of the maximal torus action associated to the curves $b_1,\ldots,b_{n-3}$ we intend to work with. We should emphasize that our choice of parametrization is different from that of Deroin-Tholozan in 
\cite{DeTh19}. Deroin-Tholozan work with the torus action given by the Hamiltonian flows of the functions $\beta_1,\ldots,\beta_{n-3}\colon \dtrelcharvar\to (0,2\pi)$ defined in \eqref{eq:defn-beta_i}. We choose to consider the Hamiltonian flows of the functions $1/2(\beta_{i+1}-\beta_{i})$ instead. They define an effective action 
\begin{equation}\label{eq:torus-action}
\T^{n-3}\coloneqq(\R/2\pi\Z)^{n-3}\acts\dtrelcharvar.
\end{equation}
The reason for considering $1/2(\beta_{i+1}-\beta_{i})$ instead of $\beta_i$ is that the expression $1/2(\beta_{i+1}-\beta_{i})$ is up to constant equal to the area of the triangle $\Delta_i$, see Table ~\ref{tab:configurations-general-case}.

To see that the Hamiltonian flows of the functions $1/2(\beta_{i+1}-\beta_{i})$ indeed give an effective torus action, we follow \cite[§4]{Gol86} and write down explicitly the action on representations. For $\theta=(\theta_1,\ldots,\theta_{n-3})\in\R^{n-3}$ we introduce the notation
\[
\overline\theta_i\coloneqq\theta_i-\theta_{i-1}, \quad i=1,\ldots,n-3,
\]
where it is understood that $\theta_0=0$. The unique elliptic element of $G$ that fixes $z\in \Hyp$ with angle of rotation $\vartheta\in (0,2\pi)$ is denoted by $$\rot_\vartheta(z),$$ see also \eqref{eq:elliptic-matrix-given-angle-and-fixed-point}. Let $[\phi]\in \dtrelcharvar$ and let $B_i\in \Hyp$ be the fixed point of $\phi(b_i)$, with the convention that $B_{n-2}=C_n$ is the fixed point of $\phi(c_n)$. Under the action \eqref{eq:torus-action} the image of $\theta\in \R^{n-3}$ acting on $[\phi]\in \dtrelcharvar$ is the conjugacy class of the representation $\theta\cdot\phi$ given by
\[
(\theta\cdot\phi)(c_i)=\left(\prod_{j=1}^{i-2}\rot_{\overline\theta_j}(B_j)\right)\cdot\phi(c_i)\cdot\left(\prod_{j=1}^{i-2}\rot_{\overline\theta_j}(B_j)\right)^{-1}.
\]
Or more explicitly

\begin{equation}\label{eq:definiton-torus-action-on-representations}\arraycolsep=1.4pt\def\arraystretch{1.5}
\left\{\begin{array}{ll}
(\theta\cdot\phi)(c_1)=\phi(c_1),\\
(\theta\cdot\phi)(c_2)=\phi(c_2),\\
(\theta\cdot\phi)(c_3)=\rot_{\overline\theta_1}(B_1)\cdot\phi(c_3)\cdot\rot_{\overline\theta_1}(B_1)^{-1},\\
(\theta\cdot\phi)(c_4)=\rot_{\overline\theta_1}(B_1)\rot_{\overline\theta_2}(B_2)\cdot\phi(c_4)\cdot\rot_{\overline\theta_2}(B_2)^{-1}\rot_{\overline\theta_1}(B_1)^{-1},\\
\vdots\\
(\theta\cdot\phi)(c_{n-1})=\left(\prod_{i=1}^{n-3}\rot_{\overline\theta_i}(B_i)\right)\cdot\phi(c_{n-1})\cdot\left(\prod_{i=1}^{n-3}\rot_{\overline\theta_i}(B_i)\right)^{-1},\\
(\theta\cdot\phi)(c_n)=\left(\prod_{i=1}^{n-3}\rot_{\overline\theta_i}(B_i)\right)\cdot\phi(c_n)\cdot\left(\prod_{i=1}^{n-3}\rot_{\overline\theta_i}(B_i)\right)^{-1}.
\end{array}\right.
\end{equation}
Observe that both $\phi(c_{n-1})$ and $\phi(c_n)$ are conjugated by the same element because they correspond to the same triangle in the chain built from $\mathfrak P([\phi])$. We leave it to the reader to check that the $\R^{n-3}$-action \eqref{eq:definiton-torus-action-on-representations} is a well-defined group action that is $2\pi$-periodic in each factor. The reader is referred to 
\cite{Gol86} for explanations on how the explicit action \eqref{eq:definiton-torus-action-on-representations} corresponds to the torus action \eqref{eq:torus-action} given by the Hamiltonian flows of the functions $1/2(\beta_{i+1}-\beta_{i})$. 

The action \eqref{eq:torus-action} is a Hamiltonian torus action on $\dtrelcharvar$ equipped with the symplectic form $1/\lambda\cdot\omega_\G$ with moment map $\mu\colon \dtrelcharvar\to \R^{n-3}$ defined by
\begin{equation}\label{eq:definiton-moment-map}
\mu_i([\phi])\coloneqq \frac{1}{2\lambda}(\alpha_{i+2}+\beta_{i+1}([\phi])-\beta_i([\phi])-2\pi).
\end{equation}
Recall that $\lambda$ is the scaling factor introduced in Definition ~\ref{def:scaling-factor}. Comparing Table ~\ref{tab:configurations-general-case} one observes that 
\[
\mu_i([\phi])=\frac{1}{\lambda}[\Delta_i].
\]
The image of $\mu$ inside $\R^{n-3}$ is the moment polytope for the action of $\T^{n-3}$ on $\dtrelcharvar$. The area of the triangles in a chain corresponding to an element of $\dtrelcharvar$ are nonnegative numbers that sum up to $\lambda/2$:
\[
[\Delta_i]\in [0,\lambda/2]\subset [0,\pi),\quad [\Delta_0]+\ldots+[\Delta_{n-3}]=\lambda/2.
\]
This is a consequence of the additivity of the volume and Lemma ~\ref{lem:volume-case-n=3}; the computation is similar to \eqref{eq:sum-areas}. Hence
\begin{equation}\label{eq:moment-polytope-equations-mu}
\mu_i\in [0,1/2],\quad \mu_1+\ldots+\mu_{n-3}\leq 1/2.
\end{equation}
This shows that the moment polytope is the $(n-3)$-simplex in $\R^{n-3}$ with side length $1/2$. If we compare Lemma ~\ref{lem:description-chain-triangles-alpha} and the range of $[\Delta_i]$ we deduce
\begin{equation}\label{eq:range-beta_i}
\beta_i\in \left[ 2(i+1)\pi-\sum_{j=1}^{i+1}\alpha_j,\sum_{j=i+2}^n\alpha_j-2\pi (n-i-2)\right]\subset (0,2\pi).
\end{equation}
Observe that the length of the range of the function $\beta_i$ is equal to $\lambda$ and that the range of the function $\beta_{i+1}$ is obtained from that of $\beta_i$ by a translation of $2\pi-\alpha_{i+2}$. The moment polytope equations \eqref{eq:moment-polytope-equations-mu} translated in terms of $\beta_i$ read
\[
\left\{\begin{array}{l}
\beta_1\geq 4\pi-\alpha_1-\alpha_2,\\
\beta_i-\beta_{i+1}\leq \alpha_{i+2}-2\pi,\quad i=1,\ldots, n-4,\\
\beta_{n-3}\leq \alpha_{n-1}+\alpha_n-2\pi.\end{array}
\right.
\]

\begin{lem}\label{lem:fibres-moment-map-degenerated-triangles}
The fibre of the moment map $\mu$ over a point of the moment polytope is an embedded torus of dimension $k\in\{0,\ldots,n-3\}$ in $\dtrelcharvar$, where $(n-3)-k$ is the number of degenerate triangles in the chain associated to any element of the fibre.
\end{lem}

Lemma ~\ref{lem:fibres-moment-map-degenerated-triangles} is a standard fact about symplectic toric manifolds. The toric fibres of maximal dimension $n-3$ form an open dense subset of $\dtrelcharvar$. They are called \emph{regular fibres} of the moment map. Their union is the preimage under $\mu$ of the interior of the moment polytope. We denote this subspace by
$$\regdtrelcharvar \subset \dtrelcharvar.$$
It is a full measure subset that consists exactly of the points where $\T^{n-3}$ acts freely.

The torus action \eqref{eq:torus-action} explicitly described by \eqref{eq:definiton-torus-action-on-representations} may look, in the words of a retired analyst, baroque. It can be easily visualized if we translate it to our polygon model. This is yet another pleasant feature of the polygonal model for the Deroin-Tholozan relative character variety. For this purpose, we declare the bijection $\mathfrak P\colon \dtrelcharvar\to \ChainTriang$ to be equivariant and define therewith an action of $\T^{n-3}$ on $\ChainTriang$. Let $\theta\in\T^{n-3}$. We denote the fixed points of $(\theta\cdot\phi)(c_i)$ and $(\theta\cdot\phi)(b_i)$ by $C_i^\theta$ and $B_i^\theta$, respectively. From \eqref{eq:definiton-torus-action-on-representations}, we obtain that
\begin{equation}\label{eq:image-of-C_i-torus-action}
C_1^\theta=C_1,\quad C_2^\theta=C_2,\quad C_3^\theta=\rot_{\overline\theta_1}(B_1)\cdot C_3,\quad \ldots,\quad C_n^\theta=\prod_{i=1}^{n-3}\rot_{\overline\theta_i}(B_i) \cdot C_n,
\end{equation}
and
\begin{equation}\label{eq:image-of-B_i-torus-action}
B_1^\theta=B_1,\quad B_2^\theta=\rot_{\overline\theta_1}(B_1)\cdot B_2,\quad \ldots,\quad B_{n-3}^\theta=\prod_{i=1}^{n-4}\rot_{\overline\theta_i}(B_i) \cdot B_{n-3}.
\end{equation}
This means that $\theta\in\T^{n-3}$ acts on a chain of triangles in $\ChainTriang$ by successive anti-clockwise rotations of the sub-chain of triangles $\Delta_{i},\ldots,\Delta_{n-3}$ by an angle $\overline\theta_i$ around $B_i$, see Figure ~\ref{fig:torus-action-chain-of-triangles}.

\begin{figure}
\begin{center}
\resizebox{8.5cm}{!}{%
\begin{tikzpicture}[framed,font=\sffamily,decoration={
    markings,
    mark=at position 1 with {\arrow{>}}}]]
    
\node[anchor=south west,inner sep=0] at (0,0) {\includegraphics[trim=0cm 7cm 1cm 7cm]{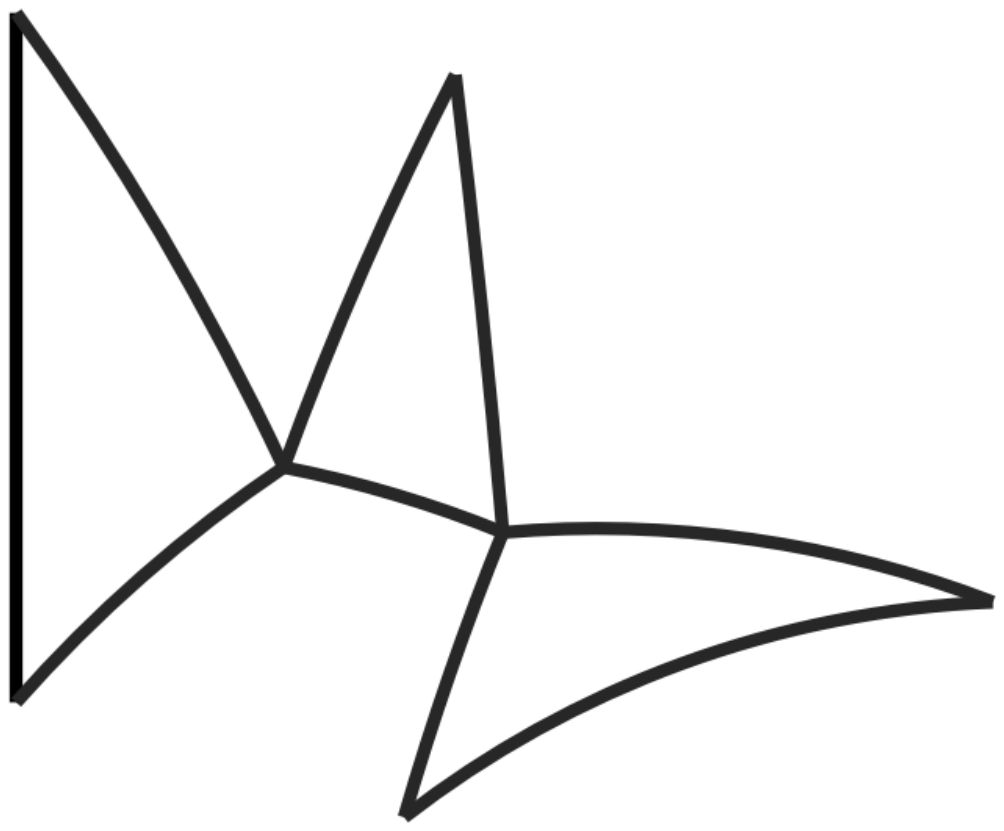}};

\draw (4,8)  node{\HUGE $\Delta_0$};
\draw (9,8.5)  node{\HUGE $\Delta_1$};
\draw (12,4.5)  node{\HUGE $\Delta_2$};
\end{tikzpicture}
}
\resizebox{8.5cm}{!}{%
\begin{tikzpicture}[framed,font=\sffamily,decoration={
    markings,
    mark=at position 1 with {\arrow{>}}}]]
    
\node[anchor=south west,inner sep=0] at (0,0) {\includegraphics[trim=0cm 6cm 1cm 6cm]{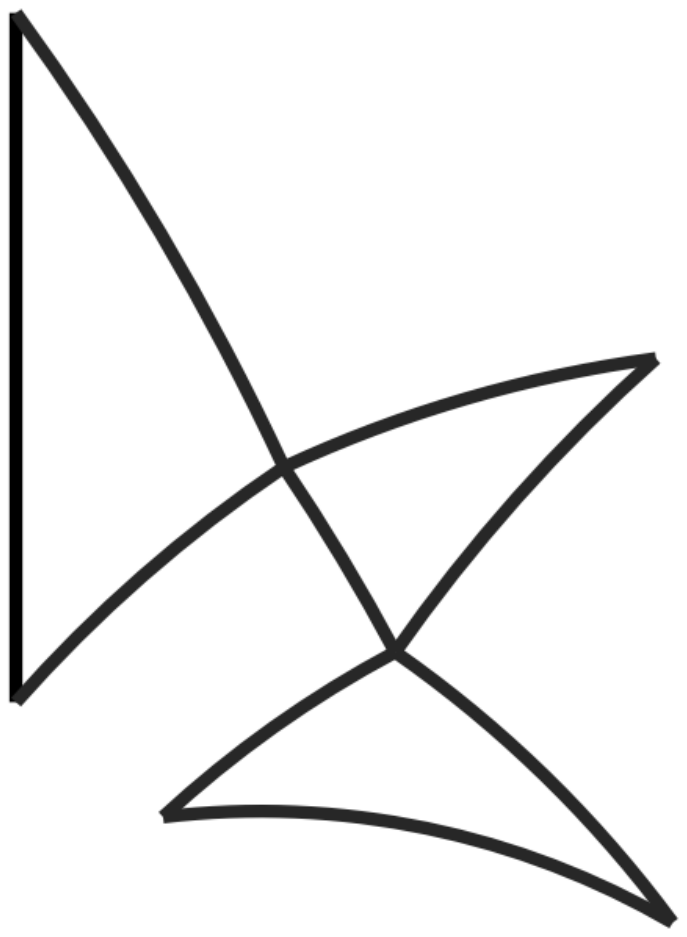}};

\draw[line width=1mm, violet, postaction={decorate}] (8,12) arc (120:15:3) node[midway, above right]{\HUGE $+ \frac{\pi}{4}$};

\draw[violet] (9.65,9.05) node{\HUGE $\bullet$};
\draw[violet] (9.55,8)  node{\HUGE $B_1$};

\end{tikzpicture}
}
\resizebox{8.5cm}{!}{%
\begin{tikzpicture}[framed,font=\sffamily,decoration={
    markings,
    mark=at position 1 with {\arrow{>}}}]]
    
\node[anchor=south west,inner sep=0] at (0,0) {\includegraphics[trim=0cm 6cm 1cm 6cm]{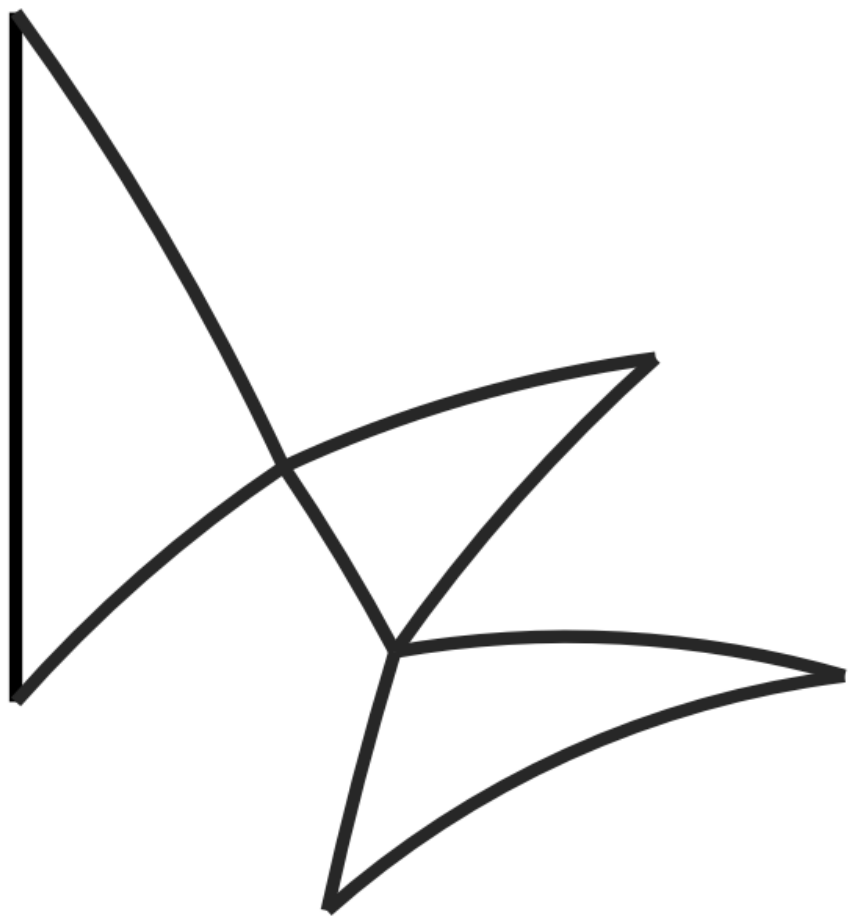}};

\draw[line width=1mm, red!50!yellow, postaction={decorate}] (14,6) arc (5:53:3.8) node[midway, above right]{\HUGE $- \frac{\pi}{4}$};

\draw[red!50!yellow] (10.05,5.95) node{\HUGE $\bullet$};
\draw[red!50!yellow] (9.05,5.95)  node{\HUGE $B_2$};

\end{tikzpicture}
}
\end{center}
\caption{The action of $(\theta_1,\theta_2)=(-\pi/4,0)\in \T^2$ in the case $n=5$. The initial configuration is drawn on top. The intermediate configuration is obtained after rotating the triangles $\Delta_1$ and $\Delta_2$ together anti-clockwise by an angle $\overline\theta_1=-\pi/4$ around $B_1$. The triangle $\Delta_0$ is not moved during this step. The final configuration is obtained from the intermediate configuration after rotating the triangle $\Delta_2$ anti-clockwise by an angle $\overline\theta_2=\pi/4$ around $B_2$. The triangles $\Delta_0$ and $\Delta_1$ are not moved during this step.}\label{fig:torus-action-chain-of-triangles}
\end{figure}

\section{Complex projective coordinates}\label{sec:complex-proj-coordinates}

In this section, we construct an explicit equivariant symplectomorphism from $\dtrelcharvar$ to $\CP^{n-3}$. It is based on the polygonal model developed in the previous section.

\subsection{Definition of the map}
Let $[\phi]\in \dtrelcharvar$. We associate to $[\phi]$ a collection of parameters defined using the chain of triangles $\Delta_0,\ldots,\Delta_{n-3}$ built from  $\mathfrak P ([\phi])\in\ChainTriang$. The first collection of parameters $a_0,\ldots,a_{n-3}\colon\dtrelcharvar\to[0,\infty)$ are called \emph{area parameters} and are defined to be twice the area of the triangle $\Delta_i$:
\[
a_i([\phi])\coloneqq 2[\Delta_i], \quad i=0,\ldots,n-3.
\]
Lemma ~\ref{lem:interior-angles-chain-of-triangles} implies that 
\begin{equation}\label{eq:area-parameters-in-terms-of-angles}
a_i([\phi])=\alpha_{i+2}+\beta_{i+1}([\phi])-\beta_i([\phi])-2\pi\geq 0.
\end{equation}
Each area parameter takes value in $[0,\lambda]$ and their sum is constant and equal to the scaling factor $\lambda>0$. This was already observed earlier when we computed the moment polytope for the moment map \eqref{eq:definiton-moment-map}. In particular, at least one area parameter is nonzero. Since the functions $\beta_i$ are analytic functions of $\dtrelcharvar$ (see Lemma ~\ref{lem:angle-of-rotation-elliptic}), the area parameters are analytic functions as well. Observe that, because of \eqref{eq:definiton-moment-map}, it holds that
\begin{equation}\label{eq:relation-area-parameters-moment-map}
a_i([\phi])=2\lambda\cdot\mu_i([\phi]).
\end{equation}

The second set of parameters $\sigma_1,\ldots,\sigma_{n-3}\colon\dtrelcharvar\to \R/2\pi\Z$ are called \emph{angle parameters}. Their definition is more subtle as one needs to consider the case where some triangles of the chain are degenerate to a point. First, assume that $a_i([\phi])\neq 0$ for every $i=0,\ldots,n-3$ or equivalently that $[\phi]$ lies in a regular fibre of the moment map. This ensures that the fixed points $B_i(\phi)$, $C_{i+1}(\phi)$, $C_{i+2}(\phi)$, abbreviated $B_i$, $C_{i+1}$, $C_{i+2}$ below, are distinct points for every $i$. In this case, we define, for $i=1,\ldots,n-3$, the angle $\gamma_i([\phi])\in \R/2\pi\Z$ to be the oriented angle between the geodesic rays $\overrightarrow{B_iC_{i+2}}$ and $\overrightarrow{B_iC_{i+1}}$ (see Figure ~\ref{fig:angle-area-parameters}):
\[
\gamma_i([\phi])\coloneqq\angle (\overrightarrow{B_iC_{i+2}},\overrightarrow{B_iC_{i+1}}).
\]
In less rigorous words, $\gamma_i$ is the angle between the triangle $\Delta_{i+1}$ and the triangle $\Delta_i$. In the case that some of the area parameters vanish, we define $\gamma_i([\phi])\in \R/2\pi\Z$ to be
\[
\gamma_i([\phi])\coloneqq\left\{\begin{array}{ll}
0, &\text{if } a_j([\phi])=0,\forall j<i,\\
\pi-\alpha_{i+2}/2, &\text{if }a_i([\phi])=0\text{ and } \exists j<i,\, a_j([\phi])> 0,\\
\angle (\overrightarrow{B_iC_{i+2}},\overrightarrow{B_iC_{m(i)+2}}), &\text{if }a_i([\phi])>0\text{ and } \exists j<i,\, a_j([\phi])> 0,
\end{array}\right.
\]
where $m(i)$ is the largest index smaller than $i$ such that $a_{m(i)}([\phi])>0$, see Figure ~\ref{fig:angle-area-parameters}. Whenever $[\phi]$ lies in a regular fibre of the moment map, then $m(i)=i-1$ for every $i$, showing that the definition of $\gamma_i$ is consistent. Note that the parameters $\gamma_i([\phi])$ are well-defined in the sense that if $a_{i}([\phi])>0$ then $B_i\neq C_{i+2}$ and $B_i\neq C_{m(i)+2}$. We finally define the angle parameters $\sigma_i([\phi])\in \R/2\pi\Z$ for $i=1,\ldots,n-3$ by
\[
\sigma_i([\phi])\coloneqq\sum_{j=1}^i\gamma_j([\phi]).
\]
Below, we will refer to both sets of parameters $\{\gamma_1,\ldots,\gamma_{n-3}\}$ and $\{\sigma_1,\ldots,\sigma_{n-3}\}$ as angle parameters, without distinction. The angle parameters $\gamma_i$ and $\sigma_i$ are analytic functions on $\regdtrelcharvar$ and may have points of discontinuity on the complement of $\regdtrelcharvar$.

\begin{figure}[h]
\begin{center}
\resizebox{7cm}{5.8cm}{%
\begin{tikzpicture}[framed,font=\sffamily,decoration={
    markings,
    mark=at position 1 with {\arrow{>}}}]]
    
\node[anchor=south west,inner sep=0] at (0,0) {\includegraphics[scale=.3, trim=0cm 7cm 1cm 7cm]{fig/step1}};

\draw (.7,1) node{\huge $\bullet$} node[below left]{$C_1$};
\draw (.7,4.4) node{\huge $\bullet$} node[above left]{$C_2$};
\draw (2.9,4.1) node{\huge $\bullet$} node[above right]{$C_3$};
\draw (5.55,1.45) node{\huge $\bullet$} node[above right]{$C_4$};
\draw (2.7,.4) node{\huge $\bullet$} node[below right]{$C_5$};

\draw (2,2.1) node{\huge $\bullet$} node[below]{$B_1$};
\draw (3.1,1.8) node{\huge $\bullet$} node[above right]{$B_2$};

\draw[thick, red!50!yellow, postaction={decorate}] (2.3,3) arc (70:110:1) node[midway, above]{\huge $\gamma_1$};

\draw[thick, red!50!yellow, postaction={decorate}] (4,1.8) arc (0:91:.9) node[midway, above right]{\huge $\gamma_2$};

\draw[violet] (1.25,2.4) node{\huge $a_0$};
\draw[violet] (2.65,2.5) node{\huge $a_1$};
\draw[violet] (3.6,1.4) node{\huge $a_2$};
\end{tikzpicture}
}
\resizebox{7cm}{5.8cm}{%
\begin{tikzpicture}[framed,font=\sffamily,decoration={
    markings,
    mark=at position 1 with {\arrow{>}}}]]
    
\node[anchor=south west,inner sep=0] at (0,0) {\includegraphics[scale=.3, trim=0cm 7cm 1cm 7cm]{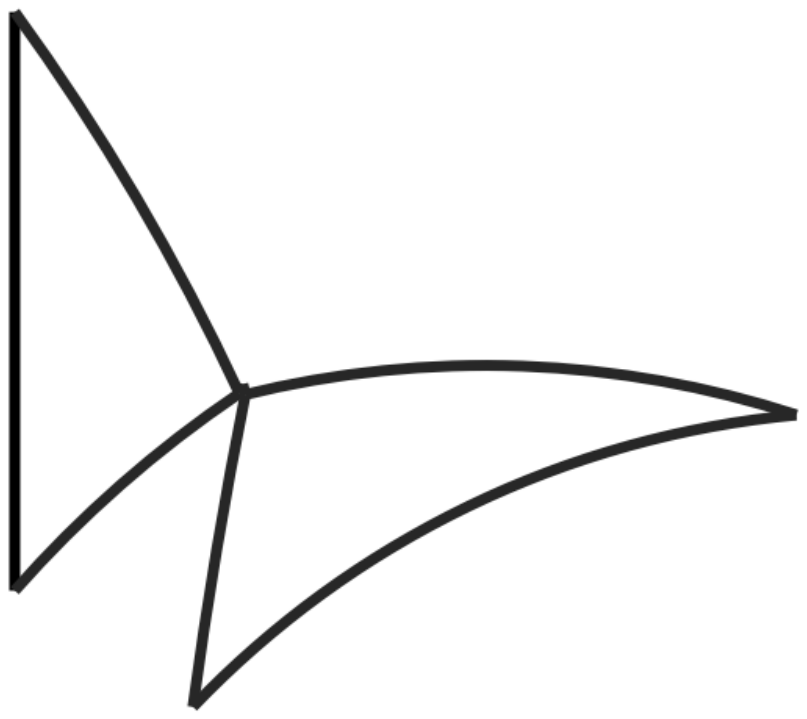}};

\draw (1.1,1.3) node{\huge $\bullet$} node[below left]{$C_1$};
\draw (1.1,4.15) node{\huge $\bullet$} node[above left]{$C_2$};

\draw (2.25,2.2) node{\huge $\bullet$};
\draw (2.15, 2.4) node[above right]{\tiny $B_1=C_3=B_2$};

\draw (5.05,2.15) node{\huge $\bullet$} node[above right]{$C_4$};
\draw (2.1,.7) node{\huge $\bullet$} node[below right]{$C_5$};

\draw[thick, red!50!yellow, postaction={decorate}] (4,2.4) arc (0:119:1.8) node[midway, above right]{\huge $\gamma_2$};

\draw[violet] (1.6,2.4) node{\huge $a_0$};
\draw[violet] (2.8,1.9) node{\huge $a_2$};

\draw[violet] (5,4.3) node{\large $a_1=0$};
\draw[red!50!yellow] (5,4.9) node{\large $\gamma_1=\pi-\alpha_3/2$};

\end{tikzpicture}
}
\end{center}
\caption{The angles $\gamma_i$ for two configurations of fixed points in the case $n=5$. The left picture corresponds to a representation in a regular fiber of the moment map. The right picture corresponds to a representation for which $a_1$ vanishes.}\label{fig:angle-area-parameters}
\end{figure}

Area and angle parameters completely characterize Deroin-Tholozan representations. To see this, we introduce the map
\begin{align}
\Cgot\colon& \dtrelcharvar\to \CP^{n-3}\nonumber\\
& [\phi] \mapsto \left[\sqrt{a_0([\phi])}:\sqrt{a_1([\phi])}e^{i\sigma_1([\phi])}:\ldots:\sqrt{a_{n-3}([\phi])}e^{i\sigma_{n-3}([\phi])}\right].\label{eq:definition_Cgot}
\end{align}
Recall that the area parameters are nonnegative and cannot vanish all at once. Moreover, recall that both the area and angle parameters are geometric invariants of $\mathfrak P([\phi])\in \ChainTriang$. We thus see that the map $\Cgot \colon \dtrelcharvar\to \CP^{n-3}$ is well-defined. 

Recall that the Deroin-Tholozan relative character variety has the structure of a symplectic toric manifold with symplectic form $1/\lambda\cdot \omega_\G$ and the torus action \eqref{eq:torus-action}. We equip $\CP^{n-3}$ with the Fubini-Study symplectic form $\omega_{\FS}$ of volume $\pi^{n-3}/(n-3)!$, see e.g.\ \cite{Can01} for more details on the symplectic nature of the complex projective space. We further equip $\CP^{n-3}$ with the $\T^{n-3}$-action defined in homogeneous coordinates by
\begin{equation}\label{eq:definition-torus-action-on-CPn}
\theta\cdot [z_0:z_1:\ldots:z_{n-3}]\coloneqq[z_0:e^{-i\theta_1}z_1:\ldots:e^{-i\theta_{n-3}}z_{n-3}], \quad \theta\in\T^{n-3}.
\end{equation}
This action is a maximal effective Hamiltonian torus action with moment map 
\begin{equation}\label{eq:definition-moment-map-torus-action-on-CPn}
\nu([z_0:z_1:\ldots:z_{n-3}])\coloneqq\left(\frac{\vert z_1\vert ^2}{2\vert z\vert},\ldots,\frac{\vert z_{n-3}\vert ^2}{2\vert z\vert}\right)\in \R^{n-3},
\end{equation}
where $\vert z\vert^2 \coloneqq \vert z_0\vert^2+\vert z_1\vert ^2+\ldots+\vert z_{n-3}\vert ^2$. The main result of this paper is

\begin{thm}[Theorem ~\ref{thm:symplectomorphism-intro}]\label{thm:symplectomorphism}
The map $\Cgot \colon \dtrelcharvar\to \CP^{n-3}$ defined in \eqref{eq:definition_Cgot} is an isomorphism of symplectic toric manifolds. In other words, $\Cgot$ is an equivariant diffeomorphism such that 
\[
\mu=\nu\circ \Cgot\quad\text{ and }\quad\Cgot^\ast \omega_\FS=1/\lambda\cdot\omega_\G.
\]
\end{thm}

The proof of Theorem ~\ref{thm:symplectomorphism} is unfolded, step by step, below. The main difficulty in the proof is showing that the map $\mathfrak C$ is differentiable at the points in the irregular fibres of the moment map --- that is, on the complement of $\regdtrelcharvar$. On these fibres the area parameters can vanish causing the angles parameters $\gamma_i$ to be discontinuous. 

A direct consequence of Theorem ~\ref{thm:symplectomorphism}, already pointed out in 
\cite{DeTh19}, says that the symplectic volume of the Deroin-Tholozan relative character variety is equal to
\[
\frac{(\lambda\pi)^{n-3}}{(n-3)!}.
\]

\subsection{A Wolpert-type formula}

Theorem ~\ref{thm:symplectomorphism} implies that the coordinates $$\{a_1,\ldots,a_{n-3},\sigma_1,\ldots,\sigma_{n-3}\}$$ are action-angle coordinates for the Deroin-Tholozan relative character variety. In particular, as a corollary of Theorem ~\ref{thm:symplectomorphism}, we prove that the coordinates are Darboux coordinates for the Goldman symplectic form.

\begin{cor}[Theorem ~\ref{thm:wolpert-like-formula}]\label{cor:analogue-Wolpert-formula}
The restriction of the Goldman form on $\dtrelcharvar$ to $\regdtrelcharvar$ can be written as
\[
\omega_\G=\frac{1}{2}\sum_{i=1}^{n-3}da_i\wedge d\sigma_i=\frac{1}{2}\sum_{i=1}^{n-3}d\gamma_i\wedge d\beta_i.
\]
\end{cor}
\begin{proof}
At any point $[z_0:z_1:\ldots:z_{n-3}]\in \CP^{n-3}$ for which $z_i\neq 0$ for all $i=0,\ldots,n-3$, the Fubini-Study form can be written as
\[
\omega_\FS=\sum_{i=1}^{n-3}d\nu_i\wedge d\theta_i,
\]
where $(\nu_1,\ldots,\nu_{n-3})$ are the components of the moment map \eqref{eq:definition-moment-map-torus-action-on-CPn} and $\theta_i$ is the complex argument of $z_i$ (defined up to a constant). The coordinates $\{\nu_1,\ldots,\nu_{n-3},\theta_1,\ldots,\theta_{n-3}\}$ are action-angle coordinates for the integrable dynamics on $\CP^{n-3}$ defined by \eqref{eq:definition-torus-action-on-CPn}. Theorem ~\ref{thm:symplectomorphism} says that $\omega_\G=\lambda\cdot\Cgot^\ast \omega_\FS$. It also implies $\Cgot^\ast  d\nu_i=d\mu_i=da_i/(2\lambda)$ (where we used \eqref{eq:relation-area-parameters-moment-map}) and $\Cgot^\ast  d\theta_i=d\sigma_i$. Hence, on $\regdtrelcharvar$, it holds that
\[
\omega_\G=\lambda\cdot\Cgot^\ast \omega_\FS=\lambda\sum_{i=1}^{n-3}\Cgot^\ast  d\nu_i\wedge \Cgot^\ast  d\theta_i=\frac{1}{2}\sum_{i=1}^{n-3}da_i\wedge d\sigma_i.
\]
Using $da_i=d\beta_{i+1}-d\beta_i$, with $d\beta_0=d\beta_{n-2}=0$, and $d\sigma_{i+1}-d\sigma_i=d\gamma_{i+1}$, it follows that
\[
\sum_{i=1}^{n-3}da_i\wedge d\sigma_i=\sum_{i=1}^{n-3}d\gamma_i\wedge d\beta_i.
\]
\end{proof}

Corollary ~\ref{cor:analogue-Wolpert-formula} implies that, even if the definition of the coordinates $\{a_1,\ldots,a_{n-3},\sigma_1,\ldots,\sigma_{n-3}\}$ depends on the choice of a pants decomposition of $\Sigma_n$, the 2-form $\sum_{i=1}^{n-3}da_1\wedge d\sigma_i$ does not. This is because the Goldman symplectic form on the Deroin-Tholozan relative character variety is defined without any reference to a pants decomposition.

\section{Proof of Theorem ~\ref{thm:symplectomorphism}}\label{sec:proof-of-theorem}

In this section we prove that the map $\Cgot \colon \dtrelcharvar\to \CP^{n-3}$ defined in \eqref{eq:definition_Cgot} is an equivariant symplectomorphism.

\subsection{Homeomorphism property} We start by proving
\begin{prop}\label{prop:homeomorphism}
The map $\Cgot \colon \dtrelcharvar\to \CP^{n-3}$ is a homeomorphism.
\end{prop}

To prove Proposition ~\ref{prop:homeomorphism}, we first show that $\Cgot$ is a continuous bijection. This is done in Lemmata ~\ref{lem:surjective}, ~\ref{lem:injective} and ~\ref{lem:continuity} below. Then, since $\dtrelcharvar$ is compact by Theorem ~\ref{thm:deroin-tholozan-symplecto} and $\CP^{n-3}$ is a Hausdorff space, it follows that $\Cgot$ is a homeomorphism.

\begin{rem}
This is the only place where we use that $\dtrelcharvar$ is compact. It is not hopeless to, alternatively, finish the proof of Proposition ~\ref{prop:homeomorphism} by computing the inverse map of $\Cgot$ and prove that it is continuous. This would be a proof that $\dtrelcharvar$ is homeomorphic to $\CP^{n-3}$ that does not use the compactness result of \cite{DeTh19}.
\end{rem}

\begin{lem}\label{lem:surjective}
The map $\Cgot \colon \dtrelcharvar\to \CP^{n-3}$ is surjective.
\end{lem}
\begin{proof}
Let $[z_0:\ldots:z_{n-3}]\in \CP^{n-3}$. We may assume that $\vert z_0\vert^2+\ldots+\vert z_{n-3}\vert ^2=\lambda$ and that the first nonzero $z_i$ is a positive real number. The goal is to build a representation $\phi\colon \pi_1(\Sigma_n)\to G$ such that $[\phi]\in \dtrelcharvar$ and $\Cgot([\phi])=[z_0:\ldots:z_{n-3}]$. To do so, we build a chain of triangles satisfying the properties of Lemma ~\ref{lem:description-chain-triangles-alpha} such that the corresponding Deroin-Tholozan representation has the desired image under $\Cgot$. The triangles are constructed in $n-2$ steps starting with $\Delta_0$.
\begin{enumerate}
\setcounter{enumi}{-1}
\item \emph{Step 0.} Let $C_1$ be any point in $\Hyp$. If $z_0=0$, then we let $C_2\coloneqq B_1\coloneqq C_1$. Now, assume $z_0\neq 0$. By assumption, $z_0$ is a positive real number.  First, observe that $\vert z_0\vert^2/2=z_0^2/2\leq \lambda/2<\pi$. Further, let $\beta_1\coloneqq z_0^2-\alpha_1-\alpha_2+4\pi$. Note that, since $4\pi>\alpha_1+\alpha_2$ and $\lambda-\alpha_1-\alpha_2<-2\pi$, it holds $\beta_1\in (0,2\pi)$. In particular, there exists a clockwise oriented hyperbolic triangle $\Delta_0=\Delta(C_1,C_2,B_1)$ such that
\begin{itemize}
\item $\Delta_0$ has area $z_0^2/2$,
\item $\Delta_0$ has interior angles $\pi-\alpha_1/2$ at $C_1$ and $\pi-\alpha_2/2$ at $C_2$.
\end{itemize}
The triangle $\Delta_0$ is not uniquely determined as it can be arbitrarily rotated around $C_1$. We fix one such triangle $\Delta_0$. By construction, $\Delta_0$ has interior angle $\pi-\beta_1/2$ at $B_1$.

\item \emph{Step 1.} If $z_1=0$, then we let $C_3=B_2=B_1$. Now, assume $z_1\neq 0$. Again, observe that $\vert z_1\vert^2/2\leq \lambda/2<\pi$ and $\beta_2\coloneqq\vert z_1\vert^2-\alpha_3+\beta_1+2\pi\in(0,2\pi)$, because $-\alpha_3+\beta_1+2\pi\geq 6\pi-\alpha_1-\alpha_2-\alpha_3>0$ and $\lambda-\alpha_1-\alpha_2-\alpha_3<-4\pi$. So, there exists a clockwise oriented hyperbolic triangle $\Delta_1=\Delta(B_1,C_3,B_2)$ such that
\begin{itemize}
\item $\Delta_1$ has area $\vert z_1\vert ^2/2$,
\item $\Delta_1$ has interior angles $\pi-\alpha_3/2$ at $C_3$ and $\beta_1/2$ at $B_1$.
\end{itemize}
If $z_0=0$, then as before $\Delta_1$ can be arbitrarily rotated around $B_1$. If $z_0\neq 0$, then $\Delta_1$ is uniquely determined if we further impose
\begin{itemize}
\item the angle $\angle (\overrightarrow{B_1C_{3}},\overrightarrow{B_1C_2})$ is equal to the complex argument of $z_1$.
\end{itemize}
If $\Delta_1$ is non-degenerate, then by construction it has interior angle $\pi-\beta_2/2$ at $B_2$.

\item \emph{Step 2.} If $z_2=0$, then we let $C_4=B_3=B_2$. Now, assume $z_2\neq 0$. It holds $\vert z_2\vert^2/2\leq \lambda/2<\pi$ and $\beta_3\coloneqq\vert z_2\vert^2-\alpha_4+\beta_2+2\pi\in(0,2\pi)$. There exists a clockwise oriented hyperbolic triangle $\Delta_2=\Delta(B_2,C_4,B_3)$ such that
\begin{itemize}
\item $\Delta_2$ has area $\vert z_2\vert ^2/2$,
\item $\Delta_2$ has interior angles $\pi-\alpha_4/2$ at $C_4$ and $\beta_2/2$ at $B_2$.
\end{itemize}
If $z_0=0$ and $z_1=0$, then $\Delta_2$ can be arbitrarily rotated around $B_2$. If $z_0\neq 0$ and $z_1=0$, then $\Delta_2$ is uniquely determined if we impose
\begin{itemize}
\item the angle $\angle (\overrightarrow{B_2C_{4}},\overrightarrow{B_2C_2})$ is equal to the complex argument of $z_2$.
\end{itemize}
If $z_1\neq 0$, then $\Delta_2$ is uniquely determined if we impose
\begin{itemize}
\item the angle $\angle (\overrightarrow{B_2C_{4}},\overrightarrow{B_2C_3})$ is equal to the complex argument of $z_2$ minus the complex argument of $z_1$.
\end{itemize}
If $\Delta_2$ is non-degenerate, then by construction it has interior angle $\pi-\beta_3/2$ at $B_3$.

\end{enumerate}

This process can be repeated $n-5$ times until the point $C_{n}=B_{n-2}$ has been constructed. The last triangle in the chain, namely $\Delta_{n-3}=\Delta(B_{n-3},C_{n-1},C_{n})$, has area $\vert z_{n-3}\vert ^2/2$ and interior angles $\pi-\alpha_{n-1}/2$ at $C_{n-1}$ and $\beta_{n-3}/2$ at $B_{n-3}$, assuming $z_{n-3}\neq 0$. Since
\[
\vert z_{n-3}\vert ^2=\lambda-\vert z_{0}\vert ^2-\ldots-\vert z_{n-4}\vert ^2=\alpha_n+\alpha_{n-1}-\beta_{n-3}-2\pi,
\]
it follows that the interior angle of $\Delta_{n-3}$ at $C_n$ is $\pi-\alpha_n/2$. Therefore, the configuration of points $(C_1,\ldots,C_n,B_1,\ldots,B_{n-3})$ we just built satisfies the properties of Lemma ~\ref{lem:description-chain-triangles-alpha}. Its preimage under $\mathfrak P$ is the conjugacy class of a Deroin-Tholozan representation $[\phi]$. It follows from the construction that $\mathfrak C ([\phi])=[z_0:\ldots:z_{n-3}]$.
\end{proof}

\begin{lem}\label{lem:injective}
The map $\Cgot \colon \dtrelcharvar\to \CP^{n-3}$ is injective.
\end{lem}
\begin{proof}
Let $[\phi]$ and $[\phi']$ be two elements of $\dtrelcharvar$ such that $\Cgot([\phi])=\Cgot([\phi'])$. We want to prove that $[\phi]=[\phi']$. To achieve this, it is sufficient to check that the chain of triangles built from $\mathfrak P([\phi])$ and $\mathfrak P([\phi'])$ are isometric because $\mathfrak P$ is injective. 

Let $a_i=\alpha_{i+2}+\beta_{i+1}-\beta_i-2\pi$ and $a_i'=\alpha_{i+2}+\beta_{i+1}'-\beta_i'-2\pi$ be the area parameters associated to $[\phi]$ and $[\phi']$, respectively. Similarly, let $\gamma_i$, $\sigma_i$ and $\gamma_i'$, $\sigma_i'$ be their respective angle parameters. Recall that $a_0+\ldots+a_{n-3}=a_0'+\ldots+a_{n-3}'=\lambda$. By definition of $\Cgot$ (see \eqref{eq:definition_Cgot}), since we assume $\Cgot([\phi])=\Cgot([\phi'])$, it follows that  $a_i=a_i'$ for every $i=0,\ldots,n-3$. Moreover, it also implies $\sigma_i=\sigma_i'+\sigma$ for every $i=1,\ldots,n-3$, where $\sigma$ is some constant. Note that, if $a_0=a_0'>0$, then $\sigma=0$. 

From $a_i=a_i'$, it follows $\beta_i=\beta_i'$ for every $i$. Thus, by Lemma ~\ref{lem:interior-angles-chain-of-triangles}, the oriented triangles $\Delta_i$ and $\Delta_i'$ inside $\Hyp$ have the same interior angles and are therefore isometric for every $i$. To conclude that the two chains are isometric, it suffices to check that the angles between consecutive non-degenerate triangles in each chain are equal. Since $\sigma_i=\sigma_i'+\sigma$, we have $\gamma_1=\gamma_1'+\sigma$ and $\gamma_i=\gamma_i'$ for every $i\geq 2$. Since $\sigma=0$ whenever $a_0=a_0'>0$, this shows that the angles between the corresponding pairs of consecutive non-degenerate triangles in each chain are equal. We conclude that $\mathfrak P([\phi])=\mathfrak P([\phi'])$ and thus $[\phi]=[\phi']$.
\end{proof}

\begin{lem}\label{lem:continuity}
The map $\Cgot \colon \dtrelcharvar\to \CP^{n-3}$ is continuous.
\end{lem}
\begin{proof}
The continuity of $\Cgot$ is immediate at any point in a regular fibre of the moment map. The task is more subtle when some area parameters vanish because of the discontinuity of the angle parameters $\gamma_i$.

Let $[\phi_0]\in \dtrelcharvar$. We prove that $\Cgot$ is continuous at $[\phi_0]$. Let $i\geq 0$ be the smallest index such that $a_i([\phi_0])>0$. We work in the chart $\{z_i\neq 0\}$ of $\CP^{n-3}$. Continuity is guaranteed for every index $j$ such that $a_j([\phi_0])=0$. It thus suffices to prove that $\sigma_j([\phi])-\sigma_{i}([\phi])$ is continuous around $[\phi_0]$ for every index $j>i$ such that $a_j([\phi])>0$. Let $i=i_1<i_2<\ldots<i_d$ denote the indices such that $a_{i_l}([\phi])>0$. Because of telescopic cancellations, it is sufficient to prove that $\sigma_{i_{l+1}}([\phi])-\sigma_{i_l}([\phi])$ is continuous around $[\phi_0]$ for every $l=1,\ldots,d-1$. 

We treat the case $l=1$. Let $i=i_1<i_2=j$. We first consider the case $j-i=1$ first. In this case, 
\[
\sigma_j([\phi])-\sigma_{i}([\phi])=\gamma_{i+1}([\phi]).
\]
Since $a_{i+1}([\phi_0])>0$ and $a_i([\phi_0])>0$ by assumption, the angle parameter $\gamma_{i+1}([\phi])$ is a continuous function around $[\phi_0]$. 

Now, we consider the general case $j-i\geq 2$. Recall that it corresponds the situation where $a_i([\phi_0])>0$, $a_j([\phi_0])>0$ and $a_l([\phi_0])=0$ for all $i<l<j$. For clarity, we let $[\phi_k]$ be a sequence that converges to $[\phi_0]$. We will assume that $a_{\ell}([\phi_k])>0$ for every $k$ and every $i\leq \ell\leq j$. The argument below can be adapted to the case where, for some $i<\ell<j$, $a_{\ell}([\phi_k])=0$ for infinitely many $k$. Since we assume $a_j([\phi_k])>0$ and $a_i([\phi_k])>0$, it holds $B_j(\phi_k)\neq C_{j+2}(\phi_k)$ and $B_{i+1}(\phi_k)\neq C_{i+2}(\phi_k)$. For $k$ large enough, we may assume that the geodesics $\overrightarrow{B_j(\phi_k)C_{j+2}(\phi_k)}$ and $\overrightarrow{B_{i+1}(\phi_k)C_{i+2}(\phi_k)}$ intersect, because they do so at the limit. Recall that, by definition, $\gamma_j=\angle\big(\overrightarrow{B_jC_{j+2}},\overrightarrow{B_{i+1}C_{i+2}}\big)$ (see Figure ~\ref{fig:angle-area-parameters}) and so
\begin{equation}\label{eq:01}
\gamma_j([\phi_0])=\lim_{k\to\infty}\angle\big(\overrightarrow{B_j(\phi_k)C_{j+2}(\phi_k)},\overrightarrow{B_{i+1}(\phi_k)C_{i+2}(\phi_k)}\big).
\end{equation}
The angle $\angle\big(\overrightarrow{B_jC_{j+2}},\overrightarrow{B_{i+1}C_{i+2}}\big)$ can be decomposed as follows:
\[
\angle\big(\overrightarrow{B_jC_{j+2}},\overrightarrow{B_{j}C_{j+1}}\big)
+\angle\big(\overrightarrow{B_{j}C_{j+1}},\overrightarrow{B_{j-1}C_{j+1}}\big)
+\angle\big(\overrightarrow{B_{j-1}C_{j+1}},\overrightarrow{B_{j-1}C_{j}}\big)
+\ldots
+ \angle\big(\overrightarrow{B_{i+1}C_{i+3}},\overrightarrow{B_{i+1}C_{i+2}}\big).
\]
Using
\[
\angle\big(\overrightarrow{B_m(\phi_k)C_{m+2}(\phi_k)},\overrightarrow{B_{m}(\phi_k)C_{m+1}(\phi_k)}\big)=\gamma_m([\phi_k])
\]
and
\[
\angle\big(\overrightarrow{B_{m-1}(\phi_k)C_{m+1}(\phi_k)},\overrightarrow{B_{m}(\phi_k)C_{m+1}(\phi_k)}\big)=\pi-\frac{\alpha_{m+1}}{2},
\]
and recalling that
\[
\gamma_m([\phi_0])=\pi-\frac{\alpha_{m+2}}{2}, \quad m=i+1,\ldots,j-1,
\]
we conclude
\begin{align*}
\angle\big(\overrightarrow{B_j(\phi_k)C_{j+2}(\phi_k)},\overrightarrow{B_{i+1}(\phi_k)C_{i+2}(\phi_k)}\big)=\sigma_j([\phi_k])-\sigma_{i}([\phi_k])-\gamma_{j-1}([\phi_0])-\ldots-\gamma_{i+1}([\phi_0]).
\end{align*}
Because of \eqref{eq:01} we conclude that $\sigma_j([\phi_k])-\sigma_{i}([\phi_k])$ converges to $\gamma_j([\phi_0])+\ldots+\gamma_{i+1}([\phi_0])=\sigma_j([\phi_0])-\sigma_{i}([\phi_0])$.
\end{proof}

\subsection{Equivariance property} We prove

\begin{prop}\label{prop:equivariance}
The map $\Cgot \colon \dtrelcharvar\to \CP^{n-3}$ is equivariant with respect to the torus actions \eqref{eq:definiton-torus-action-on-representations} and \eqref{eq:definition-torus-action-on-CPn}. Moreover,
\[
\mu=\nu\circ\Cgot,
\]
where $\mu$ and $\nu$ are the moment maps defined in \eqref{eq:definiton-moment-map} and \eqref{eq:definition-moment-map-torus-action-on-CPn}.
\end{prop}
\begin{proof}
Both torus actions and both moment maps are continuous. The map $\Cgot$ is continuous by Lemma ~\ref{lem:continuity}. It thus suffices to check the conclusion of the proposition on the dense open subset given by the regular fibres of the moment map $\mu$. Let $[\phi]$ be an element in a regular fibre and let $\theta\in\T^{n-3}$. The relations \eqref{eq:image-of-C_i-torus-action} and \eqref{eq:image-of-B_i-torus-action} (see also Figure ~\ref{fig:torus-action-chain-of-triangles}) show that, for any $i=0,\ldots,n-3$ and $j=1,\ldots,n-3$,
\[
a_i(\theta\cdot[\phi])=a_i([\phi])\quad \text{and}\quad \gamma_j(\theta\cdot[\phi])=\gamma_j([\phi])-\overline{\theta}_j.
\]
Hence $\sigma_j(\theta\cdot [\phi])=\sigma_j([\phi])-\theta_j$. This implies $\Cgot(\theta\cdot[\phi])=\theta\cdot\Cgot([\phi])$. Observe further that, for every $i=1,\ldots,n-3$, it holds that
\[
\nu_i\circ\Cgot([\phi])=\frac{a_i([\phi])}{2\lambda}=\mu_i([\phi]),
\]
where we used that the sum of the area parameters is equal to $\lambda$.
\end{proof}

\subsection{Differentiablity property} In this section, we prove that 

\begin{prop}\label{prop:differentiability}
The map $\Cgot \colon \dtrelcharvar\to \CP^{n-3}$ is continuously differentiable.
\end{prop}

The map $\Cgot$ restricted to $\regdtrelcharvar$ is analytic because both the area and angle parameters are analytic functions of $\regdtrelcharvar$. As mentioned earlier, two factors lead to complications when trying to prove differentiability on the complement of $\regdtrelcharvar$. The first one is the presence of square roots on the area parameters. The second one is the discontinuity of the angle parameters whenever triangles are degenerate. 

The proof that $\Cgot$ is a continuous function (Lemma ~\ref{lem:continuity}) showed the importance of considering consecutive indices for which the corresponding area parameters vanish. This leads to the notion of \emph{chain of degeneracy} for $[\phi]\in\dtrelcharvar$ by which we mean a \emph{maximal} collection of consecutive degenerate triangles in the chain built from $\mathfrak P([\phi])$. A chain of degeneracy is said to be \emph{of type $(j,k)$} if the maximal collection of consecutive degenerate triangles is $\Delta_j,\ldots,\Delta_{j+k-1}$. The number $k$ is the length of the chain. The maximality assumption means that the triangles $\Delta_{j-1}$ and $\Delta_{j+k}$, if they exist, are non-degenerate.

To conclude the proof of Proposition ~\ref{prop:differentiability} it remains to check that $\Cgot$ is continuously differentiable at every $[\phi]$ with at least one chain of degeneracy. For simplicity, we only cover the case where $a_{n-3}([\phi])>0$. The case $a_{n-3}([\phi])=0$ can be treated in similar manner.

Let $[\phi_0]\in\dtrelcharvar$ such that $a_{n-3}([\phi_0])>0$. Assume that $[\phi_0]$ has exactly $d\geq 1$ chains of degeneracy of types $(j_1,k_1)$, $\ldots$, $(j_d,k_d)$ with $j_1<\ldots<j_d$. This means that $a_{j_l+k_l}([\phi_0])>0$ for every $l=1,\ldots,d$ (the case $l=d$ follows from the assumption $a_{n-3}([\phi_0])>0$). This implies that the angle parameters $\gamma_i$ are analytic in a neighbourhood of $[\phi_0]$ for every index $i$ in the complement of
\[
\{j_1,\ldots,j_1+k_1\}\cup\ldots\cup \{j_{d},\ldots,j_d+k_d\}.
\]
More precautions must be taken to deal with the case where $j_1=0$, i.e.\ when $a_0([\phi_0])=0$. To prove that $\Cgot$ is continuously differentiable at $[\phi_0]$ we claim that it is sufficient to prove
\begin{lem}\label{lem:key-lemma-differentiability}
The following functions are continuously differentiable in a neighbourhood of $[\phi_0]$:
\begin{enumerate}
\item $[\phi]\mapsto\exp\big(i\cdot (\gamma_{j_1}([\phi])+\ldots+\gamma_{j_1+k_1}([\phi]))\big)$ if $j_1\neq 0$,
\item $[\phi]\mapsto\exp\big(i\cdot (\gamma_{j_l}([\phi])+\ldots+\gamma_{j_l+k_l}([\phi]))\big)$ for every $l=2,\ldots,d$,
\item $[\phi]\mapsto \sqrt{a_{i}([\phi])}\exp\big(i\cdot (-\gamma_{i+1}([\phi])-\ldots-\gamma_{j_l+k_l}([\phi]))\big)$ for every $i=j_l,\ldots,j_l+k_l-1$ and $l=1,\ldots,d$.
\end{enumerate}
\end{lem}

We now explain how Proposition ~\ref{prop:differentiability} follows from Lemma ~\ref{lem:key-lemma-differentiability}.

\begin{proof}[Proof of Proposition ~\ref{prop:differentiability}]
We prove that $\Cgot$ is continuously differentiable at $[\phi_0]$. First assume $j_1\neq 0$. The first two statements of Lemma ~\ref{lem:key-lemma-differentiability}, together with the observation made just before stating Lemma ~\ref{lem:key-lemma-differentiability}, imply that $\exp\big(i\cdot\sigma_i([\phi])\big)$ is continuously differentiable in a neighbourhood of $[\phi_0]$ for every index $i$ in the complement of
\[
\{j_1,\ldots,j_1+k_1-1\}\cup\ldots\cup \{j_{d},\ldots,j_d+k_d-1\}.
\]
These are precisely the indices $i$ for which $a_i([\phi_0])>0$. Denote the collection of these indices $\mathcal I_{reg}$. If $j_1=0$, then we may only conclude that $\exp\big(i\cdot(\sigma_i([\phi])-\sigma_{j_1+k_1}([\phi]))\big)$  is continuously differentiable in a neighbourhood of $[\phi_0]$ for every index $i$ in $\mathcal I_{reg}$. So, in both cases we know that 
\begin{equation}\label{eq:564}
\exp\big(i\cdot(\sigma_i([\phi])-\sigma_{j_1+k_1}([\phi]))\big)
\end{equation}
is continuously differentiable in a neighbourhood of $[\phi_0]$ for every index $i$ in $\mathcal I_{reg}$.

Recall that if $a_i([\phi_0])>0$, then $\sqrt{a_i([\phi])}$ is differentiable in a neighbourhood of $[\phi_0]$. We decide to work in the chart $\{z_{j_1+k_1}\neq 0\}$ of $\CP^{n-3}$. So, proving that $\Cgot$ is continuously differentiable at $[\phi_0]$ amounts to prove that all the functions
\begin{equation}\label{eq:643}
\sqrt{a_i([\phi])}\exp\big(i\cdot (\sigma_i([\phi])-\sigma_{j_1+k_1}([\phi]))\big)
\end{equation}
are continuously differentiable in a neighbourhood of $[\phi_0]$ for every $i\neq j_1+k_1$. This is immediate for $i\in \mathcal I_{reg}$. For all the indices $i$ such that $a_i([\phi_0])=0$, we proceed as follows. Recall from \eqref{eq:564} that the functions $\exp\big(i\cdot(\sigma_i([\phi])-\sigma_{j_1+k_1}([\phi]))\big)$ are continuously differentiable for $i=j_l+k_l$ with $l=2,\ldots,d$. So, proving that the functions of the type \eqref{eq:643} are continuously differentiable for $i\notin \mathcal I_{reg}$ is equivalent to proving that all the functions 
\[
\sqrt{a_i([\phi])}\exp\big(i\cdot (\sigma_i([\phi])-\sigma_{j_l+k_l}([\phi]))\big)
\]
are differentiable in a neighbourhood of $[\phi_0]$ for all $i=j_l,\ldots,j_l+k_l-1$ and all $l=1,\ldots,d$. This is exactly the third statement of Lemma ~\ref{lem:key-lemma-differentiability}.
\end{proof}

The rest of this section is devoted to proving Lemma ~\ref{lem:key-lemma-differentiability}. The idea is to express the area and angle parameters as functions of the coordinates of the points $C_i=x_{C_i}+i\cdot y_{C_i}$ and $B_i=x_{B_i}+i\cdot y_{B_i}$. We start with the area parameters. 

\begin{lem}\label{lem:formula-area-parameter}
Let $[\phi]\in \dtrelcharvar$. For any $i=0,\ldots,n-3$, we have
\[
a_i([\phi])=4\arcsin\left(\frac{\sin\left(\frac{\alpha_{i+2}}{2}\right)\sin\left(\frac{\beta_i}{2}\right)}{4\sin\left(\frac{\alpha_{i+2}+2\pi-\beta_{i+1}-\beta_{i}}{4}\right)}\cdot y_{C_{i+2}}^{-1}y_{B_i}^{-1}\left((x_{C_{i+2}}-x_{B_i})^2+(y_{C_{i+2}}-y_{B_i})^2\right)\right),
\]
where we abbreviated $\beta_i=\beta_i([\phi])$, $\beta_{i+1}=\beta_{i+1}([\phi])$, $C_{i+2}=C_{i+2}(\phi)$ and $B_i=B_i(\phi)$.
\end{lem}
\begin{proof}
The formula is true if the triangle $\Delta_i$ is degenerate because then $B_i=C_{i+2}$. Recall that the hyperbolic distance $d(C_{i+2},B_i)$ in the upper half-plane is given by
\begin{equation}\label{eq:distance-upper-half-plane}
\cosh (d(C_{i+2},B_i))=1+\frac{(x_{C_{i+2}}-x_{B_i})^2+(y_{C_{i+2}}-y_{B_i})^2}{2y_{C_{i+2}}y_{B_i}}.
\end{equation}
The hyperbolic law of cosines applied to the triangle $\Delta_i=\Delta(B_i,C_{i+2},B_{i+1})$ gives
\[
\cos\left(\pi-\frac{\beta_{i+1}}{2}\right)=-\cos\left(\pi-\frac{\alpha_{i+2}}{2}\right)\cos\left(\frac{\beta_i}{2}\right)+\sin\left(\pi-\frac{\alpha_{i+2}}{2}\right)\sin\left(\frac{\beta_i}{2}\right)\cosh(d(C_{i+2},B_i)).
\]
For geometric reasons, it makes sense to keep using $2\pi-\beta_{i+1}$ and not simplify the corresponding trigonometric terms. Using the angle sum identity for the cosine, this can be rewritten as
\begin{align}
\cos\left(\pi-\frac{\beta_{i+1}}{2}\right)&=\cos\left(\frac{\alpha_{i+2}}{2}\right)\cos\left(\frac{\beta_i}{2}\right)+\sin\left(\frac{\alpha_{i+2}}{2}\right)\sin\left(\frac{\beta_i}{2}\right)\cosh(d(C_{i+2},B_i))\nonumber \\ 
&=\cos\left(\frac{\alpha_{i+2}-\beta_i}{2}\right)+\sin\left(\frac{\alpha_{i+2}}{2}\right)\sin\left(\frac{\beta_i}{2}\right)(\cosh(d(C_{i+2},B_i))-1). \label{eq:random-equation-avant-cos-cos}
\end{align}
We use the trigonometric identity $\cos(x)+\cos(y)=-2\sin((x-y)/2)\sin((x+y)/2)$  to write
\[
\cos\left(\frac{2\pi-\beta_{i+1}}{2}\right)-\cos\left(\frac{\alpha_{i+2}-\beta_i}{2}\right)=-2\sin\left(\frac{2\pi-\beta_{i+1}-\alpha_{i+2}+\beta_i}{4}\right)\sin\left(\frac{2\pi-\beta_{i+1}+\alpha_{i+2}-\beta_i}{4}\right).
\]
Using \eqref{eq:area-parameters-in-terms-of-angles} we obtain
\begin{equation}\label{eq:random-equation-cos-cos}
\cos\left(\frac{2\pi-\beta_{i+1}}{2}\right)-\cos\left(\frac{\alpha_{i+2}-\beta_i}{2}\right)=2\sin\left(\frac{a_i([\phi])}{4}\right)\sin\left(\frac{\alpha_{i+2}+2\pi-\beta_i-\beta_{i+1}}{4}\right).
\end{equation}
The conclusion follows from \eqref{eq:distance-upper-half-plane}, \eqref{eq:random-equation-avant-cos-cos} and \eqref{eq:random-equation-cos-cos}.
\end{proof}

The formula of Lemma ~\ref{lem:formula-area-parameter} for the area parameters is relevant for the following reasons. Recall that the ranges of the functions $\beta_i$ over $\dtrelcharvar$ are compact subsets of $(0,2\pi)$ explicitly written down in \eqref{eq:range-beta_i}. So, the range of the ratio 
\[
\frac{\sin\left(\frac{\alpha_{i+2}}{2}\right)\sin\left(\frac{\beta_i}{2}\right)}{4\sin\left(\frac{\alpha_{i+2}+2\pi-\beta_i-\beta_{i+1}}{4}\right)}
\]
over $\dtrelcharvar$ is a compact interval inside the positive real numbers. On the other hand, the expression $$y_{C_{i+2}}^{-1}y_{B_i}^{-1}\left((x_{C_{i+2}}-x_{B_i})^2+(y_{C_{i+2}}-y_{B_i})^2\right)$$ is zero whenever the triangle $\Delta_i$ is degenerate. This means that the function 
\[
\frac{a_i}{y_{C_{i+2}}^{-1}y_{B_i}^{-1}\left((x_{C_{i+2}}-x_{B_i})^2+(y_{C_{i+2}}-y_{B_i})^2\right)}
\]
extends analytically to any $[\phi_0]$ such that $a_i([\phi_0])=0$. Moreover, its value at $[\phi_0]$ is the positive number
\[
\frac{\sin\left(\frac{\alpha_{i+2}}{2}\right)\sin\left(\frac{\beta_i([\phi_0])}{2}\right)}{\sin\left(\frac{\beta_{i+1}([\phi_0])}{2}\right)},
\]
which remains uniformly bounded away from zero by the above remark for every such $[\phi_0]$. We conclude that the function 
\begin{equation}\label{eq:square-root-area-parameter-analytic}
\sqrt{\frac{a_i}{y_{C_{i+2}}^{-1}y_{B_i}^{-1}\left((x_{C_{i+2}}-x_{B_i})^2+(y_{C_{i+2}}-y_{B_i})^2\right)}}
\end{equation}
also extends analytically to any $[\phi_0]$ such that $a_i([\phi_0])=0$. We proved
\begin{lem}\label{lem:sqrt-a_i-divided-by-distance-analytic}
The function defined by \eqref{eq:square-root-area-parameter-analytic} on the subspace of $\dtrelcharvar$ of all $[\phi]$ for which $a_i([\phi])>0$ extends analytically to $\dtrelcharvar$.
\end{lem}

We now proceed with a computation of the angle parameters. We start by introducing the function $\Gamma\colon \Hyp\setminus\{i\}\to \R/2\pi\Z$ defined as
\[
\Gamma(x+iy)\coloneqq\left\{\begin{array}{ll}
0,& x=0\text{ and } y>1,\\
\pi, & x=0\text{ and } y<1,\\
3\pi/2, & x^2+y^2=1\text{ and } x>0,\\
\pi/2, & x^2+y^2=1\text{ and } x<0,\\
\pi-\arctan\left(\frac{2x}{x^2+y^2-1}\right), & x^2+y^2<1,\\
-\arctan\left(\frac{2x}{x^2+y^2-1}\right),& x^2+y^2>1.
\end{array}
\right.
\]
The different domains involved in the definition of $\Gamma$ are illustrated on Figure ~\ref{fig:domain-Gamma}. 
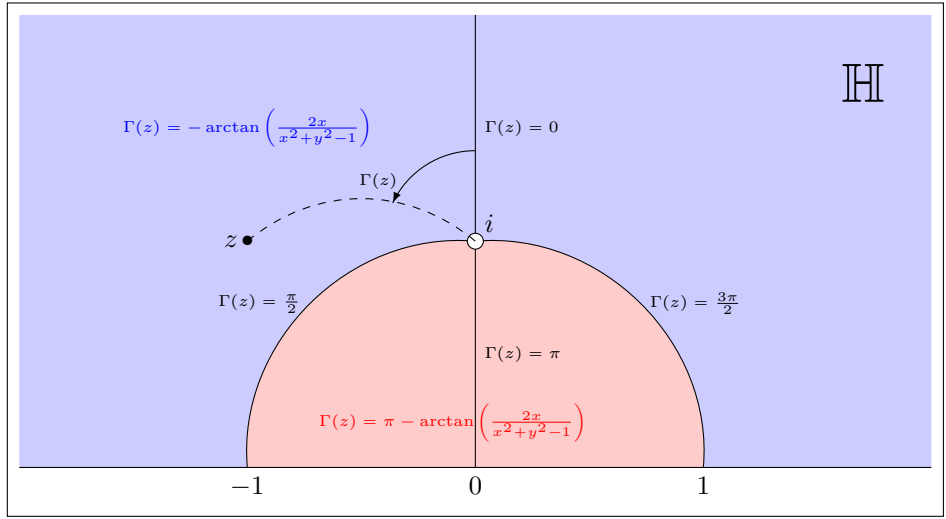
\begin{figure}[h]
\begin{center}
\begin{tikzpicture}[framed, scale=3,font=\sffamily,decoration={
    markings,
    mark=at position 1 with {\arrow{>}}}]]
    
\fill[blue!20] (0,1) to[bend left=50] (1,0)
		 to (2,0)
		 to (2,2) 
		 to  (0,2)
		 to  cycle;

\fill[blue!20] (-1,0) to[bend left=50] (0,1)
		 to (0,2)
		 to (-2,2) 
		 to  (-2,0)
		 to  cycle;

\fill[red!20] (-1,0) to[bend left=50] (0,1)
		 to (0,0)
		 to  cycle;
		 
\fill[red!20] (0,1) to[bend left=50] (1,0)
		 to (0,0)
		 to  cycle;
		 
\draw (0,1) to[bend left=50] node[right]{\tiny $\Gamma(z)=\frac{3\pi}{2}$} (1,0);
\draw (-1,0) to[bend left=50] node[left]{\tiny $\Gamma(z)=\frac{\pi}{2}$} (0,1);

\draw (0,0) to node[right]{\tiny $\Gamma(z)=\pi$} (0,1);
\draw (0,1) to node[right]{\tiny $\Gamma(z)=0$} (0,2);

\draw (-1,1.5) node[blue]{\tiny $\Gamma(z)=-\arctan\left(\frac{2x}{x^2+y^2-1}\right)$};

\draw (-.1,.2) node[red]{\tiny $\Gamma(z)=\pi-\arctan\left(\frac{2x}{x^2+y^2-1}\right)$};

\draw[fill=white] (0,1) circle (1pt);

\draw (-2,0) to (2,0);

\draw (0,1) node[above right]{$i$};
\draw (1,0) node[below]{$1$};
\draw (-1,0) node[below]{$-1$};
\draw (0,0) node[below]{$0$};
\draw (1.7,1.7) node{\Huge $\Hyp$};

\draw(-1,1) node{$\bullet$} node[left]{$z$};
\draw[dashed] (0,1) to[bend right=40] (-1,1);
\draw[postaction={decorate}] (0,1.4) arc (90:155:.4) node[near end, left]{\tiny $\Gamma(z)$};

\end{tikzpicture}
\end{center}
\caption{Illustration of the different domains involved in the definition of the function $\Gamma$ and the value of $\Gamma$ in each of these regions.}\label{fig:domain-Gamma}
\end{figure}

The function $\Gamma$ has a geometric interpretation. It measures the oriented angle between the vertical geodesic ray leaving from $i$ and the geodesic ray leaving from $i$ and going through $x+iy$. This can easily be seen after noticing that the ratio
\[
\frac{x^2+y^2-1}{2x}
\]
is the point on the boundary of the upper half-plane which is the center of the semi-circle supporting the geodesic through $i$ and $x+iy$.

\begin{lem}\label{lem:Gamma-C1}
The function $\Gamma\colon \Hyp\setminus\{i\}\to \R/2\pi\Z$ is continuously differentiable.
\end{lem}
\begin{proof}
We refer to Figure ~\ref{fig:domain-Gamma}. The function $\Gamma$ is continuously differentiable in the blue and red regions. These regions are open subdomains of $\Hyp$. If one carefully studies the limit behaviour of $\Gamma$ at the boundary of the blue and red regions, one sees that $\Gamma$ is a continuous function. The partial derivatives inside the blue and red regions are
\[
\frac{\partial}{\partial x}\left(-\arctan\left(\frac{2x}{x^2+y^2-1}\right)\right)=\frac{2(x^2-y^2+1)}{4x^2+(x^2+y^2-1)^2}
\]
and
\[
\frac{\partial}{\partial y}\left(-\arctan\left(\frac{2x}{x^2+y^2-1}\right)\right)=\frac{4xy}{4x^2+(x^2+y^2-1)^2}.
\]
These partial derivatives extend continuously to $\Hyp\setminus\{i\}$. We conclude that $\Gamma$ is continuously differentiable.
\end{proof}

\begin{lem}\label{lem:exponential-of-function-Gamma}
It holds
\[
\exp(i\cdot\Gamma(x+iy))=\frac{x^2+y^2-1-i\cdot 2x}{\sqrt{4x^2+(x^2+y^2-1)^2}}.
\]
\end{lem}
\begin{proof}
By the definition of $\Gamma$, it follows that
\[
\exp(i\cdot\Gamma(x+iy))=\left\{\begin{array}{ll}
-\exp\left(i\arctan\left(\frac{-2x}{x^2+y^2-1}\right)\right), & x^2+y^2<1,\\
\exp\left(i\arctan\left(\frac{-2x}{x^2+y^2-1}\right)\right),& x^2+y^2>1.
\end{array}
\right.
\]
We use the identity
\[
e^{i\cdot\arctan(x)}=\frac{1+ix}{\sqrt{x^2+1}}.
\]
Observe that
\[
\sqrt{\frac{4x^2}{(x^2+y^2-1)^2}+1}=\frac{\sqrt{4x^2+(x^2+y^2-1)^2}}{\vert x^2+y^2-1\vert}.
\] 
Hence
\begin{align*}
\exp(i\cdot\Gamma(x+iy))&=\frac{x^2+y^2-1}{\sqrt{4x^2+(x^2+y^2-1)^2}}+i\cdot \frac{-2x}{x^2+y^2-1}\cdot\frac{x^2+y^2-1}{\sqrt{4x^2+(x^2+y^2-1)^2}}\\
&= \frac{x^2+y^2-1-i\cdot 2x}{\sqrt{4x^2+(x^2+y^2-1)^2}}.
\end{align*}
\end{proof}

Let $p=x_p+iy_p$ be a point in $\Hyp$. We introduce the function $\Gamma_p\colon \Hyp\setminus\{p\}\to \R/2\pi\Z$ defined by
\[
\Gamma_p(z)\coloneqq\Gamma(y_p^{-1}(z-x_p)).
\]
Note that the function $\Gamma_p$ is defined to be the composition of the function $\Gamma$ with the orientation-preserving isometry
\[
\pm y_p^{1/2} \begin{pmatrix}
y_p^{-1} & -x_py_p^{-1}\\
0 & 1
\end{pmatrix}
\]
of the upper half-plane that sends $p$ to $i$. This isometry sends vertical geodesics to vertical geodesics. In other words, $\Gamma_p$ measures the oriented angle between the vertical geodesic ray leaving from $p$ and the geodesic ray leaving from $p$ and going through $z$. The analogue of Lemma ~\ref{lem:exponential-of-function-Gamma} for the function $\Gamma_p$ reads
\begin{equation}\label{eq:exponential-of-function-Gamma_p}
\exp(i\cdot\Gamma_p(x+iy))=\frac{(x-x_p)^2+y^2-y_p^2-i\cdot 2y_p(x-x_p)}{\sqrt{4y_p^2(x-x_p)^2+((x-x_p)^2+y^2-y_p^2)^2}}.
\end{equation}

\begin{lem}\label{lem:extension-exponential-Gamma-Gamma}
The function that maps a pair of distinct points $(p,z)$ in $\Hyp\times \Hyp$ to
\[
\exp \big(i\cdot(\Gamma_z(p)-\Gamma_p(z))\big)
\]
extends to a continuously differentiable function of $\Hyp\times \Hyp$.
\end{lem}
\begin{proof}
Let $p=x_p+iy_p$ and $z=x_z+iy_z$. We use \eqref{eq:exponential-of-function-Gamma_p} to compute, with the help of Wolfram Mathematica\footnote{version 12.2.0.0},
\begin{align*}
\frac{\exp \big(i\cdot\Gamma_z(p)\big)}{\exp \big(i\cdot\Gamma_p(z)\big)}&=\frac{(x_p-x_z)^2+y_p^2-y_z^2-i\cdot 2y_z(x_p-x_z)}{\sqrt{4y_z^2(x_p-x_z)^2+((x_p-x_z)^2+y_p^2-y_z^2)^2}}\cdot \frac{\sqrt{4y_p^2(x_z-x_p)^2+((x_z-x_p)^2+y_z^2-y_p^2)^2}}{(x_z-x_p)^2+y_z^2-y_p^2-i\cdot 2y_p(x_z-x_p)}\\
&=\frac{(x_p-x_z)-i(y_p+y_z)}{(x_p-x_z)+i(y_p+y_z)}.
\end{align*}
The last expression is a continuously differentiable function of $\Hyp\times \Hyp$.
\end{proof}

The relation between the function $\Gamma$ and the angle parameters is immediate. Let $[\phi]\in\dtrelcharvar$ be such that $a_i([\phi])>0$ and $a_{i-1}([\phi])>0$. Let $\ell_i(\phi)$ be the vertical geodesic ray leaving from $B_i(\phi)$. Using the definition of $\gamma_i$ we obtain
\begin{align}
\gamma_i([\phi])&=\angle (\overrightarrow{B_i(\phi)C_{i+2}(\phi)},\overrightarrow{B_i(\phi)C_{i+1}(\phi)})\nonumber \\
&= \angle (\ell_i(\phi),\overrightarrow{B_i(\phi)C_{i+1}(\phi)})-\angle (\ell_i(\phi),\overrightarrow{B_i(\phi)C_{i+2}(\phi)})\nonumber \\
&=\Gamma_{B_i(\phi)}(C_{i+1}(\phi))-\Gamma_{B_i(\phi)}(C_{i+2}(\phi)).\label{eq:gamma_i-as-Gamma-Gamma-with-B_i-C_i}
\end{align}
The second conclusion of Corollary ~\ref{cor:description-chain-triangles-alpha-non-degenerate} says that
\[
\angle (\overrightarrow{B_i(\phi)B_{i+1}(\phi)},\overrightarrow{B_i(\phi)C_{i+2}(\phi)})+\angle (\overrightarrow{B_i(\phi)C_{i+1}(\phi)},\overrightarrow{B_i(\phi)B_{i-1}(\phi)})=\beta_i/2+\pi-\beta_i/2=\pi.
\]
This implies
\begin{align}
\gamma_i([\phi])&=\pi-\angle (\overrightarrow{B_i(\phi)B_{i-1}(\phi)},\overrightarrow{B_i(\phi)B_{i+1}(\phi)})\nonumber \\
&= \pi-\big(\angle (\ell_i(\phi),\overrightarrow{B_i(\phi)B_{i+1}(\phi)})-\angle (\ell_i(\phi),\overrightarrow{B_i(\phi)B_{i-1}(\phi)})\big)\nonumber \\
&=\pi-\big(\Gamma_{B_i(\phi)}(B_{i+1}(\phi))-\Gamma_{B_i(\phi)}(B_{i-1}(\phi))\big).\label{eq:gamma_i-as-Gamma-Gamma-with-B_i-B_i}
\end{align}

\begin{lem}\label{lem:sum-gamma_i-over-chain-differentiable}
Let $j< k$ and $[\phi_0]\in\dtrelcharvar$ be such that $a_{j-1}([\phi_0])>0$, $a_{j+k}([\phi_0])>0$ and $a_i([\phi_0])=0$ for every $l=j,\ldots,j+k-1$. Then the function 
\[
\exp\big(i\cdot(\gamma_j([\phi])+\ldots+\gamma_{j+k}([\phi])\big)
\]
is continuously differentiable in a neighbourhood of $[\phi_0]$.
\end{lem}
\begin{proof}
We know from the proof of Lemma ~\ref{lem:continuity} that the function is continuous. Using \eqref{eq:gamma_i-as-Gamma-Gamma-with-B_i-B_i} we write (dropping the dependence on $\phi$)
\begin{align*}
\gamma_j+\ldots+\gamma_{j+k}&=(k+1)\pi -\sum_{i=j}^{j+k} \big(\Gamma_{B_i}(B_{i+1})-\Gamma_{B_i}(B_{i-1})\big)\\
&= (k+1)\pi +\Gamma_{B_j}(B_{j-1})-\Gamma_{B_{j+k}}(B_{j+k+1})-\sum_{i=j}^{j+k-1}\big(\Gamma_{B_i}(B_{i+1})-\Gamma_{B_{i+1}}(B_{i})\big).
\end{align*}
In a neighbourhood of $[\phi_0]$ we may assume that $a_{j-1}$ and $a_{j+k}$ are nonzero, so that $B_j\neq B_{j-1}$ and $B_{j+k}\neq B_{j+k+1}$. This means that the functions $\Gamma_{B_j}(B_{j-1})$ and $\Gamma_{B_{j+k}}(B_{j+k+1})$ are continuously differentiable around $[\phi_0]$. Lemma ~\ref{lem:extension-exponential-Gamma-Gamma} implies that every summand in the remaining sum extends to a continuously differentiable function around $[\phi_0]$. This concludes the proof of the lemma.
\end{proof}

The next step consists in expressing $\exp\big(i\cdot \Gamma_{B_{i+1}}(C_{i+2})\big)$ in terms of the coordinates of the points $B_{i}$ and $C_{i+1}$. We first need to compute the coordinates of the point $B_{i+1}$ in terms of that of $B_i$ and $C_{i+2}$. Recall that $B_{i+1}$ is the fixed point of $\phi(c_{i+2})^{-1}\phi(b_i)$. Using Proposition ~\ref{prop:elliptic-matrix-from-fixed-point-and-angle} we can write down explicitly $\phi(c_{i+2})$ and $\phi(b_i)$ in terms of the coordinates of $C_{i+2}$ and $B_i$, and the angles $\alpha_{i+2}$ and $\beta_i$. We can then compute the product $\phi(c_{i+2})^{-1}\phi(b_i)$ and deduce the coordinates of $B_{i+1}$ using formula \eqref{eq:fixed-point-elliptic}. With the help of Wolfram Mathematica, we obtain
\begin{equation*}
\resizebox{.9\hsize}{!}
{$x_{B_{i+1}}=\frac{2\left(x_{C_{i+2}}y_{B_i}\cos\left(\frac{2\pi-\beta_i}{2}\right)\sin\left(\frac{\alpha_{i+2}}{2}\right)+x_{B_i}y_{C_{i+2}}\cos\left(\frac{\alpha_{i+2}}{2}\right)\sin\left(\frac{2\pi-\beta_i}{2}\right)\right)+\sin\left(\frac{2\pi-\beta_i}{2}\right)\sin\left(\frac{\alpha_{i+2}}{2}\right)(x_{B_i}^2-x_{C_{i+2}}^2+y_{B_i}^2-y_{C_{i+2}}^2)}{2\left(y_{B_i}\cos\left(\frac{2\pi-\beta_i}{2}\right)\sin\left(\frac{\alpha_{i+2}}{2}\right)+y_{C_{i+2}}\cos\left(\frac{\alpha_{i+2}}{2}\right)\sin\left(\frac{2\pi-\beta_i}{2}\right)+\sin\left(\frac{2\pi-\beta_i}{2}\right)\sin\left(\frac{\alpha_{i+2}}{2}\right)(x_{B_i}-x_{C_{i+2}})\right)}$}
\end{equation*}
and
\begin{equation*}
\resizebox{.9\hsize}{!}
{$
x_{C_{i+2}}-x_{B_{i+1}}=\frac{-\sin\left(\frac{2\pi-\beta_i}{2}\right)\left(2y_{C_{i+2}}(x_{B_i}-x_{C_{i+2}})\cos\left(\frac{\alpha_{i+2}}{2}\right)+\sin\left(\frac{\alpha_{i+2}}{2}\right)\left((x_{B_i}-x_{C_{i+2}})^2+y_{B_i}^2-y_{C_{i+2}}^2\right)\right)}{2\left(y_{B_i}\cos\left(\frac{2\pi-\beta_i}{2}\right)\sin\left(\frac{\alpha_{i+2}}{2}\right)+y_{C_{i+2}}\cos\left(\frac{\alpha_{i+2}}{2}\right)\sin\left(\frac{2\pi-\beta_i}{2}\right)+\sin\left(\frac{2\pi-\beta_i}{2}\right)\sin\left(\frac{\alpha_{i+2}}{2}\right)(x_{B_i}-x_{C_{i+2}})\right)}.
$}
\end{equation*}
We also have
\begin{equation*}
\resizebox{.9\hsize}{!}
{$
y_{C_{i+2}}^2-y_{B_{i+1}}^2=y_{C_{i+2}}^2-\frac{4y_{B_i}^2y_{C_{i+2}}^2-\left(-2y_{B_i}y_{C_{i+2}}\cos\left(\frac{2\pi-\beta_i}{2}\right)\cos\left(\frac{\alpha_{i+2}}{2}\right)+\sin\left(\frac{2\pi-\beta_i}{2}\right)\sin\left(\frac{\alpha_{i+2}}{2}\right)\left((x_{B_i}-x_{C_{i+2}})^2+y_{B_i}^2+y_{C_{i+2}}^2\right)\right)^2}{4\left(y_{B_i}\cos\left(\frac{2\pi-\beta_i}{2}\right)\sin\left(\frac{\alpha_{i+2}}{2}\right)+y_{C_{i+2}}\cos\left(\frac{\alpha_{i+2}}{2}\right)\sin\left(\frac{2\pi-\beta_i}{2}\right)+\sin\left(\frac{2\pi-\beta_i}{2}\right)\sin\left(\frac{\alpha_{i+2}}{2}\right)(x_{B_i}-x_{C_{i+2}})\right)^2}.
$}
\end{equation*}
We apply \eqref{eq:exponential-of-function-Gamma_p} to get
\begin{equation*}
\resizebox{.9\hsize}{!}
{$
\begin{array}{ll}
e^{i\cdot \Gamma_{B_{i+1}}(C_{i+2})}=&\frac{(x_{C_{i+2}}-x_{B_{i+1}})^2+y_{C_{i+2}}^2-y_{B_{i+1}}^2-i\cdot 2y_{B_{i+1}}(x_{C_{i+2}}-x_{B_{i+1}})}{\sqrt{\left((x_{B_i}-x_{C_{i+2}})^2+(y_{B_i}-y_{C_{i+2}})^2\right)\left((x_{B_i}-x_{C_{i+2}})^2+(y_{B_i}+y_{C_{i+2}})^2\right)}} \\
&\cdot \frac{y_{B_i}\cos\left(\frac{2\pi-\beta_i}{2}\right)\sin\left(\frac{\alpha_{i+2}}{2}\right)+y_{C_{i+1}}\cos\left(\frac{\alpha_{i+2}}{2}\right)\sin\left(\frac{2\pi-\beta_i}{2}\right)+\sin\left(\frac{2\pi-\beta_i}{2}\right)\sin\left(\frac{\alpha_{i+2}}{2}\right)(x_{B_i}-x_{C_{i+2}})}{-y_{C_{i+2}}\sin\left(\frac{2\pi-\beta_i}{2}\right)}.
\end{array}
$}
\end{equation*}
The crucial observation is that the irregularity of the function $\exp\big(i\cdot \Gamma_{B_{i+1}}(C_{i+2})\big)$ comes from the presence of the term $\sqrt{(x_{B_i}-x_{C_{i+2}})^2+(y_{B_i}-y_{C_{i+2}})^2}$ in the denominator. If we compare this observation with Lemma ~\ref{lem:sqrt-a_i-divided-by-distance-analytic} we conclude
\begin{lem}\label{lem:Gamma-times-sqrt-area-differentiable}
The function
\[
[\phi]\mapsto \sqrt{a_i([\phi])}\exp\left(i\cdot \Gamma_{B_{i+1}(\phi)}(C_{i+2})(\phi)\right)
\]
is continuously differentiable on $\dtrelcharvar$.
\end{lem}

We are now ready to prove Lemma ~\ref{lem:key-lemma-differentiability}.

\begin{proof}[Proof of Lemma ~\ref{lem:key-lemma-differentiability}]
The first and second assertions of Lemma ~\ref{lem:key-lemma-differentiability} follow from Lemma ~\ref{lem:sum-gamma_i-over-chain-differentiable}. To prove the third assertion, first observe that (as in the proof of Lemma ~\ref{lem:sum-gamma_i-over-chain-differentiable}) $-\gamma_{i+1}([\phi])-\ldots-\gamma_{j_l+k_l}([\phi])$ can be written as
\[
-(j_l+k_l-i)\pi-\Gamma_{B_{i+1}}(B_i)+\Gamma_{B_{j_l+k_l}}(B_{j_l+k_l+1})+\sum_{m=i+1}^{j_l+k_l-1}\left(\Gamma_{B_m}(B_{m+1})-\Gamma_{B_{m+1}}(B_m)\right).
\]
The functions $\exp\big( i\cdot (\Gamma_{B_m}(B_{m+1})-\Gamma_{B_{m+1}}(B_m))\big)$ are continuously differentiable by Lemma ~\ref{lem:extension-exponential-Gamma-Gamma}. Since $a_{j_l+k_l}([\phi_0])>0$ by hypothesis, the points $B_{j_l+k_l}(\phi)$ and $B_{j_l+k_l+1}(\phi)$ are distinct for $[\phi]$ in a neighbourhood of $[\phi_0]$. Thus the function $\exp\big( i\cdot \Gamma_{B_{j_l+k_l}}(B_{j_l+k_l+1})\big)$ is continuously differentiable in the same neighbourhood of $[\phi_0]$. Now observe that
\[
\Gamma_{B_{i+1}}(B_i)=\Gamma_{B_{i+1}}(C_{i+2})+\pi-\beta_{i+1}/2.
\]
For the function $\beta_{i+1}$ is analytic on $\dtrelcharvar$, it is sufficient to prove that 
\[
\sqrt{a_i}\exp\big(i\cdot \Gamma_{B_{i+1}}(C_{i+2})\big)
\]
is continuously differentiable in a neighbourhood of $[\phi_0]$ to conclude the third assertion of Lemma ~\ref{lem:key-lemma-differentiability}. This is precisely the statement of Lemma ~\ref{lem:Gamma-times-sqrt-area-differentiable}.
\end{proof}

This finishes the proof of Proposition ~\ref{prop:differentiability}.

\subsection{Diffeomorphism property} We now prove

\begin{prop}\label{prop:diffeomorphism}
The map $\Cgot \colon \dtrelcharvar\to \CP^{n-3}$ is a diffeomorphism.
\end{prop}

\begin{proof}
Thanks to Proposition ~\ref{prop:homeomorphism} and Proposition ~\ref{prop:differentiability} we know that $\Cgot$ is a continuously differentiable bijection. To prove that $\Cgot$ is a diffeomorphism it is thus sufficient to prove that the differential of $\Cgot$ is regular at every point. 

Let $[\phi_0]\in\dtrelcharvar$ and assume for simplicity that $a_0([\phi_0])>0$. We decide to work in the chart $\{z_0\neq 0\}$ of $\CP^{n-3}$. In a slight abuse of notation, we write $\Cgot=(\Cgot_1,\ldots,\Cgot_{n-3})$ for the map
\[
[\phi]\mapsto \left(\sqrt{\frac{a_1([\phi])}{a_0([\phi])}}e^{i\sigma_1([\phi])},\ldots,\sqrt{\frac{a_{n-3}([\phi])}{a_0([\phi])}}e^{i\sigma_{n-3}([\phi])}\right)\in\C^{n-3}
\]
defined for every $[\phi]\in\dtrelcharvar$ such that $a_0([\phi])>0$. We prove that $(d\Cgot)_{[\phi_0]}$ is surjective. We distinguish two cases according to whether some area parameters of $[\phi_0]$ vanish.

First assume that $a_i([\phi_0])>0$ for every $i$. We consider the decomposition of the tangent space to $\CP^{n-3}$ at $\Cgot([\phi_0])$ as the direct sum of the kernel of the differential of the moment map $\nu$ defined in \eqref{eq:definition-moment-map-torus-action-on-CPn} and its complement $V$:
\[
T_{\Cgot([\phi_0])}\CP^{n-3}=\Ker\left( (d\nu)_{\Cgot([\phi_0])}\right)\oplus V.
\]
Both subspaces have dimension $n-3$ because $\Cgot([\phi_0])$ lies in a regular fibre of $\nu$. Since we are assuming $a_i([\phi_0])>0$ for every $i$, the exterior derivative of $\Cgot$ at $[\phi_0]$ is given in components by
\begin{equation}\label{eq:exterior-derivative-Cgot}
(d\Cgot_i)_{[\phi_0]}=e^{i\sigma_i([\phi_0])}\left(i\cdot (d\sigma_i)_{[\phi_0]}\sqrt{\frac{a_i([\phi_0])}{a_0([\phi_0])}}+\frac{a_0([\phi_0])(da_i)_{[\phi_0]
}-a_i([\phi_0])(da_0)_{[\phi_0]}}{2a_0([\phi_0])\sqrt{a_i([\phi_0])a_0([\phi_0])}}\right).
\end{equation}
Note that since the torus action on $\dtrelcharvar$ is by diffeomorphisms, we can neglect the term $e^{i\sigma_i([\phi_0])}$ appearing in \eqref{eq:exterior-derivative-Cgot}. Said differently, $d\Cgot$ is surjective at $[\phi_0]$ if and only if it is surjective at $[\phi_0']$, where $[\phi_0']$ and $[\phi_0]$ lie in the same fibre of the moment and $\sigma_i([\phi_0'])=0$ for every $i$.  

We consider a first family of smooth deformations of $[\phi_0]$ along the orbits corresponding to a fixed component of the torus action. Assume that we deform along the orbit corresponding to the $i$th component of the torus. Along that deformation all the area parameters $a_j$ are constant and the angle parameters $\sigma_j$ are constant for $j\neq i$. The image under the differential of $\Cgot$ of that deformation is generated by the complex direction of the $i$th component of $\C^{n-3}$ according to \eqref{eq:exterior-derivative-Cgot}. Moreover, the images of each such deformation lie in the kernel of $(d\nu)_{\Cgot([\phi_0])}$ by Proposition ~\ref{prop:equivariance}. Comparing dimensions, we conclude that the image of $(d\Cgot)_{[\phi_0]}$ contains the kernel of $(d\nu)_{\Cgot([\phi_0])}$.

Next we consider a second family of smooth deformations of $[\phi_0]$ corresponding to the complement $W$ of the kernel of the differential of the moment map $\mu$ defined in \eqref{eq:definiton-moment-map}: 
\[
T_{[\phi_0]}\dtrelcharvar=\Ker\left( (d\mu)_{[\phi_0]}\right)\oplus W.
\]
Proposition ~\ref{prop:equivariance} says that $\mu=\nu\circ\Cgot$ which implies $(d\Cgot)_{[\phi_0]}(W)\subset V$. Because both $(d\mu)_{[\phi_0]}$ and $(d\nu)_{\Cgot([\phi_0])}$ have maximal rank, we conclude that $(d\Cgot)_{[\phi_0]}$ maps $W$ surjectively onto $V$. This shows that the image of $(d\Cgot)_{[\phi_0]}$ contains $V$. We conclude that $(d\Cgot)_{[\phi_0]}$ is surjective.

Now, we deal with the case where $a_i([\phi_0])=0$ for some index $i$. The argument here relies on the existence of a smooth deformation $[\phi_t]$ of $[\phi_0]$ such that $a_i([\phi_t])>0$ for $t\neq 0$. The existence of such a deformation is a general property of symplectic toric manifolds (recall that $a_i$ is a multiple of the $i$th component of the moment map $\mu$). Let us abbreviate $a_i(t)\coloneqq a_i([\phi_t])$. Note that, by assumption, $a_i(0)=a_i'(0)=0$. We can choose the deformation $[\phi_t]$ to ensure that $a_i''(0)>0$. Proposition ~\ref{prop:differentiability} implies that the function 
\[
[\phi]\mapsto \sqrt{\frac{a_i([\phi])}{a_0([\phi])}}e^{i\sigma_i([\phi])}
\]
is continuously differentiable. We claim that its derivative at $[\phi_0]$ along the deformation $[\phi_t]$ is nonzero. This is the case because we assumed $a_i''(0)>0$. This means that the image of $(d\Cgot_i)_{[\phi_0]}$ inside $\C$ has real dimension at least one. However, whenever $a_i([\phi_0])=0$, there is an effective action of the $i$th component of the torus on $T_{[\phi_0]}\dtrelcharvar$. Since the torus action is by diffeomorphisms and $\Cgot$ is equivariant, we conclude that $(d\Cgot_i)_{[\phi_0]}$ is surjective. We can repeat this argument for every index $i$ such that $a_i([\phi_0])=0$. Combined with the previous case, this shows that $(d\Cgot)_{[\phi_0]}$ is surjective even when some area parameters vanish.
\end{proof}

\subsection{Symplectomorphism property} Finally, we prove that

\begin{prop}\label{prop:symplectomorphism}
The map $\Cgot \colon \dtrelcharvar\to \CP^{n-3}$ is a symplectomorphism, i.e.\
\[
\lambda\cdot\Cgot^\ast \omega_\FS=\omega_\G.
\]
\end{prop}

To prove Proposition ~\ref{prop:symplectomorphism}, we need the following result about symplectic toric manifolds. The result is folklore; a proof is included for completeness.

\begin{lem}\label{lem:property-rigidity-troic-manifold}
Let $M$ be a compact connected smooth manifold of dimension $2m$. Assume that $M$ is equipped with an effective action of an $m$-dimensional torus. Let $\omega_1$ and $\omega_2$ be two symplectic forms on $M$ for which the torus action is Hamiltonian with respect to the same moment map $\mu\colon M\to \R^m$. Then $\omega_1=\omega_2$.
\end{lem}

\begin{proof}
Let $\mathring M$ denote the preimage of the interior of the moment polytope. This is an open and dense subset of $M$. It is thus sufficient to check that $\omega_1=\omega_2$ on $\mathring M$. The Arnold-Liouville Theorem states (see e.g.\ \cite{Can01}) the existence of angle coordinates $(\varphi_1,\ldots,\varphi_m)$ and $(\psi_1,\ldots,\psi_m)$ defined on $\mathring M$ such that 
\[
\omega_1\restriction_{\mathring M}=\sum_{i=1}^m d\mu_i\wedge d\varphi_i,\quad \omega_2\restriction_{\mathring M}=\sum_{i=1}^m d\mu_i\wedge d\psi_i,
\]
where $\mu=(\mu_1,\ldots,\mu_m)$ is the moment map.

We are assuming that the torus action on $M$ is Hamiltonian with moment map $\mu$ for both $\omega_1$ and $\omega_2$. So, for any $\theta\in \R^m$, if $\Theta$ denotes the vector field on $M$  defined by the infinitesimal action of $\theta$, then $\omega_1(\Theta,\cdot)=d\langle \mu,\theta\rangle=\omega_2(\Theta,\cdot)$. By letting $\theta$ range over the standard basis of $\R^m$, we observe that $d\varphi_i=d\psi_i$ must hold for every $i$. Hence $\omega_1\restriction_{\mathring M}=\omega_2\restriction_{\mathring M}$. This concludes the proof of the lemma.
\end{proof}

\begin{proof}[Proof of Proposition ~\ref{prop:symplectomorphism}]
We want to apply Lemma ~\ref{lem:property-rigidity-troic-manifold} for the torus action \eqref{eq:torus-action} on $\dtrelcharvar$. This action is Hamiltonian with respect to the symplectic form $\omega_\G$ and the moment map $\mu$ defined in \eqref{eq:definiton-moment-map}. Propostition ~\ref{prop:diffeomorphism} says that $\Cgot$ is a diffeomorphism. So, $\lambda\cdot \Cgot^\ast \omega_\FS$ is another symplectic form on $\dtrelcharvar$. Proposition ~\ref{prop:equivariance} implies that the torus action \eqref{eq:torus-action} on $\dtrelcharvar$ is Hamiltonian with respect to the symplectic form $\lambda\cdot \Cgot^\ast \omega_\FS$ and the moment map $\mu$. Hence, by Lemma ~\ref{lem:property-rigidity-troic-manifold}, $\lambda\cdot \Cgot^\ast \omega_\FS=\omega_\G$.
\end{proof}

Propositions ~\ref{prop:equivariance} and ~\ref{prop:symplectomorphism} together prove Theorem ~\ref{thm:symplectomorphism}.

\section{Appendix}
\addcontentsline{toc}{section}{Appendix}
\renewcommand{\thesubsection}{\Alph{subsection}}

\subsection{Conjugacy classes inside $\psl$}\label{appendix-PSL2R}

The Lie group $G=\psl$ is the center-free quotient of the group SL$(2,\R)$ of real $2\times 2$ matrices with determinant $1$ by $\{\pm I\}$, where $I$ is the identity $2\times 2$ matrix. It can be identified with the group of orientation-preserving isometries of the upper half-plane $\Hyp=\{z\in \C: \Im(z)>0\}$. It acts on $\Hyp$ by Möbius transformations
\[
\pm \begin{pmatrix}
a & b\\
c & d
\end{pmatrix}\cdot z\coloneqq \frac{az+b}{cz+d}.
\]
The open subspace of $G$ consisting of elements whose trace in absolute value is smaller than 2 is called the subspace of \emph{elliptic elements} of $G$. It is denoted by $\E\subset G$. Equivalently, an element of $G$ is elliptic if and only if it has a unique fixed point in $\Hyp$.

\begin{lem}\label{lem:c-neq-0-for-elliptic}
If $
A=\pm\begin{pmatrix}
a & b\\
c & d
\end{pmatrix}
$
is elliptic, then $b\neq 0$ and $c\neq 0$.
\end{lem}
\begin{proof}
If $b=0$ or $c=0$, then $\det (A)=ad=1$. So, $\tr (A)^2=(a+d)^2\geq 4ad=4$ and $A$ is not elliptic.
\end{proof}

Let $A=\pm\begin{pmatrix}
a & b\\
c & d
\end{pmatrix}$ be an elliptic element. We denote the unique fixed point of $A$ in $\Hyp$ by $\fix(A)$. It defines a map $\fix\colon \E\to \Hyp$. 

\begin{lem}
The unique fixed point of $A$ is 
\begin{equation}\label{eq:fixed-point-elliptic}
\fix (A)=\frac{a-d}{2c}+i\cdot \frac{\sqrt{4-(a+d)^2}}{2\vert c\vert},
\end{equation}
and the map $\fix\colon \E\to \Hyp$ is analytic.
\end{lem}
\begin{proof}
The first assertion is a straightforward computation. Since $c\neq 0$ by Lemma ~\ref{lem:c-neq-0-for-elliptic}, the map $\fix\colon \E\to \Hyp$ is analytic. 
\end{proof}

The elliptic elements that fix the complex unit $i\in \Hyp$ are of the form
\[
\rot_{\vartheta} \coloneqq 
 \pm\begin{pmatrix}
  \cos(\vartheta/2) & \sin(\vartheta/2)  \\
  -\sin(\vartheta/2) & \cos(\vartheta/2)  
 \end{pmatrix}
\]
for $\vartheta\in (0,2\pi)$. Every $A\in \E$ is conjugate to a unique $\rot_{\vartheta(A)}$. This defines a function $\vartheta\colon \E\to (0,2\pi)$. The number $\vartheta(A)\in (0,2\pi)$ is called the (anti-clockwise) \emph{angle of rotation} of $A$. \begin{lem}\label{lem:angle-of-rotation-elliptic}
The angle of rotation of $A$ is
\begin{equation}\label{eq:angle-of-rotation-elliptic}
\vartheta(A)=\arctan\left(\frac{-c}{\vert c\vert}\cdot\frac{a+d}{(a+d)^2-2}\sqrt{4-(a+d)^2}\right)+\varepsilon(A),
\end{equation}
where
\begin{equation*}
\varepsilon(A)\coloneqq\left\{\begin{array}{ll}
0, &\text{if } (a+d)^2>2 \text{ and } (a+d)\frac{-c}{\vert c\vert}>0,\\
\pi, &\text{if } (a+d)^2<2, \\
2\pi, &\text{if } (a+d)^2>2 \text{ and } (a+d)\frac{-c}{\vert c\vert}<0.\\
\end{array}\right.
\end{equation*}
Moreover, the function $\vartheta\colon \E\to (0,2\pi)$ is analytic.
\end{lem}
\begin{proof}
The number $\vartheta(A)$ can be computed as the complex argument of the complex number
\begin{equation}\label{eq:complex-number-defining-angle-of-rotation}
\frac{dA}{dz}\bigg\rvert_{z=\fix A}=\left(\frac{(a+d)^2}{2}-1\right)-i\cdot (a+d)\frac{c}{\vert c\vert}\frac{\sqrt{4-(a+d)^2}}{2}.
\end{equation}
Observe that the imaginary part of \eqref{eq:complex-number-defining-angle-of-rotation} vanishes if and only if $a+d=0$, in which case its real part is equal to $-1$. This means that the complex number defined by \eqref{eq:complex-number-defining-angle-of-rotation} takes values inside $\C\setminus \R_{\geq 0}$. If we think of the complex argument of a number inside $\C\setminus \R_{\geq 0}$ as a function $\C\setminus \R_{\geq 0}\to (0,2\pi)$, then it is analytic. This shows that $\vartheta\colon \E\to (0,2\pi)$ is an analytic function. 
\end{proof}

\begin{prop}\label{prop:elliptic-matrix-from-fixed-point-and-angle}
The map
\[
(\fix,\vartheta)\colon \E\to \Hyp\times (0,2\pi)
\]
is an analytic diffeomorphism that identifies the subset of elliptic elements in $\psl$ with an open ball.
\end{prop}
\begin{proof}
We explained above that the map $(\fix,\vartheta)$ is analytic. The inverse map sends a point $z=x+i\cdot y\in \Hyp$ and an angle $\vartheta \in (0,2\pi)$ to the elliptic element
\begin{equation}\label{eq:elliptic-matrix-given-angle-and-fixed-point}
\rot_\vartheta (z)= \pm\begin{pmatrix}
  \cos(\vartheta/2)-xy^{-1}\sin (\vartheta/2) & (x^2y^{-1}+y)\sin(\vartheta/2)  \\
  -y^{-1}\sin(\vartheta/2) & \cos(\vartheta/2)+xy^{-1}\sin (\vartheta/2)
 \end{pmatrix}.
\end{equation}
Indeed, an immediate computation gives
\begin{align*}
\fix(\rot_\vartheta (z))&=\frac{-2xy^{-1}\sin(\vartheta/2)}{-2y^{-1}\sin(\vartheta/2)}+i\cdot\frac{2\sin(\vartheta/2)}{2y^{-1}\sin(\vartheta/2)}\\&=x+iy,
\end{align*}
and
\begin{align*}
\vartheta(\rot_\vartheta (z))&=\arg \left(\left(\frac{4\cos(\vartheta/2)^2}{2}-1\right)-i\cdot(2\cos(\vartheta/2))\cdot (-1)\cdot \frac{2\sin(\vartheta/2)}{2}\right)\\&=\arg(\cos(\vartheta)+i\sin(\vartheta))\\&=\vartheta.
\end{align*}
\end{proof}

The elements of $\psl$ whose trace in absolute value is equal to 2 are called \emph{parabolic}. Parabolic elements are those that have a unique fixed point of the boundary of $\Hyp$. There are two conjugacy classes of parabolic elements represented by
\begin{equation}\label{eq:parabolic-PSL(2,R)}
\parab^+\coloneqq \pm\begin{pmatrix}
1 & 1\\
0 & 1
\end{pmatrix} \quad\text{and}\quad \parab^-\coloneqq\pm\begin{pmatrix}
1 & 0\\
1 & 1
\end{pmatrix}.
\end{equation}
The elements conjugate to $\parab^+$ are called \emph{positively parabolic} and those conjugate to $\parab^-$ \emph{negatively parabolic}. Each conjugacy class of parabolic elements is an open annulus whose closures intersect at the identity.

The elements of $\psl$ with a trace larger than 2 in absolute value are called \emph{hyperbolic}. Hyperbolic elements have precisely two fixed points on the boundary of $\Hyp$. Any hyperbolic element of $\psl$ is conjugate to
\[
\hyp_\lambda\coloneqq \pm\begin{pmatrix}
\lambda & 0\\
0 & \lambda^{-1}
\end{pmatrix},
\]
for a unique $\lambda > 0$. Hyperbolic conjugacy classes are open annuli.

Elliptic, parabolic and hyperbolic conjugacy classes foliate $\psl$ in a way that is illustrated on Figure ~\ref{fig:PSL(2,R)}.

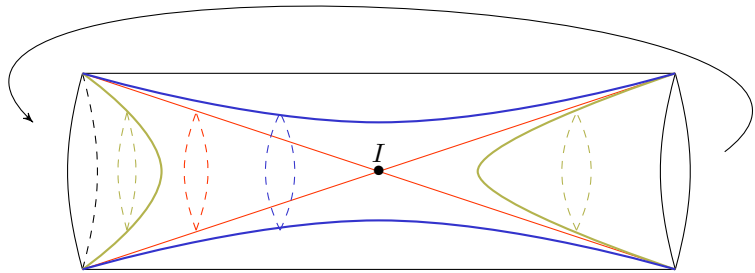
\begin{figure}[h]
\begin{center}
\begin{tikzpicture}[scale=1.3,>=stealth',font=\sffamily,decoration={
    markings,
    mark=at position 0.999 with {\arrow{>}}}]
 \draw (0,0) to (6,0);
  \draw (0,2) to (6,2);
 \draw(0,0) to[bend left=15]  (0,2);
 \draw[dashed] (0,0) to[bend right=15] (0,2);
 \draw(6,0) to[bend left=15]  (6,2);
 \draw (6,0) to[bend right=15] (6,2);
 
 \draw[red!80!yellow] (0,0) to (6,2);
 \draw[red!80!yellow] (0,2) to (6,0);
 
 \draw (3,1) node[above]{$I$};
 \draw (3,1) node{$\bullet$};
 
 \draw[thick, blue!30!yellow] plot [smooth, tension=.6] coordinates {(6,0) (4,1) (6,2)};
 \draw[dashed,blue!30!yellow] (5,0.4) to[bend left=25]  (5,1.6);
 \draw[dashed,blue!30!yellow] (5,0.4) to[bend right=25]  (5,1.6);
 
 \draw[thick, blue!30!yellow] plot [smooth, tension=.8] coordinates {(0,0) (0.8,1) (0,2)};
 \draw[dashed,blue!30!yellow] (.45,0.4) to[bend left=15]  (.45,1.6);
 \draw[dashed,blue!30!yellow] (.45,0.4) to[bend right=15]  (.45,1.6);
 
 \draw[dashed, red!80!yellow] (1.15,0.4) to[bend left=20]  (1.15,1.6);
 \draw[dashed, red!80!yellow] (1.15,0.4) to[bend right=20]  (1.15,1.6);
 
 \draw[thick, blue!80!yellow] plot [smooth, tension=.8] coordinates {(0,2) (3,1.5) (6,2)};
 \draw[thick, blue!80!yellow] plot [smooth, tension=.8] coordinates {(0,0) (3,.5) (6,0)};
 \draw[dashed, blue!80!yellow] (2,0.4) to[bend left=25]  (2,1.6);
 \draw[dashed, blue!80!yellow] (2,0.4) to[bend right=25]  (2,1.6);

 \draw[postaction={decorate}] (6.5,1.2) to[bend left=220] (-.5,1.5);	
 
\end{tikzpicture}
\caption{The elliptic conjugacy classes are drawn in green. They foliate an open ball into disks. The open ball is bounded by the two parabolic conjugacy classes which have the shape of two red cones joined at the identity. The hyperbolic conjugacy classes foliate an open solid torus, bounded by the red cones, into blue annuli.}\label{fig:PSL(2,R)}
\end{center}
\end{figure}

\subsection{Pants decompositions of punctured spheres}\label{appendix:pants-decomposition}

The purpose of this appendix is to bring some clarity on a particular step of the construction introduced in this paper. It was said in introduction that the coordinates we produce depend on the choice of a pants decomposition of the oriented punctured sphere $\Sigma_n$. In multiple occasions, we studied the image of a curve of the pants decomposition under a representation $\phi\colon \pi_1(\Sigma_n)\to \psl$. This is an abuse of language, since isotopy classes of simple closed curves have multiple lifts inside $\pi_1(\Sigma_n)$ (that differ by conjugation). This distinction is irrelevant as long as we are only interested in the angle of rotation of the image, assuming the image is elliptic, because it is conjugacy invariant. It is however essential if one is interested, for instance, in the fixed point of the image. Lemma \ref{lem:main-lemma-pair-of-pants} below shows how one can coherently chose lifts of the pants curves inside $\pi_1(\Sigma_n)$.

It is convenient to work with surfaces with boundaries instead of punctures. Replacing each puncture of $\Sigma_n$ by a boundary component we obtain an oriented sphere $\widehat \Sigma_n$ with boundary. The boundary components of $\widehat \Sigma_n$ are labelled $\gamma_1,\ldots,\gamma_n$, accordingly with the labels of the punctures of $\Sigma_n$, and are given the induced orientation. Note that $\Sigma_n$ identifies homeomorphically with $\widehat \Sigma_n\setminus \partial \widehat\Sigma_n$. Let $\Mod (\Sigma_n)$ denote the \emph{pure mapping class group} of $\Sigma_n$, that is the group of isotopy classes of orientation-preserving homeomorphisms of $\Sigma_n$ that fix each puncture individually. Dropping the restriction to orientation-preserving homeomorphisms, we obtain a larger mapping class group denoted $\Mod^\pm(\Sigma_n)$.

A \emph{pants decomposition} of $\widehat\Sigma_n$ consists of a maximal collection of isotopy classes of pairwise disjoint simple closed curves on $\widehat\Sigma_n$. The $n$ isotopy classes of the boundary curves $\gamma_i$ are elements of any pants decomposition. The cardinality of a pants decomposition is always $2n-3$. The $n-3$ isotopy classes of curves that are not boundary curves are said to be \emph{interior}. We denote by $\mathbb P$ the collection of all pants decompositions of $\widehat\Sigma_n$. The action of $\Mod (\Sigma_n)$ on the set of free isotopy classes of simple closed curves induces an action on $\mathbb P$.

To every pants decomposition $\mathcal P\in \mathbb P$ we can associate a graph $\Gamma(\mathcal P)$ on $2n-2$ vertices defined as follows. There are two types of vertices: $n$ empty vertices, one for each boundary curve of $\widehat\Sigma_n$, and $n-2$ full vertices, one for each pair of pants determined by $\mathcal P$. Two full vertices are connected by an edge if they share a common boundary. An empty vertex is connected to a full vertex by an edge if the boundary curve corresponding to the empty vertex bounds the pair of pants corresponding to the full vertex. An example of such a graph is given by Figure ~\ref{fig:pants-graph}. Let $\Gamma \coloneqq \{\Gamma(\mathcal P):\mathcal P\in\mathbb P\}$ be the set of all such graphs. Clearly, $\Gamma$ is a finite set. 

\begin{figure}
\begin{tikzpicture}[tqft/flow=north,tqft/cobordism/.style={draw}, tqft/every boundary component/.style={draw}]
\node[tqft/pair of pants,draw] (a) {};
\node[tqft/reverse pair of pants,draw,anchor=incoming boundary 1] (b) at (a.outgoing boundary 2) {};
\node[tqft/pair of pants,draw,anchor=incoming boundary 1] (c) at (a.outgoing boundary 1) {};
\node[tqft/reverse pair of pants,draw,anchor=outgoing boundary 1] (d) at (b.incoming boundary 2) {};
\node[tqft/pair of pants,draw,anchor=incoming boundary 1] (e) at (b.outgoing boundary 1) {};
\end{tikzpicture}
\begin{tikzpicture}
\draw (3,0)--(3,4/3);
\draw (5,0)--(6,4/3);
\draw (7,0)--(6,4/3);
\draw (5,8/3)--(6,4/3);
\draw (3,4/3)--(5,8/3);
\draw (3,4/3)--(2,8/3);
\draw (1,4)--(2,8/3);
\draw (3,4)--(2,8/3);
\draw (5,14/3)--(5,8/3);
\draw (5,14/3)--(4,6);
\draw (5,14/3)--(6,6);

\draw (3,0) node{$\circ$};
\draw (5,0) node{$\circ$};
\draw (7,0) node{$\circ$};
\draw (1,4) node{$\circ$};
\draw (3,4) node{$\circ$};
\draw (4,6) node{$\circ$};
\draw (6,6) node{$\circ$};

\draw (3,4/3) node{$\bullet$};
\draw (2,8/3) node{$\bullet$};
\draw (6,4/3) node{$\bullet$};
\draw (5,8/3) node{$\bullet$};
\draw (5,14/3) node{$\bullet$};
\end{tikzpicture}
\caption{On the left: a pants decomposition of a 7-punctured sphere. On the right: the corresponding graph.}\label{fig:pants-graph}
\end{figure}
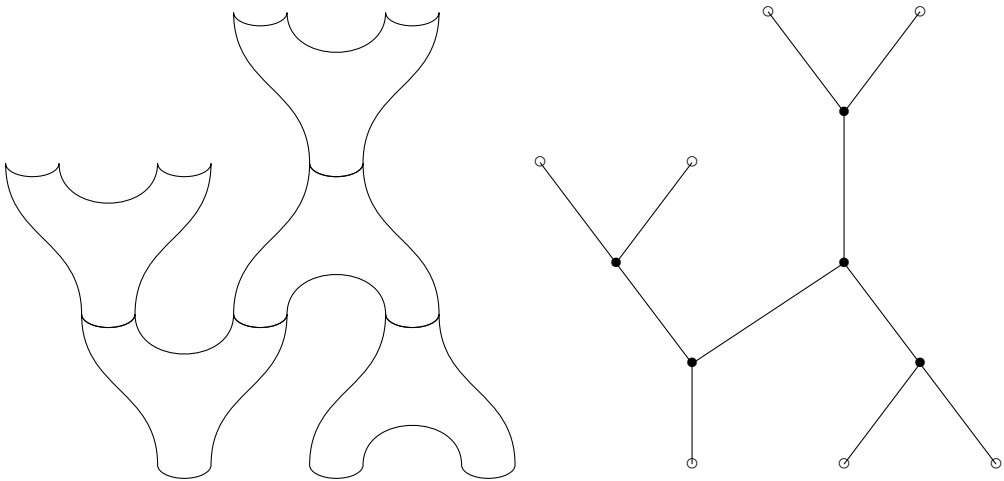

\begin{lem}[see e.g.\ \cite{Ute}]
There is a bijection between $\Gamma$ and the orbit space $\mathbb P/\Mod(\Sigma_n)$. In particular, there are only finitely many pants decomposition of $\widehat \Sigma_n$ up to mapping class group action. The stabilizer of a pants decomposition $\mathcal P\in\mathbb P$ is the subgroup of $\Mod(\Sigma_n)$ generated by Dehn twists along the interior curves of $\mathcal P$, hence isomorphic to $\Z^{n-3}$.
\end{lem}

Let $\pi_1(\Sigma_n)$ denote the fundamental group of $\Sigma_n$. The inclusion $\Sigma_n\cong \widehat \Sigma_n\setminus\partial \widehat \Sigma_n\subset \widehat \Sigma_n$ induces a group isomorphism $\pi_1(\Sigma_n)\cong \pi_1(\widehat\Sigma_n)$. Let $\Aut(\pi_1(\Sigma_n))$ denote the group of automorphisms of $\pi_1(\Sigma_n)$. It contains the subgroup $\Aut^\ast(\pi_1(\Sigma_n))$ of automorphisms that preserve the set of conjugacy classes of simple closed curves surrounding punctures, see \cite[§8]{FaMa12} for more details. Let $\Out^\ast(\pi_1(\Sigma_n))\coloneqq \Aut^\ast(\pi_1(\Sigma_n))/\Inn(\pi_1(\Sigma_n))$ be the corresponding subgroup of the group of outer automorphisms of $\pi_1(\Sigma_n)$.

\begin{thm}[Dehn-Nielsen-Baer]
The mapping class group $\Mod^\pm(\Sigma_n)$ is isomorphic to the group $\Out^\ast(\pi_1(\Sigma_n))$.
\end{thm}

The set of conjugacy classes in $\pi_1(\Sigma_n)$ is denoted by $\pi_1(\Sigma_n)/\pi_1(\Sigma_n)$. Forgetting the basepoint gives a bijection between $\pi_1(\Sigma_n)/\pi_1(\Sigma_n)$ and the set of free isotopy classes of closed curves in $\widehat \Sigma_n$, which we denote by $[S^1,\widehat \Sigma_n]$. Let $\pi\colon \pi_1(\Sigma_n)\to [S^1,\widehat\Sigma_n]$ denote the induced projection. If $\alpha\subset \widehat\Sigma_n$ is a closed curve, then we write $[\alpha]\in [S^1,\widehat\Sigma_n]$ for its free isotopy class.

The main claim of this appendix is the following.

\begin{lem}\label{lem:main-lemma-pair-of-pants}
Let $\mathcal P=\{[\gamma_1],\ldots,[\gamma_n],[\beta_1],\ldots,[\beta_{n-3}]\}\in \mathbb P$ be a pants decomposition of $\widehat \Sigma_n$. There exist representatives $c_i\in \pi^{-1}([\gamma_i])$ and $b_j\in \pi^{-1}([\beta_j])$ such that
\begin{itemize}
\item there exists a permutation $\sigma\in S_n$ for which $\prod_{i=1}^n c_{\sigma(i)}=1$,
\item for every pair of pants $P$ determined by $\mathcal P$, if $[\alpha_1],[\alpha_2], [\alpha_3]\in \mathcal P$ denote the boundary curves of $P$, and $a_1,a_2,a_3\in \{c_1,\ldots,c_n,b_1,\ldots,b_{n-3}\}$ the corresponding representatives, then the cyclic product of $a_1,a_2$, and $a_3$ is trivial (up to maybe considering their inverses for compatible orientation with the orientation of $P$).
\end{itemize}
\end{lem}

An immediate corollary of Lemma ~\ref{lem:main-lemma-pair-of-pants} is that the inclusion
\[
\langle c_1,\ldots,c_n:\prod_{i=1}^n c_{\sigma(i)}=1\rangle \subset \pi_1(\Sigma_n)
\]
is a group isomorphism. Obviously, the representatives $c_i,b_i$ of Lemma ~\ref{lem:main-lemma-pair-of-pants} are not unique; Two different choices differ by an element of $\Aut^\ast(\pi_1(\Sigma_n))$.

\begin{proof}[Proof of Lemma ~\ref{lem:main-lemma-pair-of-pants}.]
We proceed by induction on $n$. In the case $n=3$, since $\Mod(\Sigma_3)$ is trivial, it holds that $\Aut^\ast(\pi_1(\Sigma_3))=\Inn(\pi_1(\Sigma_3))$. So any choice of presentation $\pi_1(\Sigma_3)=\langle c_1,c_2,c_3:c_1c_2c_3=1\rangle$ with $\pi(c_i)=[\gamma_i]$ will do.

Now, assume that the result holds for every $k\leq n$ with $n\geq 4$. Let $\Gamma(\mathcal P)$ be the graph of $\mathcal P$. There exists a full vertex connected to two empty vertices. This corresponds to a pair of pants with boundary curves, say, $[\beta_k]$, $[\gamma_a]$, and $[\gamma_b]$. Cutting $\Sigma_n$ along $\beta_k$ gives a decomposition $\Sigma_n=\Sigma_{n-1}\sqcup_{\beta_k} \Sigma_3$. The graph of $\mathcal P$ decomposes accordingly, see Figure ~\ref{fig:pants-graph-decomposition}.

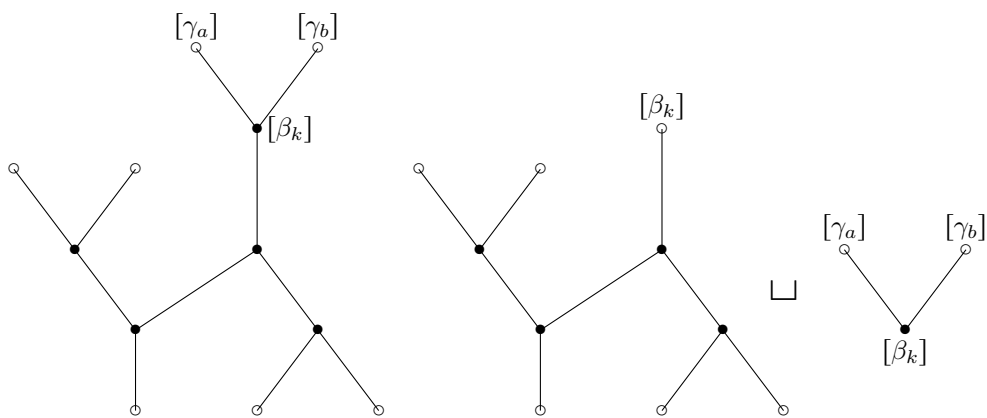
\begin{figure}
\begin{tikzpicture}[scale=.8]
\draw (3,0)--(3,4/3);
\draw (5,0)--(6,4/3);
\draw (7,0)--(6,4/3);
\draw (5,8/3)--(6,4/3);
\draw (3,4/3)--(5,8/3);
\draw (3,4/3)--(2,8/3);
\draw (1,4)--(2,8/3);
\draw (3,4)--(2,8/3);
\draw (5,14/3)--(5,8/3);
\draw (5,14/3)--(4,6);
\draw (5,14/3)--(6,6);

\draw (3,0) node{$\circ$};
\draw (5,0) node{$\circ$};
\draw (7,0) node{$\circ$};
\draw (1,4) node{$\circ$};
\draw (3,4) node{$\circ$};
\draw (4,6) node{$\circ$} node[above]{$[\gamma_a]$};
\draw (6,6) node{$\circ$} node[above]{$[\gamma_b]$};

\draw (3,4/3) node{$\bullet$};
\draw (2,8/3) node{$\bullet$};
\draw (6,4/3) node{$\bullet$};
\draw (5,8/3) node{$\bullet$};
\draw (5,14/3) node{$\bullet$} node[right]{$[\beta_k]$};
\end{tikzpicture}
\begin{tikzpicture}[scale=.8]
\draw (3,0)--(3,4/3);
\draw (5,0)--(6,4/3);
\draw (7,0)--(6,4/3);
\draw (5,8/3)--(6,4/3);
\draw (3,4/3)--(5,8/3);
\draw (3,4/3)--(2,8/3);
\draw (1,4)--(2,8/3);
\draw (3,4)--(2,8/3);
\draw (5,14/3)--(5,8/3);

\draw (3,0) node{$\circ$};
\draw (5,0) node{$\circ$};
\draw (7,0) node{$\circ$};
\draw (1,4) node{$\circ$};
\draw (3,4) node{$\circ$};

\draw (3,4/3) node{$\bullet$};
\draw (2,8/3) node{$\bullet$};
\draw (6,4/3) node{$\bullet$};
\draw (5,8/3) node{$\bullet$};
\draw (5,14/3) node{$\circ$} node[above]{$[\beta_k]$};

\draw (7,2) node{\huge $\sqcup$};

\draw (9,4/3) node{$\bullet$} node[below]{$[\beta_k]$};
\draw (8,8/3) node{$\circ$} node[above]{$[\gamma_a]$};
\draw (10,8/3) node{$\circ$} node[above]{$[\gamma_b]$};

\draw (9,4/3)--(8,8/3);
\draw (9,4/3)--(10,8/3);

\end{tikzpicture}
\caption{On the left: the initial graph. On the right: its decomposition into two graphs corresponding to the sub-surfaces.}\label{fig:pants-graph-decomposition}
\end{figure}

By Van Kampen's Theorem, $\pi_1(\Sigma_n)$ is the free product of $\pi_1(\Sigma_{n-1})$ with $\pi_1(\Sigma_3)$ amalgated over the fundamental group of $\beta_k= \Sigma_{n-1}\cap \Sigma_3$. Let $\pi_1(\Sigma_{n-1})\subset \pi_1(\Sigma_n)$ and $\pi_1(\Sigma_3)\subset \pi_1(\Sigma_n)$ be the two induced inclusions. We apply the induction hypothesis to the two sub-surfaces $\Sigma_{n-1}$ and $\Sigma_3$, and their pants decomposition induced from $\mathcal P$. This gives $c_1,\ldots,c_n\in \pi_1(\Sigma_n)$, $b_1,\ldots,b_{n-3}\in \pi_1(\Sigma_n)$ and $b_k'\in \pi_1(\Sigma_n)$ such that
\begin{itemize}
\item $c_i\in \pi^{-1}([\gamma_i])$, for $i=1,\ldots,n$, $b_j\in \pi^{-1}([\beta_j])$, for $j=1,\ldots,n-3$, and $b_k'\in \pi^{-1}([\beta_k])$,
\item there exist a permutation $\sigma \in S_{n-2}$ of the numbers $\{1,\ldots,n\}\setminus \{a,b\}$ and a permutation $\tau \in S_2$ of the numbers $\{a,b\}$ with 
\[
\prod_{i\neq a,b}c_{\sigma(i)}\cdot b_k=1 \quad \text{and} \quad c_{\tau(a)}c_{\tau(b)}(b_k')^{-1}=1.
\]
\end{itemize}

Since $\pi(b_k)=\pi(b_k')$, there exists $x\in\pi_1(\Sigma_n)$ such that $b_k=xb_k'x^{-1}$.  Hence
\[
\prod_{i\neq a,b}c_{\sigma(i)}\cdot xc_{\tau(a)}x^{-1}\cdot xc_{\tau(b)}x^{-1}=1 .
\]
Note that $xc_{\tau(a)}x^{-1}\in \pi^{-1}([\gamma_a])$ and $xc_{\tau(b)}x^{-1}\in \pi^{-1}([\gamma_b])$. Let $\overline{\sigma}\in S_n$ be the permutation of the numbers $\{1,\ldots,n\}$ given by $\overline{\sigma}(n-1)=\tau(a)$, $\overline{\sigma}(n)=\tau(b)$, and, for $i=1,\ldots, n-2$, $\overline{\sigma}(i)=\sigma(e(i))$, where $e\colon \{1,\ldots,n-2\}\to \{1,\ldots,n\}\setminus \{a,b\}$ is the increasing bijection. Let further $\overline c_i\coloneqq c_i$ for $i\neq a,b$, $\overline c_a\coloneqq xc_{a}x^{-1}$, and $\overline c_b\coloneqq xc_{b}x^{-1}$. Similarly, let $\overline b_j\coloneqq b_j$ for $j\neq k$ and $\overline b_k\coloneqq b_k'=x^{-1}b_kx$. It is easy to verify that the elements $\overline c_i, \overline b_j\in \pi_1(\Sigma_n)$ and the permutation $\overline \sigma\in S_n$ satisfy the conclusion of the lemma.
\end{proof}

\bibliographystyle{amsalpha}
\bibliography{literature-bis}

\end{document}